\documentclass{amsart}
\usepackage[a4paper,tmargin=40mm,bmargin=35mm,hmargin=30mm]{geometry}
\usepackage{amscd,amsmath,mathrsfs,amssymb}
\usepackage{mathtools}
\usepackage[utf8]{inputenc}
\usepackage[T1]{fontenc}
\usepackage{tikz-cd}
\usepackage[colorlinks=true,linkcolor=red,citecolor=blue,urlcolor=green]{hyperref}
\usepackage{eucal}
\usepackage[capitalise]{cleveref}
\usepackage[shortlabels]{enumitem}
\usepackage{mathtools} 
\usepackage{extarrows}

\definecolor{mygreen}{rgb}{0.1,0.6,0.3}
\definecolor{mybrown}{rgb}{0.6,0.3.5,0}

\newtheorem{theorem}{Theorem}[section]
\newtheorem{proposition}[theorem]{Proposition}
\newtheorem{lemma}[theorem]{Lemma}
\newtheorem{corollary}[theorem]{Corollary}
\crefname{theorem}{Theorem}{Theorems}
\crefname{proposition}{Proposition}{Propositions}
\crefname{lemma}{Lemma}{Lemmas}
\crefname{corollary}{Corollary}{Corollaries}
\theoremstyle{definition}
\newtheorem{definition}[theorem]{Definition}
\newtheorem{example}[theorem]{Example}
\newtheorem{question}[theorem]{Question}
\newtheorem{remark}[theorem]{Remark}

\newtheorem{notation}[theorem]{Notation}

\crefname{definition}{Definition}{Definitions}
\crefname{example}{Example}{Examples}
\crefname{problem}{Problem}{Problems}
\crefname{fact}{Fact}{Facts}
\crefname{question}{Question}{Questions}
\crefname{remark}{Remark}{Remarks}
\crefname{convention}{Convention}{Conventions}
\crefname{notation}{Notation}{Notations}

\DeclareMathOperator{\Spec}{Spec}
\DeclareMathOperator{\Kdim}{Kdim}
\DeclareMathOperator{\Supp}{Supp}
\DeclareMathOperator{\supp}{supp}
\DeclareMathOperator{\cosupp}{cosupp}
\DeclareMathOperator{\mSpec}{mSpec}

\DeclareMathOperator{\height}{\mathrm{ht}}
\DeclareMathOperator{\grade}{\mathrm{gr}}

\DeclareMathOperator{\depth}{depth}
\DeclareMathOperator{\width}{width}

\DeclareMathOperator{\Mod}{Mod}

\DeclareMathOperator{\Proj}{Proj}
\DeclareMathOperator{\Flat}{Flat}
\DeclareMathOperator{\Inj}{Inj}
\DeclareMathOperator{\Add}{Add}
\DeclareMathOperator{\Prod}{Prod}
\DeclareMathOperator{\D}{\mathbf{D}}
\DeclareMathOperator{\K}{\mathbf{K}}
\DeclareMathOperator{\C}{\mathbf{C}}

\DeclareMathOperator{\Gen}{Gen}

\DeclareMathOperator{\thick}{thick}
\DeclareMathOperator{\Ctra}{Ctra}

\DeclareMathOperator{\Id}{Id}

\DeclareMathOperator{\Ker}{Ker}
\DeclareMathOperator{\Ima}{Im}
\DeclareMathOperator{\Cone}{Cone}
\DeclareMathOperator{\Coker}{Coker}
\DeclareMathOperator{\Hom}{Hom}
\DeclareMathOperator{\RHom}{\mathbf{R}Hom}
\DeclareMathOperator{\RGamma}{\mathbf{R}\Gamma}
\DeclareMathOperator{\LLambda}{\mathbf{L}\Lambda}
\newcommand{\LotimesR}{\mathop{\otimes_R^{\mathbf L}}\nolimits}
\newcommand{\LotimesRp}{\mathop{\otimes_{R_\pp}^{\mathbf L}}\nolimits}

\DeclareMathOperator{\Ext}{Ext}
\DeclareMathOperator{\End}{End}
\DeclareMathOperator{\Tor}{Tor}
\newcommand{\Ad}{\mathrm{A}}

\newcommand{\ZZ}{\ensuremath{\mathbb{Z}}}
\newcommand{\QQ}{\ensuremath{\mathbb{Q}}}

\newcommand{\pp}{\mathfrak{p}}
\newcommand{\qq}{\mathfrak{q}}

\newcommand{\mm}{\mathfrak{m}}

\newcommand{\cA}{\ensuremath{\mathcal{A}}}

\newcommand{\cC}{\ensuremath{\mathcal{C}}}

\newcommand{\cE}{\ensuremath{\mathcal{E}}}

\newcommand{\cH}{\ensuremath{\mathcal{H}}}

\newcommand{\cL}{\ensuremath{\mathcal{L}}}

\newcommand{\cS}{\ensuremath{\mathcal{S}}}
\newcommand{\cT}{\ensuremath{\mathcal{T}}}
\newcommand{\cU}{\ensuremath{\mathcal{U}}}
\newcommand{\cV}{\ensuremath{\mathcal{V}}}
\newcommand{\cW}{\ensuremath{\mathcal{W}}}
\newcommand{\cY}{\ensuremath{\mathcal{Y}}}
\newcommand{\cX}{\ensuremath{\mathcal{X}}}

\newcommand*{\Perp}[1]{{}^{\perp_{#1}}}
\newcommand{\bd}{\mathrm{b}}
\newcommand{\cp}{\mathrm{c}}
\newcommand{\fg}{\mathrm{fg}}
\newcommand{\op}{\mathrm{op}}
\newcommand{\cd}{\mathsf{d}}
\newcommand{\sm}{\setminus}
\newcommand{\f}{\mathsf{f}}
\newcommand{\isoto}{\xrightarrow{\smash{\raisebox{-0.25em}{$\sim$}}}}
\newcommand{\ha}{\hookrightarrow}

\crefname{enumi}{}{}
\Crefname{enumi}{}{}
\creflabelformat{enumi}{\normalfont #2#1#3}
\crefrangeformat{enumi}{\normalfont #3#1#4--#5#2#6}
\crefname{enumii}{}{}
\Crefname{enumii}{}{}
\creflabelformat{enumii}{\normalfont #2#1#3}
\crefrangeformat{enumii}{\normalfont #3#1#4--#5#2#6}
\crefname{enumiii}{}{}
\Crefname{enumiii}{}{}
\creflabelformat{enumiii}{\normalfont #2#1#3}
\crefrangeformat{enumiii}{\normalfont #3#1#4--#5#2#6}

\numberwithin{equation}{section}
\crefname{equation}{}{}
\Crefname{equation}{}{}
\creflabelformat{equation}{(#2#1#3)}
\makeatletter
\@namedef{subjclassname@2020}{\textup{2020} Mathematics Subject Classification}
\makeatother

\title[Tilting complexes and codimension functions]{Tilting complexes and codimension functions over commutative noetherian rings}

\author{Michal Hrbek}
\address[M. Hrbek]{Institute of Mathematics of the Czech Academy of Sciences, \v{Z}itn\'{a} 25, 115 67 Prague, Czech Republic}
\email{hrbek@math.cas.cz}

\author{Tsutomu Nakamura}
\address[T. Nakamura]{Department of Mathematics, Faculty of Education, Mie University, 1577 Kurimamachiya-cho, Tsu city, Mie, 514-8507, Japan}
\email{nakamura@edu.mie-u.ac.jp}
\address{Osaka Central Advanced Mathematical Institute, Osaka Metropolitan University, 3-3-138 Sugimoto, Sumiyoshi-ku, Osaka, 558-8585, Japan}
\email{t.nakamura.math@gmail.com}

\author{Jan \v{S}\v{t}ov\'{i}\v{c}ek}
\address[J. \v{S}\v{t}ov\'{i}\v{c}ek]{Charles University, Faculty of Mathematics and Physics, Department of Algebra, Sokolovsk\'{a} 83, 186 75 Praha, Czech Republic }
\email{stovicek@karlin.mff.cuni.cz}

\thanks{M. Hrbek was supported by the GAČR project 20-13778S and RVO: 67985840. T. Nakamura was supported by PRIN-2017  ``Categories, Algebras: Ring-Theoretical and Homological Approaches (CARTHA)'' and Grant-in-Aid for JSPS Fellows JP20J01865. J. Šťovíček was supported by the GAČR project 20-13778S}

\begin{document}
\dedicatory{Dedicated to Lidia Angeleri H\"ugel on the occasion of her 60th birthday}

\begin{abstract}
In the derived category of a commutative noetherian ring, we explicitly construct a silting object associated with each sp-filtration of the Zariski spectrum satisfying the ``slice'' condition. Our new construction is based on local cohomology and it allows us to study when the silting object is tilting. For a ring admitting a dualizing complex, this occurs precisely when the sp-filtration arises from a codimension function on the spectrum. In the absence of a dualizing complex, the situation is more delicate and the tilting property is closely related to the condition that the ring is a homomorphic image of a Cohen--Macaulay ring. We also provide dual versions of our results in the cosilting case.
\end{abstract}

\maketitle

\tableofcontents
\section{Introduction}
Rickard's theorem \cite{Ri89} states that derived equivalences between categories of modules over rings are fully governed by compact tilting complexes. While these are ubiquitous in representation theory, it is well known that any derived equivalence boils down to a Morita equivalence whenever one of the rings happens to be commutative (e.g., \cite[Proposition 5.14]{AHKLY17}). As a consequence, in commutative algebra, the classical tilting theory reduces just to studying progenerators of the module category. However, there is a much richer supply of large (= possibly non-compact) tilting objects. They were first introduced in the module-theoretic incarnation by Colpi and Trlifaj \cite{CT95} and Angeleri H\"{u}gel and Coelho \cite{AHC01}.
Psaroudakis and Vitória \cite{PV18} and Nicolás, Saorín, and Zvonareva \cite{NSZ19} introduced large silting and tilting objects in triangulated categories, where they naturally generalized the concepts of compact silting objects (\cite{KV88,AI12}) and large silting complexes (\cite{Wei13,AHMV16a}).
In contrast with the classical setting, the endomorphism ring of a large tilting module is typically not derived equivalent to the original ring. However, if the tilting module is ``good'', then one obtains a triangulated equivalence after localizing the derived category of the endomorphism ring (\cite{BMT11}). This discrepancy can be measured by a recollement; see \cite[\S 5]{AHKL11}, \cite{CX12,CX19}, and \cite{BP13}.
Silting objects are defined as those inducing t-structures in a natural way, and the heart of such a t-structure is an abelian category having a projective generator. Under mild assumptions, the inclusion from the heart of a silting t-structure to the derived category extends to a derived equivalence if and only if the silting object is tilting (\cite{PV18}). An analogous dual theory exists for cosilting and cotilting objects, for which the hearts are abelian categories having an injective cogenerator (in fact, the cosilting hearts are often Grothendieck categories). See \cref{preliminaries} for more details. We also remark that recently a mutation theory for large silting and cosilting objects has been developed in \cite{AHLSV22}; pure-injective cosilting objects play an important role there, and mutation for such objects encompasses mutation for compact silting complexes in the classical context (\cite{KV88,AI12}). The theory is meaningful in the setting of the present paper as well; see \cite[Example 4.11]{AHLSV22} and the last paragraph of \cref{subsec-real-func}.

Let $R$ be a commutative noetherian ring. A lot of structural information about the derived category $\D(R)$ is known to be controlled by the Zariski spectrum $\Spec R$.  
Neeman \cite{Nee92} proved that there is a canonical bijection between the localizing subcategories of $\D(R)$ and the subsets of $\Spec R$.
Alonso Tarr\'{\i}o, Jerem\'{\i}as L\'{o}pez and Saor\'{i}n
\cite{ATJLS10} classified the compactly generated t-structures in $\D(R)$ by using the sp-filtrations of $\Spec R$ (= the filtrations of specialization closed subsets of $\Spec R$); see \cref{ATJLS-bijec}.
The compactly generated t-structures induced by silting objects have also been classified in terms of sp-filtrations of $\Spec R$. For tilting modules, this was essentially achieved by Angeleri H\"{u}gel, Pospíšil, the third author, and Trlifaj \cite{AHPST14}, and in general, by Angeleri H\"{u}gel and the first author \cite{AHH21}; see \cref{silt-fin-sp}.
On the other hand, explicit constructions of the silting objects inducing these t-structures are known only in rather limited instances; see \cite[\S 3]{AHHT06} and \cite{PT11}.

The first main goal of this paper is to provide such a construction for sp-filtrations of special type, which we call \emph{slice} sp-filtrations. The slice condition asserts that in each step of the filtration we are only allowed to remove at most a zero-dimensional (=``thin'') layer of prime ideals. Such sp-filtrations are rather ubiquitous; see \cref{hight-codim} and the next paragraph. The philosophy behind these filtrations is akin to the notion of the slice filtration in equivariant or motivic homotopy theory; see \cref{slice-remark,slice-functor}. In particular, the slice condition is sufficient for the Bousfield localization and colocalization functors induced by the consecutive members of the sp-filtration to decompose objects into ``computable pieces''. Our construction of silting objects is also related to the method that builds tilting modules and complexes from ring epimorphisms in the representation theoretical context; see, e.g.,  \cite[Theorem 3.5]{AHS11} and \cite[\S 5]{AHH21}. To successfully produce a silting object, the existing constructions need to impose a strong homological restriction on the ring epimorphism(s) appearing there, while the slice condition can be viewed as a suitable geometric analog. See the end of \cref{proof-theorem} for a more detailed discussion. 

To state our first main result, we remark that there is a bijection between the sp-filtrations $\Phi$ of $\Spec R$ and the order-preserving functions $\f_{\Phi}:\Spec R \to \ZZ \cup\{\infty, -\infty\}$. Then $\Phi$ is a slice sp-filtration if and only if it corresponds to a strictly increasing function $\f_{\Phi}:\Spec R \to \ZZ$; see \cref{order-preserv} and \cref{slice-strict}.

\begin{theorem}[\cref{slice-silting}]\label{intro-slice}
	Let $R$ be a commutative noetherian ring and $\Phi$ a slice sp-filtration of $\Spec R$. Then
	\[T_\Phi := \bigoplus_{\pp \in \Spec R} \Sigma^{\f_\Phi(\pp)}\RGamma_{\pp}R_{\pp}\]
	is a silting object in $\D(R)$, and this induces the t-structure whose aisle is \[\cY_{\Phi}:=\{X \in \D(R) \mid \width_{R}(\pp,X)>n~\forall n \in \ZZ~\forall \pp\in \Phi(n)\}.\]
\end{theorem}

In general, the tilting property is not easily read from the silting t-structure alone, as the vanishing of negative self-extensions is typically not tested by orthogonality properties with respect to any set of compact objects.
Our explicit construction of silting objects enables us to study this question:
\begin{center}\emph{When is the silting object $T_\Phi$ tilting?}\end{center} 
In fact, if $T_\Phi$ is tilting, then $\Phi$ is necessarily a codimension filtration (i.e., $\f_{\Phi}$ is a codimension function on $\Spec R$); see \cref{necessity}. This makes a limitation on $R$ because the existence of a codimension function implies $R$ is catenary. 
Our second main result shows that, for rings which are nice enough in a certain sense, $T_\Phi$ is tilting for any codimension filtration $\Phi$.

\begin{theorem}[\cref{finite-theorem,tilt-dc}]\label{intro-tilt}
Assume one of the following conditions hold:
\begin{enumerate}[label=(\arabic*), font=\normalfont]
\item \label{intro-tilt-dc} $R$ admits a (strongly pointwise) dualizing complex; 
\item \label{intro-tilt-CM} $R$ is a homomorphic image of a Cohen--Macaulay ring of finite Krull dimension.
\end{enumerate}
Let $\Phi$ be a codimension filtration of $\Spec R$, which exists under the assumption. Then $T_{\Phi}$ is a tilting object in $\D(R)$.
\end{theorem}

Note that $R$ can have infinite Krull dimension under the condition \cref{intro-tilt-dc} of \cref{intro-tilt}.
When a (strongly pointwise) dualizing complex $D$ for $R$ exists, $R$ has finite Krull dimension if and only if $D$ has finite injective dimension. In this case, we refer to $D$ as a classical dualizing complex (\cref{strong-pw}).
We also remark that, for rings of finite Krull dimension, \cref{intro-tilt-dc} is strictly stronger than \cref{intro-tilt-CM}. Indeed, the existence of a classical dualizing complex implies that $R$ is a homomorphic image of a Gorenstein ring of finite Krull dimension. This fact (originally conjectured by Sharp \cite{Sha79}) is due to Kawasaki \cite{Kaw02} (\cref{Kawasaki}\cref{Kawasaki-dc}). 

Given a classical dualizing complex $D$ for $R$, we know that it is a cotilting object in the bounded derived category $\D^{\bd}_{\fg}(R)$ (\cref{dc-cotilt-rem}). Under this interpretation, Kawasaki's result characterizes homomorphic image of Gorenstein rings of finite Krull dimension in a tilting theoretic way.
His other work \cite{Kaw08} provides a necessary and sufficient condition for a commutative noetherian ring to be a homomorphic image of a Cohen--Macaulay ring (\cref{Kawasaki}\cref{Kawasaki-nonloc}), while this characterization itself looks far from tilting theory. Nevertheless, it turns out that his work is closely related to the tilting property of $T_{\Phi}$ for a codimension filtration $\Phi$, and it actually makes sense to ask the next question:
\begin{center}\emph{Is $T_\Phi$ tilting if and only if $R$ is a homomorphic image of a Cohen--Macaulay ring?}\end{center} 
An important viewpoint to study this question is flatness of the endomorphism ring of $T_\Phi$ over $R$.
Indeed, an essential ingredient of our proof of \cref{intro-tilt} under the assumption \cref{intro-tilt-CM} is that $\End_{\D(R)}(T_{\Phi})$ is flat whenever $R$ is a Cohen--Macaulay ring of finite Krull dimension (\cref{flat-End}).
Furthermore, if $R$ is a 1-dimensional commutative noetherian ring, then it has a codimension filtration $\Phi$, and $T_{\Phi}$ is always tilting (\cref{1-dim}). In this case, the endomorphism ring can explicitly be computed as follows:  
\begin{equation}\label{intro-End}
\End_{\D(R)}(T_{\Phi})\cong
\begin{pmatrix}
	S^{-1}R &  0 \vspace{2mm}\\
	\displaystyle{(\prod_{\mm\in \mSpec R} \widehat{R_{\mm}})\otimes_RS^{-1}R} & \displaystyle{\prod_{\mm\in \mSpec R} \widehat{R_{\mm}}}
	\end{pmatrix},
\end{equation}
where $S$ stands for the complement of the union of all non-maximal prime ideals and $\widehat{R_{\mm}}$ stands for the $\mm$-adic completion of the local ring $R_{\mm}$ (\cref{1-dim-example}).
Evidently, $\End_{\D(R)}(T_{\Phi})$ is flat over $R$, and in fact, every 1-dimensional ring satisfies the condition \cref{intro-tilt-CM} of \cref{intro-tilt} by results of Kawasaki (\cref{1-dim-CM}).
More generally, given any commutative noetherian ring $R$ with a codimension filtration $\Phi$, the flatness of $\End_{\D(R)}(T_{\Phi})$ implies that all the formal fibers of all the localizations of $R$ are Cohen--Macaulay (\cref{stalk}). 
This fact along with a result of Kawasaki (\cref{Kawasaki}\cref{Kawasaki-local}) enables us to show the next theorem. Although Kawasaki's result requires that $R$ is universally catenary, perhaps surprisingly, this follows from the tilting property of $T_\Phi$ (\cref{univ-catenary}).

\begin{theorem}[\cref{flat-End-thm}]\label{intro-flat-End}
	Let $R$ be a commutative noetherian local ring with a codimension filtration $\Phi$. The following conditions are equivalent:
	\begin{enumerate}[label=(\arabic*), font=\normalfont]
		\item\label{intro-flat-End-cond} $T_\Phi$ is tilting and $\End_{\D(R)}(T_\Phi)$ is a flat $R$-module.
		\item $R$ is a homomorphic image of a Cohen--Macaulay local ring.
	\end{enumerate}
\end{theorem}

It is natural to ask if we can get rid of the second condition from \cref{intro-flat-End-cond} of \cref{intro-flat-End} so that Cohen--Macaulay homomorphic images can be determined purely by the tilting property of $T_\Phi$. Namely, we are interested in the following question:

\begin{center}
\emph{Is $\End_{\D(R)}(T_\Phi)$ flat as an $R$-module whenever $T_\Phi$ is tilting?}
\end{center}
We will show that the answer is affirmative for any ring $R$ of Krull dimension two (\cref{2-dim-tilt-flat-end}). Note that this is the least dimension in which rings may \emph{not} be homomorphic images of Cohen--Macaulay rings; see \cref{1-dim-CM,non-CM-im}.

Up to now, we have focused on the silting case, but the last three sections provide the cosilting counterparts for all of the results mentioned so far. Most of them can not be deduced by formally dual arguments, while it is actually efficient to start with silting objects, because we can translate each silting object $T_\Phi$ to a cosilting object $\RHom_R(T_\Phi, E)$ in $\D(R)$ by an injective cogenerator $E$ in $\Mod R$; see \cref{(co)silt-dual}. 
We further equivalently replace each cosilting object $\RHom_R(T_\Phi, E)$ by a more explicit one $C_\Phi$ (\cref{CPhi}) using local duality and Matlis duality; see \cref{slice-cosilt} and its proof.

Given a slice sp-filtration $\Phi$, we obtain the cosilting object $C_\Phi$ as mentioned above, and it induces a t-structure in a natural way. This assignment is compatible with the classification of compactly generated t-structures in $\D(R)$ based on the sp-filtrations of $\Spec R$ (\cite{ATJLS10}). 
When $R$ admits a classical dualizing complex $D$, $D$ gives a codimension function (\cref{cd-func}), from which we obtain a codimension filtration $\Phi$. Then the t-structure induced by $C_\Phi$ is precisely the compactly generated t-structure in $\D(R)$ which restricts to the \emph{Cohen--Macaulay t-structure} in $\D^{\bd}_{\fg}(R)$ with respect to $D$ in the sense of \cite[\S 6]{ATJLS10}; see \cref{dc-cotilt-rem}.
Even if a dualizing complex is not available, as far as $R$ admits a codimension filtration $\Phi$, we obtain the heart of the t-structure in $\D(R)$ induced by $C_\Phi$, and this heart, up to equivalence, does not depend on the choice of the codimension filtration. 
We call this canonical abelian category the \emph{Cohen--Macaulay heart} (\cref{cm-heart}). 
As far as $R$ has finite Krull dimension, $C_{\Phi}$ is cotilting if and only if the inclusion from the Cohen--Macaulay heart induces a derived equivalence.
Moreover, if $R$ is local and of Krull dimension 2, we can verify that $C_{\Phi}$ is cotilting if and only if $R$ is a homomorphic image of a Cohen--Macaulay local ring (\cref{2-dim-tilt}), as in the silting case.

The paper is structured as follows. In \cref{preliminaries}, we gather required facts and notions from several different topics: the silting theory in triangulated categories, localization theory of $\D(R)$, depth and width over commutative noetherian rings, sp-filtrations of $\Spec R$, and dualizing complexes. We also recall the technical notions of cofiltrations and Lukas lemma for complexes which is required to complete the proof of \cref{intro-slice}. \cref{slice-section} introduces slice sp-filtrations and codimension filtrations, and we provide a rather concise proof for \cref{intro-slice,intro-tilt} restricted to rings admitting a classical dualizing complex (\cref{cdc-shortcut}); for this, the infinite completion of Grothendieck duality (\cref{NIK}) plays a key role. The goal of \cref{proof-theorem} is to establish \cref{intro-slice} in full generality, which demands a more technical approach. In \cref{cosilting-section} we deal with the dual setting for cosilting objects, and in \cref{end-section} we study the flatness of the endomorphism ring of both tilting and cotilting objects induced by codimension functions. Finally, in \cref{Hom-Im-CM}, we consider rings which are homomorphic images of Cohen--Macaulay rings and finish the proofs of \cref{intro-tilt,intro-flat-End}. The questions listed above are restated for the silting and cotilting case both (\cref{question,Q-flatend}\cref{Q-flatend-codim}), and they are affirmatively answered for local and non-local rings, respectively, of Krull dimension at most two.

\noindent {\bf Acknowledgment.} The authors would like to thank Takesi Kawasaki, Sergio Pavon, Ryo Takahashi, and Jorge Vitória for very helpful discussions concerning the manuscript. 
The authors would also like to thank the referee for carefully reading the manuscript and giving useful suggestions.

The second author was a postdoc at the University of Verona when this project was started. Part of the project was completed while he was a JSPS Research Fellow at the Nagoya University and later at the University of Tokyo.

\section{Preliminaries}\label{preliminaries}
\noindent {\bf Convention.} Throughout this paper, subcategories of a given category are assumed to be full, additive and closed under isomorphisms.

\subsection{Silting objects in derived categories}
We begin with recalling basic facts on (co)silting objects in the sense of \cite{NSZ19} and \cite{PV18}. 
Some of the facts work well in more general triangulated categories, but we confine ourselves to the case of unbounded derived categories of rings.

\subsubsection{} 
Let $R$ be a ring and let $\D(R)$ be the unbounded derived category of the category of all right $R$-modules. We denote by $\Sigma$ the suspension functor on $\D(R)$.
For a class $\cC$ of objects in $\D(R)$ and a set $I$ of integers, define the following subcategories of $\D(R)$:
		\begin{align*}
		\cC\Perp{I} & := \{X \in \D(R) \mid \Hom_{\D(R)}(C,\Sigma^i X) = 0 ~\forall C \in \mathcal{C} ~\forall i \in I\}\text{ and}\\
		\Perp{I}\cC & := \{X \in \D(R) \mid \Hom_{\D(R)}(X,\Sigma^i C) = 0 ~\forall C \in \mathcal{C} ~\forall i \in I\}.
		\end{align*}
The role of the set $I$ will often be played by symbols of the form $n$, $>n$, $<n$, $\geq n$, $\leq n$, and $\neq n$, interpreted in the obvious way, where $n$ is an integer. For example, $\cC\Perp{n}:=\cC\Perp{\{n\}}$ and $\cC\Perp{>n}:=\cC\Perp{\{i\in \ZZ \mid i>n\}}$.
If $\mathcal{C}=\{X\}$ for a single object $X\in \D(R)$, we use $X\Perp{I}$ and $\Perp{I}X$ instead of $\{X\}\Perp{I}$ and $\Perp{I}\{X\}$, respectively.

\subsubsection{} A \emph{t-structure} in $\D(R)$ is a pair $(\cU, \cV)$ of subcategories satisfying the following axioms (\cite{BBD82}):
	\begin{enumerate}[label=(t-\arabic*), font=\normalfont,align=left]
		\item \label{t-1} $\Hom_{\D(R)}(U,V) = 0$ for all $U \in \cU$ and $V \in \cV$,
		\item \label{t-2} for any $X \in \D(R)$ there is a triangle
			\[U \rightarrow X \rightarrow V \rightarrow \Sigma U\]
			with $U \in \cU$ and $V \in \cV$, and
		\item \label{t-3} $\Sigma \cU \subseteq \cU$.
	\end{enumerate}
	The subcategory $\cU$ is called  the \emph{aisle} of the t-structure, and $\cV$ is called the \emph{coaisle}. 
	A consequence of the axioms is that $\cV = \cU\Perp{0}$ and $\cU = \Perp{0}\cV$. In particular, the aisle and the coaisle both are closed under direct summands, and the former is closed under coproducts, while the latter is closed under products. Moreover, \cref{t-3} is equivalent to
	\begin{enumerate}[align=left]
		\item[(t-3')] $\Sigma^{-1} \cV \subseteq \cV$
	\end{enumerate}
	under the other two axioms.
	
	If $(\cU,\cV)$ is a t-structure, then the triangle in \cref{t-2} is functorially unique up to isomorphism. More precisely, the inclusions from $\cU$ and $\cV$ into $\D(R)$ admit a right adjoint $u: \D(R) \rightarrow \cU$ and a left adjoint $v: \D(R) \rightarrow \cV$ respectively, and the counit $u \to  \Id_{\D(R)}$ and the unit $\Id_{\D(R)} \rightarrow v$ of these adjunctions yield a triangle
\begin{equation}
u X \rightarrow X \rightarrow vX \rightarrow \Sigma uX
\end{equation}
for every $X\in \D(R)$, and this triangle can be identified with the triangle in \cref{t-2} via the isomorphisms $U\isoto u X$ and $v X\isoto V$ induced by adjointness.
Note also that the essential images $\Ima u$ and $\Ima v$ coincide with $\cU$ and $\cV$, respectively, and the kernels $\Ker u$ and $\Ker v$ coincide with $\cV$ and $\cU$, respectively; see \cite[Proposition 4.1.4]{KS06} for example.

The intersection $\cH = \cU \cap \Sigma\cV$ is called the \emph{heart} of the t-structure $(\cU,\cV)$. It is shown in \cite[Theorem 1.3.6]{BBD82} that the heart of a t-structure is an abelian category such that its exact structure is induced by the triangles belonging to the heart. In addition, the truncation functors $u$ and $v$ of the t-structure give rise to a cohomological functor $H^0_{\cH}: \D(R) \to \cH$, which can be defined as the composition $u \circ\Sigma \circ v\circ \Sigma^{-1}$. See also \cite[\S 10.1]{KS94} for details.

For each $n \in \ZZ$, there is a canonical t-structure $(\D^{\leq n}(R),\D^{>n}(R))$ in $\D(R)$, where $\D^{\leq n} (R) := \{X \in \D(R) \mid H^i(X) = 0 ~\forall i > n\}$ and $\D^{> n} (R) := \{X \in \D(R) \mid H^i(X) = 0 ~\forall i \leq n\}$. We say that a t-structure $(\cU,\cV)$ is \emph{intermediate} if there are integers $n \leq m$ such that $\D^{\leq n}(R) \subseteq \cU \subseteq \D^{\leq m}(R)$. We also call such a t-structure \emph{$l$-intermediate} in case the integers satisfy $m-n \leq l$ for some non-negative integer $l$. 
A t-structure $(\cU,\cV)$ is called \emph{non-degenerate} provided that $\bigcap_{n \in \ZZ}\Sigma^n \cU = 0$ and $\bigcap_{n \in \ZZ}\Sigma^n \cV = 0$.  
It is straightforward to check that any intermediate t-structure is non-degenerate. Non-degenerate t-structures are well-behaved in the sense that they are fully determined by the induced cohomological functor $H^0_{\cH}$ to the heart; see \cite[Proposition 1.3.7]{BBD82} or \cite[Proposition 10.1.10]{KS94}.

A t-structure $(\cU,\cV)$ in $\D(R)$ is called \emph{stable} if $\Sigma^{-1}\cU \subseteq \cU$, or equivalently, if both $\cU$ and $\cV$ are triangulated subcategories of $\D(R)$.
The notion of stable t-structures can be essentially identified with the notion of Bousfield localizations; see \cref{Bous-loc} below.

A t-structure $(\cU,\cV)$ in $\D(R)$ is called \emph{compactly generated} if $\cV=\cS\Perp{0}$ for some set $\cS$ of compact objects in $\D(R)$. Recall that an object in $\D(R)$ is compact if and only if it is isomorphic to a bounded complex of finitely generated projected $R$-modules; see \cite[2.1.3]{Nee96}. See also \cite[Theorem 2.2]{BIK12a}.
\begin{remark}\label{compact-t-Gro}
If a t-structure in $\D(R)$ is compactly generated, then its heart is a Grothendieck category by \cite[Theorem D]{SSV17}.
In fact, this is a locally finitely presented Grothendieck category by \cite[Theorem 1.6]{SS20}.
\end{remark}

\subsubsection{}\label{silt-equiv}

An object $T \in \D(R)$ is called \emph{silting} if the pair $(T\Perp{>0},T\Perp{\leq 0})$ is a t-structure in $\D(R)$.
 Dually, an object $C \in \D(R)$ is \emph{cosilting} if $(\Perp{\leq 0}C, \Perp{> 0}C)$ is a t-structure in $\D(R)$.
Given a silting object $T$ (resp. a cosilting object $C$), it is easily seen that $T\in T\Perp{>0}$ (resp. $C\in \Perp{> 0}C$).
A t-structure is called \emph{silting} (resp. \emph{cosilting}) if it is induced by a silting (resp. cosilting) object as above.
All silting t-structures and cosilting t-structures are non-degenerate; see \cite[Proposition 4.3]{PV18}.

Silting objects $T$ and $T'$ are said to be \emph{equivalent} if the induced t-structures $(T\Perp{>0},T\Perp{\leq 0})$ and $(T'\Perp{>0},T'\Perp{\leq 0})$ coincide. Similarly, cosilting objects $C$ and $C'$ are  said to be equivalent if the induced t-structures $(\Perp{\leq 0}C, \Perp{> 0}C)$ and $(\Perp{\leq 0}C', \Perp{> 0}C')$ coincide.

If an object $X$ satisfies $\Add(X) = \Add(X')$ (resp. $\Prod(X) = \Prod(X')$) for a silting (resp. cosilting) object $X'$, then $X$ is a silting (resp. cosilting) object equivalent to $X'$, and the converse is also true (see, e.g.,  \cite[Lemma 4.5]{PV18}), where we denote by $\Add(X)$ (resp. $\Prod(X)$) the subcategory of $\D(R)$ consisting of all objects being direct summands of a coproduct (resp. of a product) of copies of $X$.

\begin{lemma}\label{silting-cond}
	An object $T \in \D(R)$ is silting if and only if the following conditions hold:
	\begin{enumerate}[label=(\arabic*), font=\normalfont]
		\item $T \in T\Perp{>0}$,
		\item $T\Perp{\ZZ} = 0$, and
		\item $T\Perp{>0}$ is closed under coproducts.
	\end{enumerate}
\end{lemma}

\begin{proof}
	See \cite[Proposition 4.13]{PV18}.
\end{proof}

\begin{lemma}\label{cosilting-cond}
	A pure-injective object $C \in \D(R)$ is cosilting if and only if the following conditions hold:
	\begin{enumerate}[label=(\arabic*), font=\normalfont]
		\item $C \in \Perp{> 0}C$,
		\item $\Perp{\ZZ}C=0$, and
		\item $\Perp{> 0}C$ is closed under products. 
	\end{enumerate}
\end{lemma}

\begin{proof}
	The assumptions imply by \cite[Lemma 4.8]{AHMV17} that there is a t-structure of the form $(\cU,\Perp{> 0}C)$ in $\D(R)$. Since $C \in \Perp{> 0}C$ and $\cU=\Perp{0}(\Perp{> 0}C)=\Perp{\leq 0}(\Perp{> 0}C)$, it follows that $\cU\subseteq \Perp{\leq 0}C$. It remains to show the converse inclusion. Let $X \in \Perp{\leq 0}C$ and consider the triangle $uX \to X \to vX \to \Sigma uX$ with $uX \in \cU$ and $vX \in \Perp{> 0}C$. Since both $X$ and $uX$ belong to $\Perp{\leq 0}C$, we see that $vX \in \Perp{\leq 0}C$. But then $vX \in \Perp{\leq 0}C \cap \Perp{> 0}C = \Perp{\ZZ}C=0$. Therefore, $X \cong uX \in \cU$.
\end{proof}

\begin{remark}\label{rem-silt-cond}
Alonso Tarr\'{\i}o, Jerem\'{\i}as L\'{o}pez and Souto Salorio \cite[Theorem 3.4]{ATJLSS03} proved that, for an object $X$ in the unbounded derived category of a Grothendieck category, there exists a t-structure whose aisle is the smallest subcategory containing $X$ and closed under extensions, suspensions, and coproducts. This fact, which is essential for \cite[Proposition 4.13]{PV18}, has been extended by Neeman \cite[Theorem 2.3]{Nee21} to well generated triangulated categories. Consequently, we may replace $\D(R)$ in \cref{silting-cond} by a well generated triangulated category. On the other hand, the proof of \cref{cosilting-cond} will work when $\D(R)$ is replaced by a compactly generated triangulated category $\cT$ as \cite[Lemma 4.8]{AHMV17} is proved for $\cT$.
\end{remark}

Let $X=(\cdots \to X^{i-1} \to X^i\to X^{i+1}\to \cdots)$ be a complex of $R$-modules. We say that $X$ is \emph{bounded above} (resp. \emph{bounded below}) if $X^i=0$ for $i\gg 0$ (resp. $i\ll 0$).
If $X$ is bounded above and below, it is called \emph{bounded}.

\begin{proposition}\label{silt-bound}
	Let $T$ be a bounded complex of projective $R$-modules. Then $T$ is a silting object in $\D(R)$ if and only if the following conditions hold:
	\begin{enumerate}[label=(\arabic*), font=\normalfont]
		\item\label{silt-bound-1} $\Add(T) \subseteq T\Perp{>0}$, and
		\item\label{silt-bound-2} all bounded complexes of projective $R$-modules belong to the smallest triangulated subcategory of $\D(R)$ containing $\Add(T)$.
	\end{enumerate}
\end{proposition}
\begin{proof}
See \cite[Example 2.9(4) and Proposition 4.2]{AHMV16a} or \cite[Proposition 5.3]{AH19}.
\end{proof}

\begin{proposition}\label{cosilt-bound}
	Let $C$ be a bounded complex of injective $R$-modules. Then $C$ is a cosilting object in $\D(R)$ if and only if the following conditions hold:
	\begin{enumerate}[label=(\arabic*), font=\normalfont]
		\item\label{cosilt-bound-1} $\Prod(C) \subseteq \Perp{>0}C$, and
		\item\label{cosilt-bound-2} all bounded complexes of injective $R$-modules belong to the smallest triangulated subcategory of $\D(R)$ containing $\Prod(C)$.
	\end{enumerate}
\end{proposition}
\begin{proof}
This is stated in \cite[Proposition 6.8(2)]{AH19} with a sketch proof, and the ``if'' part follows from \cite[Proposition 3.10]{MV18}. Let us prove the ``only if'' part in a little more detail.
Assume that $C$ is a cosilting object in $\D(R)$. Then we have the t-structure $(\Perp{\leq 0}C, \Perp{> 0}C)$ and $C\in \Perp{> 0}C$, so \cref{cosilt-bound-1} holds. 
By assumption, we have $C\in \D^{>n}(R)$ for some integer $n$.
Thus $\Sigma^{i}C\in \D^{>n}(R)$ for all $i\leq 0$. It follows that $\Perp{\leq 0}C\supseteq \Perp{0}\D^{> n}(R) = \D^{\leq n}(R)$.
Applying $(-)\Perp{0}$ to this inclusion, we have $\Perp{> 0}C \subseteq \D^{> n}(R)$, which implies 
\cref{cosilt-bound-2} by \cite[Proposition 2.10]{ZW17}.
\end{proof}

If a bounded complex $T$ of projective $R$-modules (resp. a bounded complex $C$ of injective $R$-modules) satisfies the two conditions in \cref{silt-bound} (resp. \cref{cosilt-bound}), then $T$ (resp. C) is called a \emph{silting complex} (resp. a \emph{cosilting complex}); see e.g. \cite[Definitions 5.1 and 6.1]{AH19}.
By the propositions, a silting (resp. cosilting) complex is a silting (resp. cosilting) object in $\D(R)$, but in general the latter may not be isomorphic to any bounded complex of projective (resp. injective) modules; see \cref{ex-(co)tilt}.

\begin{remark}\label{thick-closure}
Let $\cT$ be a triangulated category with small coproducts (resp. small products) and $\cS_0$ be a subcategory of $\cT$ closed under small coproducts (resp. small products).
Let $\cS$ be the smallest triangulated subcategory of $\cT$ containing $\cS_0$.
Then, for any object $X\in \cS$ and any cardinal $\varkappa$, the coproduct $X^{(\varkappa)}$ (resp. the product $X^\varkappa$) belongs to $\cS$. This fact follows from the construction of $\cS$ by \cite[\S 3.3, Lemma (1)]{Kra07} along with \cite[Proposition 1.2.1 and Remark 1.2.2]{Nee01}. It also follows that $\cS$ is a \emph{thick subcategory}, i.e., a triangulated subcategory closed under direct summands (see, e.g., \cite[Lemma 3.6]{AH19}).
Hence we have $\cS=\thick (\cS_0)$, where $\thick (\cS_0)$ denotes the smallest triangulated subcategory containing $\cS_0$ and closed under direct summands.

By the above observation applied to $\cT:=\D(R)$, \cref{silt-bound}\cref{silt-bound-2} is equivalent to the condition that $R\in \thick (\Add (T))$. Dually, \cref{cosilt-bound}\cref{cosilt-bound-2} is equivalent to the condition that $E\in \thick (\Prod (C))$, where $E$ is an injective cogenerator of $\Mod R$. 
\end{remark}

Let $l$ be a positive integer. By an \emph{$l$-term} silting complex (resp. an \emph{$l$-term} cosilting complex), we mean a silting (resp. cosilting) complex concentrated degrees from $k$ to $l+k-1$ for some integer $k$.
If an object $X$ in $\D(R)$ is isomorphic to a silting (resp. cosilting) complex, then $X$ is called a \emph{bounded silting object} (resp. a \emph{bounded cosilting object}); cf. \cite[Definition 4.15 and Proposition 4.17]{PV18}.
A silting (resp. cosilting) object $X$ in $\D(R)$ is isomorphic to an $l$-term silting complex (resp. an $l$-term cosilting complex) if and only if the induced t-structure is $(l-1)$-intermediate (see \cite[Definition 4.1 and Lemma 4.5]{AHMV16a} and \cite[Lemma 4.5 and Proposition 4.17]{PV18}); in particular, $X$ is a bounded silting object (resp. a bounded cosilting object) if and only if the induced t-structure is intermediate.

\subsubsection{}
\label{of-finite-type}
A silting object $T$ is called \emph{of finite type} provided that there is a set $\cS$ of compact objects in $\D(R)$ such that $T\Perp{>0} = \cS\Perp{0}$.  A cosilting object $C$ is called \emph{of cofinite type} provided that there is a set $\cS$ of compact objects in $\D(R)$ such that $\Perp{>0}C = \cS\Perp{0}$.
Any bounded silting object in $\D(R)$ is of finite type by \cref{silt-bound} and \cite[Theorem 3.6]{MV18}, while any bounded cosilting object is pure-injective by \cref{cosilt-bound} and \cite[Proposition 3.10]{MV18}.

A \emph{silting t-structure of finite type} refers to a t-structure which is induced by a silting object of finite type, and a \emph{cosilting t-structure of cofinite type} refers to a t-structure which is induced by a cosilting object of cofinite type. We remark that there is a known equality:
\begin{equation}\label{cosilt-cof}
    \begin{Bmatrix}
	\text{cosilting t-structures}\\
	\text{of cofinite type}\\
	\text{in $\D(R)$}  
    \end{Bmatrix}
    =
    \begin{Bmatrix} 
	\text{non-degenerate}\\
	\text{compactly generated}\\
	\text{t-structures in $\D(R)$}
	\end{Bmatrix}.
\end{equation}
The inclusion ``$\subseteq$'' is clear, and the converse inclusion follows from \cite[Theorem 3.5]{AHMV17} and \cref{compact-t-Gro}.

\subsubsection{}\label{charact-dual}
Let $S$ be a commutative ring and let $R$ be an $S$-algebra, that is, $R$ is a ring endowed with a ring homomorphism $S\to R$  whose image is contained in the center of $R$.
Let $E$ be an injective cogenerator of $\Mod S$ and denote by $(-)^+$ the contravariant functor $\RHom_S(-,E)$ from $\D(R)$ to $\D(R^\op)$, where $R^\op$ is the opposite ring of $R$. Given a silting object $T$ of finite type in $\D(R)$, $T^+$ is a  cosilting object of cofinite type in $\D(R^\op)$. This assignment gives rise to an injection from the equivalence classes of silting objects of finite type in $\D(R)$ to the equivalence classes of cosilting objects of cofinite type in $\D(R^\op)$.
Moreover, this injection restricts to a bijection from 
the equivalence classes of bounded silting objects in $\D(R)$ and 
the equivalence classes of bounded cosilting objects of cofinite type in $\D(R^\op)$. See \cite[Theorem 3.3]{AHH21}. Note that every bounded silting object is of finite type as mentioned in \cref{of-finite-type}.

\subsubsection{}\label{subsec-real-func}
Let $T \in \D(R)$ be a silting object and denote by $\cH_T$ the heart of the t-structure $(T\Perp{>0},T\Perp{\leq 0})$ induced by $T$. Then $\cH_T$ has a projective generator (\cite[Proposition 4.3]{PV18}), and the inclusion $\cH_T \ha \D(R)$ extends to a triangulated functor $\D^{\bd}(\cH_T) \to \D(R)$ called a \emph{realization functor} (\cite[Theorem 3.11]{PV18}).
This can be regarded as a functor $\D^{\bd}(\cH_T)\to \D^{\bd}(R)$ between the bounded derived categories if the silting object $T$ is bounded (\cite[Lemma 4.14]{PV18}).
The realization functor is in principle not unique and depends on the choice of an ``f-enhancement''  of $\D(R)$. See \cite[\S 3]{PV18} for details.
We also remark that the existence of a realization functor $\D(\cH_T) \to \D(R)$ between the unbounded derived categories has been established in \cite[\S 6]{Vir18}.

The dual theory for a cosilting object $C \in \D(R)$ is also available in the references above. 
We note that if the cosilting object $C$ is bounded, then not only does the associated heart $\cH_C$ admit an injective cogenerator, but it is in fact a Grothendieck category (\cite[Proposition 3.10]{MV18}). 
This fact more generally holds as far as the cosilting object $C$ is pure-injective (\cite[Theorem 4.6]{Lak20}), where recall that every bounded cosilting object is pure-injective as mentioned in \cref{of-finite-type}. 

If the cosilting object $C$ is of cofinite type, then $C$ is pure-injective (\cite[Theorem 4.6]{Lak20}) and $\cH_C$ is a locally finitely presentable Grothendieck category (\cref{compact-t-Gro}).
In general, a bounded cosilting object may not be of cofinite type (\cite[Example 3.12]{MV18}), but if $R$ is a commutative noetherian ring, then every pure-injective cosilting object in $\D(R)$ is of cofinite type as shown by the first and second authors (\cite[Corollary 2.14]{HN21}). 
All the cosilting objects we consider for commutative noetherian rings in this paper are of cofinite type, and so the induced hearts are locally finitely presentable Grothendieck categories.

\subsubsection{}\label{subsec-der-eq}
A silting object $T\in \D(R)$ is called \emph{tilting} provided that $\Add(T) \subseteq T\Perp{<0}$. 
Assume that $T$ is a bounded silting object and let $\D^{\bd}(\cH_T)\to \D^{\bd}(R)$ be a realization functor. 
Then $T$ is tilting if and only if the realization functor is a triangulated equivalence; see \cite[Corollary 5.2]{PV18} (which is stated for a realization functor with respect to a specific f-enhancement \cite[Example 3.2]{PV18}, but its proof works for any f-enhancement of $\D(R)$).
The realization functor $\D(\cH_T)\to \D(R)$ due to Virili is also a triangulated equivalence if and only if $T$ is tilting (\cite[Theorem 7.12]{Vir18}).

We remark that a similar theory so far does not seem to be fully developed for silting objects which are not bounded. It at least holds that, given a silting object $T\in \D(R)$ and a realization functor $\D^{\bd}(\cH_T)\to \D(R)$, $T$ is tilting if and only if the realization functor is fully faithful (\cite[Proposition 5.1]{PV18}).

A cosilting object $C$ is called \emph{cotilting} provided that $\Prod(C) \subseteq \Perp{<0}C$. There are analogous results about realization functors and derived equivalences; see the references above. 

If a silting (resp. cosilting) complex $X$ is a tilting (resp. cotilting) object in $\D(R)$, then $X$ is called a \emph{tilting complex} (resp. a \emph{cotilting complex}).
A \emph{bounded tilting object} (resp. a \emph{bounded cotilting object}) $X\in \D(R)$ means a bounded silting object (resp. a bounded cosilting object) which is tilting (resp. cotilting) in $\D(R)$.

\begin{remark}\label{tilt-module}
An $R$-module $M$ is tilting (resp. cotilting) in the sense of \cite{AHC01} if and only if $M$ is a bounded tilting object (resp. a bounded cotilting object) in $\D(R)$ by \cref{silt-bound} and \cite[Corollary 3.7]{Wei13} (resp. \cref{cosilt-bound} and \cite[Proposition 2.6]{ZW17}).
If a tilting (resp. cotilting) module has projective (resp. injective) dimension at most $n$, then it is called \emph{$n$-tilting} (resp. \emph{$n$-cotilting}).
\end{remark}

\subsection{Bousfield localization}\label{Bous-loc}
A triangulated functor $\lambda: \D(R) \to \D(R)$ is called a \emph{(Bousfield) localization functor} if there exists a morphism $\eta:\Id_{\D(R)}\to \lambda$ such that $\lambda \eta : \lambda \to \lambda^2$ is invertible and $\lambda \eta =\eta \lambda$.
If $\lambda:\D(R)\to \D(R)$ is a localization functor, then 
$(\Ker \lambda, \Ima \lambda)$ is a stable t-structure. In particular, $\Ima \lambda$ (resp. $\Ker \lambda$) is a reflective (resp. coreflective) subcategory, that is, the inclusion functor $\Ima \lambda\hookrightarrow \D(R)$ (resp. $\Ker \lambda\hookrightarrow \D(R)$) admits a left (resp. right) adjoint, where $\lambda$ can be naturally regarded as the left adjoint $\D(R) \to \Ima \lambda$. Conversely, given a stable t-structure $(\cL,\cC)$, by definition, the inclusion functors $\cL \hookrightarrow \D(R)$ and $\cC \hookrightarrow\D(R)$ admit  a right adjoint $\gamma: \D(R)\to \cL$ and a left adjoint $\lambda:\D(R)\to \cC$ respectively, where $\gamma$ and $\lambda$ are triangulated (\cite[Lemma 5.3.6]{Nee01}).
Further, if $\lambda$ is interpreted as the composition $\D(R)\xrightarrow{\lambda} \cC\hookrightarrow \D(R)$, then the functor $\lambda: \D(R)\to \D(R)$ is a localization functor.
Dually, if $\gamma$ is interpreted as the composition $\D(R)\xrightarrow{\gamma} \cL\hookrightarrow \D(R)$, then the functor $\gamma: \D(R)\to \D(R)$ is a \emph{colocalization functor}, that is, the counit morphism $\varepsilon: \gamma\to \Id_{\D(R)}$ satisfies that $\gamma \varepsilon : \gamma^2\to \gamma$ is invertible and $\gamma \varepsilon =\varepsilon \gamma$. See \cite[Proposition 4.9.1]{Kra10} for details.

A triangulated subcategory of $\D(R)$ is called \emph{localizing} (resp. \emph{colocalizing}) if it is closed under small coproducts (resp. small products).
Given a subcategory $\cX$, the pair $(\cX,\cX\Perp{0})$ is a stable t-structure if and only if $\cX$ is a coreflective localizing subcategory; similarly, the pair $(\Perp{0}\cX, \cX)$ is a stable t-structure if and only if $\cX$ is a reflective colocalizing subcategory. A localization functor $\lambda: \D(R)\to \D(R)$ is called \emph{smashing} if it commutes with small coproducts, or equivalently, if $\Ima \lambda$ is closed under small coproducts.
In this case, $\Ker \lambda$ is called a \emph{smashing subcategory}.

Now, assume that $R$ is a commutative noetherian ring. The theorem of Neeman \cite[Theorem 2.8]{Nee92} asserts that localizing subcategories of the derived category $\D(R)$ are in bijective correspondence with subsets of the Zariski spectrum $\Spec R$. 
If $W\subseteq \Spec R$ is the subset corresponding to a localizing subcategory $\cL$, then $\cL$ is the smallest localizing subcategory containing the small set $\{\kappa(\pp) \mid \pp\in W\}$, where $\kappa(\pp) = R_\pp/\pp R_\pp$. Thus, it also follows that every localizing subcategory of $\D(R)$ is coreflective; see \cite[\S 5.1 and \S 7.2]{Kra10}.  In other words, $(\cL, \cL\Perp{0})$ is a stable t-structure.

 This correspondence can be described by using an invariant in $\D(R)$. The \emph{support} of a complex $X\in \D(R)$ is the set
	\[\supp X := \{\pp \in \Spec R \mid \kappa(\pp) \otimes_R^\mathbf{L} X \neq 0\}\text{.}\]
Note that $\supp X=\emptyset$ if and only if $X=0$ in $\D(R)$ by \cite[Lemma 2.12]{Nee92}.
For each $W\subseteq \Spec R$, the subcategory
	\[\cL_W := \{X \in \D(R) \mid \supp X \subseteq W\}\]
is localizing, and the bijection due to Neeman is given by $W\mapsto \cL_W$.
As mentioned above, $\cL_W$ is coreflective, so the inclusion functor $\cL_W\hookrightarrow \D(R)$ admits a right adjoint, which we denote by $\gamma_W:\D(R)\to \cL_W$. Since the pair $(\cL_W, \cL_W^{\perp_0})$ is a stable t-structure, the inclusion $\cL_W^{\perp_0}\hookrightarrow \D(R)$ admits a left adjoint, which we denote by $\lambda_W: \D(R)\to \cL_W^{\perp_0}$.
Regarding this functor as the composition $\D(R) \xrightarrow{\lambda_W} \cL_W^{\perp_0}\hookrightarrow \D(R)$, we obtain a localization functor $\lambda_W:\D(R)\to \D(R)$. By \cite[Theorem 3.3]{Nee92}, the localization functor $\lambda_W$ is smashing if and only if $W$ is \emph{specialization closed}, that is, $W$ is an upper subset of the poset $(\Spec R,\subseteq)$, or equivalently, $W$ is an arbitrary union of Zariski closed subsets of $\Spec R$. To understand $\lambda_W$ from another viewpoint, let us use a dual invariant to the notion of support.

The \emph{cosupport} of a complex $X\in \D(R)$ is the set
	\[\cosupp X  := \{\pp \in \Spec R \mid \RHom_R(\kappa(\pp),X) \neq 0\}.\]
Note that $\cosupp X=\emptyset$ if and only if $X=0$ in $\D(R)$ by \cite[Theorem 2.8]{Nee92}.
For each $W\subseteq \Spec R$, the subcategory 
	\[\cC^{W} := \{X \in \D(R) \mid \cosupp X \subseteq W\}\] is colocalizing (and in fact every colocalizing subcategory of $\D(R)$ is of this form by \cite{Nee11b}).
Since the localizing subcategory $\cL_{W^c}$ for $W^\cp:=\Spec R\sm W$
is generated by $\{\kappa(\pp) \mid \pp \in W^\cp\}$, we have 
	\[\cL_{W^\cp}^{\perp_0} = \cC^{W}.\] 
Denote by $\lambda^W : \D(R)\to \cC^{W}$ the left adjoint to the inclusion functor $ \cC^{W}=\cL_{W^\cp}^{\perp_0}\hookrightarrow \D(R)$, that is, 
\[\lambda^W := \lambda_{W^\cp}.\]
By definition, $(\cL_{W^\cp}, \cC^W)$ is a stable t-structure in $\D(R)$, so we have the approximation triangle
\begin{equation}\label{stable-approx}
\gamma_{W^\cp}X\to X\to \lambda^{W}X\to \Sigma\gamma_{W^\cp}X
\end{equation}
for every $X\in \D(R)$.

If $W_0\subseteq W\subseteq \Spec R$, then $\cL_{W_0}\subseteq \cL_W$, so we have
\begin{equation}\label{trivial-iso}
\gamma_{W_0}\gamma_{W}\cong \gamma_{W_0}\cong \gamma_{W}\gamma_{W_0}\hspace{8pt}\text{and}\hspace{8pt}\lambda^{W_0}\lambda^{W}\cong\lambda^{W_0}\cong \lambda^{W}\lambda^{W_0}
\end{equation}
by a standard argument; see \cite[Remark 3.7(i)]{NY18a} and \cite[Remark 2.7(ii)]{NY18b}.

If $V\subseteq \Spec R$ is a specialization closed subset, then the following equalities hold:
\begin{equation}\label{bi-loc}
\cL\Perp{0}_V = \cL_{V^\cp} = \cC^{V^\cp}= \Perp{0}\cC^{V},
\end{equation}
which can be deduced from \cite[Theorem 3.3]{Nee92}.

A subset $W\subseteq \Spec R$ is called \emph{generalization closed} if its complement $W^\cp$ is specialization closed.
If we take the generalization closed subset $U(\pp) := \{\qq \in \Spec R \mid \qq \subseteq \pp\}$ for a prime ideal $\pp$, then 
\begin{equation}\label{typic-(co)loc}
\gamma_{U(\pp)}\cong \RHom_R(R_\pp, -)\hspace{8pt}\text{and}\hspace{8pt}\lambda^{U(\pp)}\cong -\otimes_RR_\pp;
\end{equation}
see \cite[p.~2584]{NY18a} and \cite[(2.8)]{NY18b} for example.

For an $R$-module $M$, define $\Supp M:=\{\pp\in \Spec R \mid M_\pp\neq 0\}$, which we refer to as the \emph{classical support} of $M$.
Let $V$ be a specialization closed subset $V$, and denote by $\Gamma_V: \Mod R\to \Mod R$ the functor that assigns to each $R$-module $M$ the submodule $\Gamma_V M:=\{x\in M \mid \Supp Rx \subseteq V\}$ (\cite[Chapter IV, \S 1, Variation 1]{Har66}).
It is well known that
\begin{equation}\label{(co)smash-iso} \gamma_V\cong\RGamma_V\hspace{8pt}\text{and}\hspace{8pt}\lambda^V\cong \RHom_{R}(\RGamma_VR,-);\end{equation}
see \cite[Proposition 3.5.7]{Lip02} and \cite[Proposition 8.3]{BIK12b}. Note that $(\RGamma_VR)\LotimesR X\cong \RGamma_VX$ for every $X\in \D(R)$ (\cite[Proposition 3.5.5(ii)]{Lip02}), and so 
$\lambda^V:\D(R)\to \D(R)$ is a right adjoint to $\gamma_V:\D(R)\to \D(R)$, by \cref{(co)smash-iso} and tensor-hom adjunction.

If we take the Zariski closed subset $V(I):= \{\pp \in \Spec R \mid I \subseteq \pp\}$ for an ideal $I$, then 
\begin{equation}\label{closed-iso}
\gamma_{V(I)}\cong \RGamma_I\hspace{8pt}\text{and}\hspace{8pt}\lambda^{V(I)}\cong \LLambda^I,
\end{equation}
where $\Gamma_{I}$ is the \emph{$I$-torsion functor} $\varinjlim_{n\geq 1}\Hom_R(R/I^n,-)$ and $\Lambda^I$ is the \emph{$I$-adic completion functor} $\varprojlim_{n\geq 1} (-\otimes_R R/I^n)$.
The first isomorphism of \cref{closed-iso} follows from the isomorphism $\Gamma_{I}\cong\Gamma_{V(I)}$ of functors $\Mod R\to \Mod R$; see \cite[Chapter V, Corollary 4.2]{Har66} or \cite[Chapter II, Exercise 5.6]{Har77}. The second isomorphism of \cref{closed-iso} follows from \cref{(co)smash-iso} and the isomorphism
\begin{equation}\label{GM-dual}
\LLambda^{I}\cong \RHom_R(\RGamma_{I} R,- ),
\end{equation}
which is known as \emph{Greenlees--May duality} (\cite{GM92}); see \cite[\S 4]{Lip02} or \cite[Corollary 9.2.4]{SS18}.
In particular, this fact ensures that $\LLambda^{I}:\D(R)\to \D(R)$ is a right adjoint to $\RGamma_{I}:\D(R)\to \D(R)$.
Moreover,
there are isomorphisms 
\begin{align}
\RGamma_{I}\LLambda^I\cong \RGamma_I~\text{ and }~
\LLambda^I\RGamma_{I}\cong \LLambda^I;
\label{ATJLL97}
\end{align}
see \cite[Theorem 9.1.3]{SS18} (cf. \cite[Corollary 5.1.1]{ATJLL97}). More generally, we have $\gamma_V\lambda^V \cong \gamma_V$ and $\lambda^V\gamma_V \cong \lambda^V$ for a specialization closed subset $V$ by a formal argument using \cref{stable-approx} and \cref{bi-loc}.

Let $I$ be an ideal of $R$ and $\boldsymbol{x}=x_1,\ldots,x_t$ be a system of generators of $I$. 
For each $x_i$, consider the complex $(R \to R_{x_i})$ concentrated in degrees $0$ and $1$, where the map $R \to R_{x_i}$ is the localization $R\to S^{-1}R$ with respect to the multiplicatively closed subset $S$ generated by $x_i$.
The \emph{(extended) \v{C}ech complex} with respect to $\boldsymbol{x}$ is the complex $\check{C}(\boldsymbol{x}):=\bigotimes_{i=1}^n (R \to R_{x_i})$. 
For every $X\in \D(R)$, there is a natural isomorphism 
\begin{align}\label{Cech-comp}
\RGamma_{I}X\cong \check{C}(\boldsymbol{x}) \otimes_R X
\end{align}
in $\D(R)$; see \cite[\S 3.1]{Lip02} for example.

\begin{notation}\label{Cech-notation}
For each $\pp\in \Spec R$, we fix once and for all a system of generators $\boldsymbol{x}=x_1,\ldots,x_t$ of $\pp$, and write $\check{C}(\pp):=\check{C}(\boldsymbol{x})$.  
\end{notation}


\subsection{Depth and width}\label{depth-width}
Let $R$ be a commutative noetherian ring and let $X$ be a complex of $R$-modules. The \emph{infimum} and \emph{supremum} of $X$ are defined as 
\[
\inf X:= \inf\{n\in \ZZ \mid H^n(X)\neq 0\} \text{ and } \sup X:= \sup\{n\in \ZZ \mid H^n(X)\neq 0\},\]
respectively, where $\inf X=\infty$ and $\sup X=-\infty$ if $X$ is acyclic (i.e., $H^n(X) =0$ for all $i\in \ZZ$).
Let $I$ be an ideal of $R$, and take a system of generators $\boldsymbol{x}=x_1,\ldots,x_t$ of $I$.
For each $x_i$, consider the complex $(R \xrightarrow{\cdot x_i} R)$ concentrated in degrees $-1$ and $0$.
The \emph{Koszul complex} with respect to $\boldsymbol{x}$ is the complex $K(\boldsymbol{x}):=\bigotimes_{i=1}^n (R \xrightarrow{\cdot x_i} R)$.
By \cite[Theorems 2.1 and 4.1]{FI03}, the following hold:
\begin{align}
&\inf \RGamma_{I}X=\inf\RHom_R(R/I,X)=\inf\Hom_R(K(\boldsymbol{x}),X)\text{ and}\label{FI1}\\ 
&\sup \LLambda^{I}X=\sup ((R/I)\LotimesR X)=\sup (K(\boldsymbol{x})\otimes_R X).\label{FI2}
\end{align}
The \emph{$I$-depth} and \emph{$I$-width} of $X$ are defined as 
\[\depth_R(I,X):=\inf \RGamma_{I}X \text{ and } \width_R(I,X):=-\sup \LLambda^{I}X,\]
respectively.
If $R$ is a local ring with maximal ideal $\mm$, the \emph{depth} and \emph{width} of $X$ are defined as 
\[\depth_R X:=\depth_R(\mm, X) \text{ and } \width_R X:=\width_R(\mm, X),\]
respectively.

For a specialization closed subset $W$ of $\Spec R$, the \emph{$W$-depth} of $X$ is defined as 
\[\depth_R(W,X):=\inf\RGamma_WX;\]
see \cite[Chapter IV, \S 2]{Har66}. Given an ideal $I$ of $R$ with $V(I)\subseteq W$, we have $\RGamma_{V(I)}\RGamma_{W}\cong \RGamma_{V(I)}$; see \cref{trivial-iso} or \cref{Gamma-Inj}. Hence $\inf\RGamma_{W}X\leq \inf\RGamma_IX=\depth_R(I,X)$. Conversely, if there is an integer $n$ such that $n\leq \inf\RGamma_{V(I)}X$ for every ideal $I$ with $V(I)\subseteq W$, then $n\leq \inf\RGamma_WX$ since for each complex Y, $\Gamma_WY$ can be written as $\varinjlim_{V(I)\subseteq W}\Gamma_IY$ in the category $\C(R)$ of complexes of $R$-modules, where $I$ runs through all ideals of $R$ with $V(I)\subseteq W$.
Thus it holds that 
\begin{align}\label{W-depth}
\depth_R(W,X)=\inf \{\depth_R(I,X) \mid V(I)\subseteq W\}.
\end{align}
By using \cref{W-depth} and \cite[Propositions 2.10 and 2.11]{FI03}, we can easily deduce that
\begin{align}\label{depth-loc-glo}
\depth_R(W,X)=\inf \{\depth_{R_\pp}X_\pp \mid \pp \in W\}&=\inf \{\depth_R (\pp, X) \mid \pp \in W\}.
\end{align}

A similar phenomenon to the above can not be expected for width and a specialization closed subset $W$. Nevertheless, we define the \emph{$W$-width} of a complex $X\in \D(R)$ as 
\[\width_R(W,X):= -\sup\lambda^WX,\]
where $\lambda^WX\cong \RHom_R(\RGamma_WR,X)$ in $\D(R)$; see \cref{(co)smash-iso}.
If $W=V(J)$ for some ideal $J$, then $\width_R(W,X)=\width_R(J,X)$ by \cref{closed-iso}.
In general, there is an inequality
\[\sup \LLambda^I X \leq \sup\lambda^WX\]
for every ideal $I$ with $V(I)\subseteq W$ by \cref{trivial-iso}, and hence
\begin{equation*}
\width(W,X)\leq \inf \{\width_R(I,X) \mid V(I)\subseteq W\}.
\end{equation*}
As described in \cref{width-remark} below, there are a specialization closed subset $W$ and a complex $X$ such that 
$\sup X< \sup \lambda^WX$, while $\sup \LLambda^I X \leq \sup X$ for every ideal $I$ with $V(I)\subseteq W$, so it can happen that
\[\width_R(W,X)<-\sup X\leq \inf \{\width_R(I,X) \mid V(I)\subseteq W\}.\]

\begin{remark}\label{width-remark}
Recall that there is a (commutative noetherian) integral domain $R$ of finite global dimension such that the projective dimension  of $R_{(0)}$ over $R$ is greater than one; see \cite[Theorem 2]{Kap66} or \cite[\S 6]{Oso68}. Letting $W:=\Spec R\sm\{(0)\}$, we obtain a triangle
\[\RHom_R(R_{(0)}, X) \to X \to \lambda^{W}X \to \Sigma \RHom_R(R_{(0)}, X) \]
in $\D(R)$, where $\gamma_{W^\cp}X\cong \RHom_R(R_{(0)}, X)$; see \cref{typic-(co)loc,stable-approx}.
Since the projective dimension of $R_{(0)}$ is greater than one, we can choose $X$ as an $R$-module such that $\sup \RHom_R(R_{(0)}, X)>1$, and then $\sup \lambda^WX>0$ by the above triangle.
Thus it happens that $0=\sup X<\sup \lambda^WX$.
\end{remark}

\subsection{Sp-filtrations}\label{sp-filt}
Let $R$ be a commutative noetherian ring.
An \emph{sp-filtration} of $\Spec R$ is a map $\Phi: \ZZ \rightarrow 2^{\Spec R}$ such that $\Phi(n)$ is a specialization closed subset of $\Spec R$ and $\Phi(n) \supseteq \Phi(n+1)$ for every $n \in \ZZ$. An sp-filtration $\Phi$ is
 \emph{non-degenerate} if $\bigcup_{n \in \ZZ}\Phi(n) = \Spec R$ and $\bigcap_{n \in \ZZ}\Phi(n) = \emptyset$, where the former condition is equivalent to $\Phi(n) = \Spec R$ for some $n \in \ZZ$, since $R$ has only finitely many minimal prime ideals.
 
\begin{remark}\label{order-preserv}
To each sp-filtration $\Phi$, assign a function $\f_\Phi: \Spec R \rightarrow \ZZ \cup \{-\infty,\infty\}$ given by 
\[\f_{\Phi}(\pp) := \sup \{n\in \ZZ \mid \pp \in \Phi(n)\}+1\]
for each $\pp\in \Spec R$.
By definition, $\f_\Phi$ is order-preserving, that is, $\pp \subseteq \qq$ implies $\f_\Phi(\pp) \leq \f_\Phi(\qq)$.
Then the assignment $\Phi \mapsto \f_\Phi$ yields
a bijection from the sp-filtrations of $\Spec R$ to the order-preserving functions $\Spec R \rightarrow \ZZ \cup \{-\infty,\infty\}$.\footnote{The authors naturally reached this fact in order to formulate \cref{intro-slice}, while Ryo Takahashi had also noticed this bijection (in relation with \cite{DT15}) and he recently reported it as well; see \cite[Proposition 4.3]{Tak22}. He kindly suggested that the authors more emphasize the bijection, and his suggestion actually clarified their work; e.g., \cref{slice-strict} was added after that.}
The inverse map is described as follows.
To each order-preserving function $\f: \Spec R \rightarrow \ZZ \cup \{-\infty,\infty\}$, assign an sp-filtration $\Phi_{\f}: \ZZ \to \Spec R$ given by
\[\Phi_{\f}(n) := \{ \pp \in \Spec R \mid \f(\pp) >n\}\]
for each $n\in \ZZ$.

Notice that non-degeneracy of an sp-filtration $\Phi$ is equivalent to that the function $\f_\Phi: \Spec R \rightarrow \ZZ \cup \{-\infty,\infty\}$
corestricts to a function $\Spec R \rightarrow \ZZ$.
\end{remark}

By \cite[Theorems 3.10 and 3.11]{ATJLS10}, there is a bijection
\begin{equation}\label{ATJLS-bijec}
	\begin{Bmatrix}
	\text{sp-filtrations}\\
	\text{of $\Spec R$} 
	\end{Bmatrix} 
\isoto
	\begin{Bmatrix}
	\text{compactly generated}\\
	\text{t-structures in $\D(R)$} 
	\end{Bmatrix},
\end{equation}
given by $\Phi\mapsto (\cU_\Phi,\cV_\Phi),$
	where
	\begin{align*}
	\cU_{\Phi} &:= \{X \in \D(R) \mid \Supp H^n(X) \subseteq \Phi(n)~\forall n \in \ZZ\} \text{ and }\\
	 \cV_{\Phi} &:= \{X \in \D(R) \mid \depth_{R}(\Phi(n),X)>n~\forall n \in \ZZ\}.
	 \end{align*}
	 
To describe a set of compact objects generating each t-structure $(\cU_\Phi, \cV_\Phi)$, we use the following notation.
\begin{notation}
For each $\pp\in \Spec R$, we fix once and for all a system of generators $\boldsymbol{x}=x_1,\ldots,x_t$ of $\pp$ (as in \cref{Cech-notation}), and write $K(\pp):=K(\boldsymbol{x})$.  
\end{notation}

Let $\Phi$ be an sp-filtration of $\Spec R$, and let 
\[\cS_{\Phi}:=\{\Sigma^{-n}K(\pp)\mid n\in \ZZ, \pp\in \Phi(n) \}.\]
Using \cref{FI1,depth-loc-glo}, we can show that 
\begin{equation}\label{compact-gen-t}
\cS_{\Phi}^{\perp_0}=\cV_\Phi.
\end{equation}
In particular, the t-structure $(\cU_\Phi, \cV_\Phi)$ is (compactly) generated by $\cS_{\Phi}$. \cref{compact-gen-t} is stated in \cite[Proposition 4.14]{SP16} (cf. \cite[Corollary 3.9 and Theorem 3.10]{ATJLS10}).

We remark that an sp-filtration $\Phi$ is non-degenerate if and only if the t-structure $(\cU_\Phi, \cV_\Phi)$ is non-degenerate; see the proof of \cite[Theorem 3.8]{AHH21}.
Hence, combining \cref{cosilt-cof} and \cref{ATJLS-bijec}, we obtain the bijection
\begin{equation}\label{cosilt-cof-sp}
\begin{Bmatrix}
	\text{non-degenerate}\\
	\text{sp-filtrations}\\
	\text{of $\Spec R$} 
	\end{Bmatrix}
\isoto
	\begin{Bmatrix}
	\text{cosilting t-structures}\\
	\text{of cofinite type}\\
	\text{in $\D(R)$}.  
\end{Bmatrix}
\end{equation}
given by $\Phi\mapsto (\cU_\Phi, \cV_\Phi)$.

We can also classify silting t-structures of finite type by using non-degenerate sp-filtrations as in \cite[Theorem 3.8]{AHH21}.
For the reader's sake, we here give a more direct proof of this fact, describing the classification explicitly. Let $\Phi$ be an sp-filtration. Applying $(-)^*:=\Hom_R(-,R)$ to each object in $\cS_{\Phi}$, we obtain 
the set of compact objects
\[\cS_{\Phi}^*:=\{\Sigma^{n}(K(\pp)^*)\mid n\in \ZZ, \pp\in \Phi(n) \}.\]
By \cite[Theorem 4.15]{SP16}, the assignment
$\Phi\mapsto (\Perp{0}({\cS_{\Phi}^*}\Perp{0}), {\cS_{\Phi}^*}\Perp{0})$ yields a bijection from the sp-filtrations to the compactly generated \emph{co-t-structures} in $\D(R)$.\footnote{We here follow the definition of co-t-structures in \cite[\S 2.5]{AHH21}.
If one follows \cite[Definition 4.2]{SP16}, the above co-t-structure should be just written as $(\Perp{0}({\cS_{\Phi}^*}\Perp{0}), \Sigma^{-1} ({\cS_{\Phi}^*}\Perp{0})$).
}
Further, there is a t-structure in $\D(R)$ of the form $({\cS_{\Phi}^*}\Perp{0},({\cS_{\Phi}^*}\Perp{0})\Perp{0})$ by \cite[Theorem 3.11]{SP16}; see also \cite[Theorem 3.1 and Lemma 3.2]{AHH21} and the two paragraphs after \cite[Example 2.9]{AHH21}. 
Then we see that the assignment $\Phi\mapsto ({\cS_{\Phi}^*}\Perp{0},({\cS_{\Phi}^*}\Perp{0})\Perp{0})$ yields an injective map from the sp-filtrations to the t-structures in $\D(R)$.
The t-structure $({\cS_{\Phi}^*}\Perp{0},({\cS_{\Phi}^*}\Perp{0})\Perp{0})$ is non-degenerate if and only if it is induced by a silting object, which is of finite type by definition; see \cite[Theorem 4.11]{AH19}.
Moreover, non-degeneracy of the t-structure $({\cS_{\Phi}^*}\Perp{0},({\cS_{\Phi}^*}\Perp{0})\Perp{0})$ is equivalent to non-degeneracy of $\Phi$; this fact follows from \cite[Theorem 3.1 and Lemma 3.2]{AHH21} along with \cref{ATJLS-bijec} and \cref{cosilt-cof-sp}.
Therefore, the map given by $\Phi\mapsto ({\cS_{\Phi}^*}\Perp{0},({\cS_{\Phi}^*}\Perp{0})\Perp{0})$ restricts to an injective map from 
the non-degenerate sp-filtrations to the silting t-structures of finite type. On the other hand, given a silting t-structure $(\cY, \cW)$ of finite type, there is a compactly generated co-t-structure of the form $(\Perp{0}\cY, \cY)$ by \cite[Theorem 4.3]{AI12}; see also the third paragraph of \cite[p.~688]{AHH21}. Then \cite[Theorem 4.15]{SP16} implies that $(\Perp{0}\cY, \cY)=(\Perp{0}({\cS_{\Phi}^*}\Perp{0}), {\cS_{\Phi}^*}\Perp{0})$ for some sp-filtration $\Phi$, so that $(\cY,\cW)=({\cS_{\Phi}^*}\Perp{0},({\cS_{\Phi}^*}\Perp{0})\Perp{0})$.
Now, we define 
\[
\cY_\Phi:=\{X \in \D(R) \mid \width_{R}(\pp,X)>n~\forall n \in \ZZ~\forall \pp\in \Phi(n)\},
\]
and remark that 
\begin{equation}\label{comp-gen-co-t}
\cS_{\Phi}^*\Perp{0}=\cY_\Phi
\end{equation}
for each sp-filtration $\Phi$; see \cref{FI2} and the proof of \cite[Theorem 4.15]{SP16}.
Then it follows from the above argument that there is a bijection  
\begin{equation}\label{silt-fin-sp}
	\begin{Bmatrix}
	\text{non-degenerate}\\
	\text{sp-filtrations}\\
	\text{of  $\Spec R$} 
	\end{Bmatrix}
\isoto
	\begin{Bmatrix}
	\text{silting t-structures}\\
	\text{of finite type}\\
	\text{in $\D(R)$}
	\end{Bmatrix}
\end{equation}
given by $\Phi\mapsto (\cY_\Phi,\cY_\Phi^{\perp_0})$.

Given a non-degenerate sp-filtration $\Phi$, the t-structures $(\cU_\Phi, \cV_\Phi)$ and $(\cY_\Phi,\cY_\Phi^{\perp_0})$ are induced by some cosilting object $C$ and some silting object $T$, respectively, by \cref{cosilt-cof-sp} and \cref{silt-fin-sp}, but the classifications do not concretely tell us what $C$ and $T$ are in general.
One of our main purposes is to explicitly realize such a cosilting object and a silting object under some assumption on $\Phi$.

We will first deal with the silting side in \cref{slice-section,proof-theorem}. When we next deal with the cosilting side in \cref{cosilting-section}, it will be important to know that the injection mentioned in \cref{charact-dual} becomes bijective for any commutative noetherian ring $R$. That is, there is a bijection
\begin{equation}\label{(co)silt-dual}
	\begin{Bmatrix}
	\text{equivalence classes of}\\
	\text{silting objects of finite type}\\
	\text{in $\D(R)$}
	\end{Bmatrix}
\isoto
	\begin{Bmatrix}
	\text{equivalence classes of}\\
	\text{cosilting objects of cofinite type}\\
	\text{in $\D(R)$}
	\end{Bmatrix}
\end{equation}
given by $T\mapsto  T^+$, where $(-)^+:=\RHom_R(-,E)$ and $E$ is any injective cogenerator $E$ in $\Mod R$; see \cite[Theorem 3.8]{AHH21} and the last paragraph of \cref{subsec-real-func}. This bijection is compatible with \cref{cosilt-cof-sp} and \cref{silt-fin-sp} in the following sense: 
If $T$ is a silting object with $(T\Perp{>0}, T\Perp{\leq 0})=(\cY_\Phi, \cY\Perp{0}_\Phi)$ for a non-degenerate sp-filtration $\Phi$, then $T^{+}$ is a cosilting object with $(\Perp{\leq 0}(T^+), \Perp{>0}(T^+))=(\cU_\Phi, \cV_\Phi)$; see \cite[Theorems 3.2 and 3.3]{AHH21}, \cref{compact-gen-t}, and \cref{comp-gen-co-t}.

\begin{remark}\label{rem-0-tilt} 
Let $S$, $R$, and $E$ be as in \cref{charact-dual}.
Then $R$ is a tilting object of finite type in $\D(R)$ and it induces the standard t-structure $(\D^{\leq 0}(R), \D^{>0}(R))$, while $R^{+}:=\RHom_S(R,E)$ is a cotilting object of cofinite type in $\D(R^\op)$ and it induces the (shifted) standard t-structure $(\D^{\leq -1}(R^{\op}), \D^{> -1}(R^\op))$.

If $R$ is a commutative noetherian ring and $S=R$, these t-structures can be described by a non-degenerate sp-filtration, as explained above. 
Indeed, take the sp-filtration $\Phi$ defined by $\Phi(n):= \Spec R$ for $n\leq -1$ and $\Phi(n):= \emptyset$ for $n >-1$. Then
 $(\D^{\leq -1}(R), \D^{> -1}(R))=(\cU_\Phi, \cV_\Phi)$ clearly, so $(\D^{\leq 0}(R), \D^{>0}(R))=(\cY_\Phi, \cY\Perp{0}_\Phi)$ by  
 \cref{(co)silt-dual}. 
\end{remark}

\subsection{Dualizing complexes}\label{dc}
Let $R$ be a commutative noetherian ring. We state here various facts about dualizing complexes. We allow them to have infinite injective dimension, following Neeman's approach \cite[Definition 3.1]{Nee10}. 

\subsubsection{}
We say that a complex $X$ of $R$-module is \emph{cohomologically bounded} if $H^i(X)=0$ for all $i\ll 0$ and $i\gg 0$.
We denote by $\D^\bd_{\fg}(R)$ the subcategory of $\D(R)$ formed by cohomologically bounded complexes $X$ with finitely generated cohomology modules. 

A complex $D\in \D^\bd_\fg (R)$ is called a \emph{dualizing complex} for $R$ if the functor $\RHom_R(-,D):\D(R)^\op\to \D(R)$ induces a duality $\D^\bd_\fg(R)^\op \isoto\D^\bd_\fg(R)$, or equivalently, if for every $X\in \D^\bd_\fg (R)$, we have $\RHom_R(X,D)\in \D^\bd_\fg (R)$ and the canonical morphism $X\to \RHom_R(\RHom_R(X,D),D)$ is an isomorphism; see \cite[Proposition 3.6]{Nee10}.

\begin{remark}\label{strong-pw}
Recall that a complex $D$ of $R$-modules is called a \emph{pointwise} dualizing complex for $R$ if $D_\pp=R_\pp \otimes_R D$ is a dualizing complex for $R_\pp$ at each point $\pp\in \Spec R$.
A complex $D$ of $R$-modules is a dualizing complex in our sense if and only if $D$ is a pointwise dualizing complex and cohomologically bounded.
This characterization is due to Gabber (see \cite[Lemma 3.1.5]{Con00}) and he called a complex satisfying the latter conditions a \emph{strongly} pointwise dualizing complex, as stated in \cite[Footnote 1]{Nee10}. See also \cite[Theorem 6.2.2]{AIL10}, where the above characterization is recovered.

We say that a dualizing complex is \emph{classical} if it is isomorphic to a bounded complex of injective modules. 
A complex of $R$-modules is a classical dualizing complex if and only if it is a pointwise dualizing complex and $R$ has finite Krull dimension (\cite[Chapter V, Corollary 7.2 and Proposition 8.2]{Har66}). In the literature, dualizing complexes usually refer to classical ones.
\end{remark}

If $D$ is a pointwise dualizing complex, then, for each $\pp\in \Spec R$, there is a unique integer $\cd_D(\pp)$ such that $\Ext^{\cd_D(\pp)}_{R_\pp}(\kappa(\pp), D_\pp)\neq 0$ (\cite[pp.~282 and 287]{Har66}).
Moreover, it holds that \begin{equation}\label{dc-loc-coh}
\RGamma_\pp D_\pp  \cong\Sigma^{-\cd_D(\pp)}E_R(R/\pp)
\end{equation}
for each $\pp$ (\cite[Chapter V, Proposition 6.1]{Har66}), where $E_R(R/\pp)$ stands for the injective hull of $R/\pp$ over $R$.

If $D$ and $D'$ are dualizing complexes for $R$, then there is an \emph{invertible} complex $L\in \D(R)$ (in the sense of \cite[\S 5]{AIL10}) such that $D \cong D' \otimes_R^\mathbf{L} L$ in $\D(R)$; see \cite[Lemma 3.9]{Nee10} and \cite[Proposition 5.1]{AIL10} (cf. \cite[Chapter V, Theorem 3.1]{Har66}). In particular, if $R$ is local and admits a dualizing complex $D$, then $D$ is uniquely determined in $\D(R)$ up to shift and isomorphism.

\begin{notation}\label{notation-dc}
If $R$ is a local ring having a dualizing complex, we denote by $D_R$ a dualizing complex for $R$ such that $\inf D_R=0$.
\end{notation}

\begin{remark}\label{dc-facts}
Suppose that $R$ admits a dualizing complex $D$. 
\begin{enumerate}[label=(\arabic*), font=\normalfont, leftmargin=*]
	\item \label{localize-dc} If $S$ is a multiplicatively closed subset of $R$, then $S^{-1}R\otimes_R D$ is a dualizing complex for $S^{-1}R$.
	\item \label{complete-dc} 
	If $R$ is local and $\mm$ is the maximal ideal of $R$, then $\widehat{R}\LotimesR D_R\cong \widehat{R} \otimes_R D_R\cong D_{\widehat{R}}$ in $\D(R)$ and in $\D(\widehat{R})$ (see \cite[Chapter V, Corollary 3.5]{Har66} and \cite[Theorem 8.14]{Mat89})), where $\widehat{R}:=\Lambda^\mm R$.
	\item \label{image-dc} If $R\to A$ is a homomorphism of commutative noetherian rings such that $A$ is finitely generated as an $R$-module, then $\RHom_R(A, D)\in \D^{\bd}_{\fg}(A)$ is a dualizing complex for $A$; see \cref{strong-pw} and \cite[Corollary 6.2.4]{AIL10}. If $R$ is a Gorenstein ring (i.e., $R_\mm$ has finite injective dimension over $R_\mm$ for every maximal ideal $\mm$), then $R$ itself is a dualizing complex for $R$ (by \cref{strong-pw} and \cite[Chapter V, Theorem 9.1]{Har66}), and hence, given any ideal $I$ of $R$, $\RHom_R(R/I, R)\in \D^{\bd}_{\fg}(R/I)$ is a dualizing complex for $R/I$.
\end{enumerate}
\end{remark}

If $(R,\mm)$ is a local ring, then $\widehat{R}$ is a homomorphic image of a regular local ring by Cohen's structure theorem (\cite[Theorem 29.4(ii)]{Mat89}), so $\widehat{R}$ admits a dualizing complex (\cite[p.~299]{Har66}).

\subsubsection{}\label{minimal}
Recall that a complex $X$ of $R$-modules is called \emph{minimal} if every homotopy equivalence $X\to X$ is an isomorphism in $\C(R)$ (\cite[p.~397]{AM02}).
A complex $X$ of $R$-modules is called \emph{K-projective}, (resp. \emph{K-injective}, \emph{K-flat}) if the functor $\Hom_R(X,-)$ (resp. $\Hom_R(-,X)$, $-\otimes_RX$) from $\C(R)$ to $\C(R)$ preserves acyclicity of complexes (\cite{Spa88}).
Suppose that $R$ admits a dualizing complex $D$. According to the proof of \cite[Chapter I, Lemma 4.6(3)]{Har66}, one can construct a K-injective resolution $D\to I$ such that $I$ is a minimal complex of injective modules with $I^n=0$ for $n<\inf D$; see also \cite[Appendix B]{Kra05}.

It essentially follows from \cref{dc-loc-coh} and the structure theorem of injective $R$-modules (\cite[Theorems 18.4 and 18.5]{Mat89}) that
\[I^n \cong \bigoplus_{\substack{\pp\in \Spec R\\ \cd_D(\pp)=n}} E_R(R/\pp)\]
for every $n\in \ZZ$. This well-known fact can be verified by noting that $R_{\pp}\otimes_R-$ and $\Gamma_{\pp}$ send every minimal complex of injective $R$-modules to a minimal complex of injective $R$-modules; see \cref{Gamma-Inj} and \cite[Lemma B.1]{Kra05}. If $R$ is local and $D=D_R$ (as in \cref{notation-dc}), then $\cd_D(\pp)=\dim R-\dim R/\pp$; see \cref{cd-func} below.

\subsubsection{}\label{cd-func}
A strictly increasing chain $\pp_0 \subsetneq \pp_1  \subsetneq \cdots \subsetneq \pp_n$ of prime ideals in $\Spec R$ is said to be \emph{saturated} if, for each $i$  with $0\leq i<n$, there is no prime ideal $\qq$ such that $\pp_i \subsetneq \qq\subsetneq  \pp_{i+1}$.
A function $\cd: \Spec R \rightarrow \ZZ$ is called a \emph{codimension function} on $\Spec R$ if it satisfies that $\cd(\pp_1)=\cd(\pp_0) + 1$ for every saturated chain $\pp_0 \subsetneq \pp_1$ of length $1$ in $\Spec R$ (\cite[p.~283]{Har66}). Clearly, this definition is equivalent to that $\cd(\pp_n)=\cd(\pp_0) + n$ for every saturated chain $\pp_0 \subsetneq \pp_1  \subsetneq \cdots \subsetneq \pp_n$ of arbitrary length $n$ in $\Spec R$.
One can also characterize a codimension function as a function $\cd: \Spec R \rightarrow \ZZ$ satisfying $\cd(\pp_2)-\cd(\pp_1)=\height(\pp_2/\pp_1)$ for any inclusion $\pp_1 \subsetneq \pp_2$ in $\Spec R$, where $\height(\pp_2/\pp_1)$ is the height of the prime ideal $\pp_2/\pp_1$ of $R/\pp_1$.

Existence of a codimension function on $\Spec R$ implies that $R$ is \emph{catenary} (\cite[p.~284]{Har66}), that is, whenever $\pp$ and $\qq$ are prime ideals with $\pp \subsetneq \qq$, every saturated chain starting from $\pp$ and ending at $\qq$ has the same length. When a codimension function $\cd$ on $\Spec R$ exists, $\cd+c$ is a codimension function on $\Spec R$ as well for any constant $c \in \ZZ$, and $\cd$ is uniquely determined up to constant if $\Spec R$ is connected (e.g., \cite[Definition 5.1]{Kaw02}). More precisely, given two codimension functions $\cd$ and $\cd'$ on $\Spec R$ and a connected component $S$ of $\Spec R$, $\cd-\cd'$ is constant on $S$.

If $R$ admits a pointwise dualizing complex $D$, the function given by $\pp\mapsto \cd_D(\pp)$ is a codimension function (see \cref{dc-loc-coh} and \cite[p.~287]{Har66}). 
There are cases where $R$ may not admit a (pointwise) dualizing complex but a codimension function exists. 
Indeed, if $R$ is Cohen--Macaulay (i.e., $\depth_{R_\mm} R_\mm=\dim R_\mm$ for every maximal ideal $\mm$), the height function $\height: \Spec R \rightarrow \ZZ$ is a codimension function by \cite[Theorem 17.4(i)]{Mat89}; see also the proof of \cite[Theorem 17.4(ii)]{Mat89}. Moreover, if $R$ is catenary and local, the assignment $\pp \mapsto \dim R - \dim R/\pp$ defines a codimension function. We also remark that, when $R$ is a 1-dimensional ring, the height function and the function given by $\pp\mapsto \dim R - \dim R/\pp$ are both codimension functions, but $R$ may not admit a dualizing complex; see \cref{1-dim-no-dc}.

\begin{example}
Let $k$ be a field.
\begin{enumerate}[label=(\arabic*), font=\normalfont]
\item Let $R:=k[\![x]\!][y]$ and $\pp:=(1-xy)$. Then $\height(\pp)=1$ and $\dim R/\pp=0$, so $\dim R-\dim R/\pp=2>\height(\pp)$.
Note that $R$ is a dualizing complex for $R$ and $\cd_R=\height$.
Since $\Spec R$ is connected, a codimension function on $\Spec R$ is unique up to a constant. 
Thus the function $\Spec R \to \ZZ$ given by $\qq \mapsto \dim R-\dim R/\qq$ is not a codimension function.
\item Let $R:=k[\![x,y,z]\!]/(xy, xz)$. Then $\dim R=2$, and $\pp:=(y,z)\subseteq R$ is a minimal prime ideal, while  $\dim R/\pp=1$. Hence $\dim R-\dim R/\pp=1>\height(\pp)=0$.
The ring $R$ is catenary and local, so the function $\Spec R \to \ZZ$ given by $\qq \mapsto \dim R-\dim R/\qq$ is a codimension function. Thus the height function $\height:\Spec R\to\ZZ$ is not a codimension function, since $\Spec R$ is connected.
\end{enumerate}
\end{example}

\subsubsection{}
The existence of a dualizing complex allows us to use a powerful method, by which we can prove \cref{intro-slice} and \cref{intro-tilt}\cref{intro-tilt-dc} under additional yet mild conditions in a very efficient way; see \cref{cdc-shortcut}. 
We denote by $\Proj R$ (resp. $\Inj R$, $\Flat R$) the category of projective (resp. injective, flat) $R$-modules.
For an additive subcategory $\cA$ of $\Mod R$, we denote by $\K(\cA)$ the homotopy category of complexes of modules in $\cA$. 
The \emph{pure derived category} $\D(\Flat R)$ of flat $R$-modules is defined as the Verdier quotient category of the homotopy category $\K(\Flat R)$ by the subcategory of pure acyclic complexes of flat modules (i.e., acyclic K-flat complexes of flat modules). If we regard $\Flat R$ as the exact category in a natural way, then $\D(\Flat R)$ is the derived category of $\Flat R$ in the sense of \cite[\S 10]{Buh10}.

Under the existence of a classical dualizing complex $D$ (being a bounded complex of injective modules), Iyengar and Krause \cite[Theorem 1]{IK06} showed that the triangulated functor $D\otimes_R-: \K(\Proj R)\to \K(\Inj R)$ is an equivalence, where its quasi-inverse is given by composing $\Hom_R(D,-):\K(\Inj R)\to \K(\Flat R)$ with a right adjoint $\K(\Flat R)\to \K(\Proj R)$ to the inclusion $\K(\Proj R)\ha \K(\Flat R)$.

On the other hand, Neeman \cite[Theorem 1.2 and Facts 2.14]{Nee08} proved that 
the canonical functor $\K(\Proj R)\to \D(\Flat R)$ is a triangulated equivalence (for any ring) and pointed out that $D\otimes_R-: \K(\Proj R)\to \K(\Inj R)$ factors through the quotient functor $\K(\Flat R)\to \D(\Flat R)$ (\cite[Facts 2.17]{Nee08}).

As a consequence, if $D$ is as above, then $D\otimes_R-: \D(\Flat R)\to \K(\Inj R)$ is a triangulated equivalence whose quasi-inverse is given by $\Hom_R(D,-):\K(\Inj R)\to \D(\Flat R)$.
In fact, Neeman \cite[Corollary 3.19]{Nee11} extended the triangulated equivalence $D\otimes_R-: \K(\Proj R)\to \K(\Inj R)$ to the case that $D$ is a (strongly pointwise) dualizing complex (see \cref{strong-pw}).

In conclusion, we have the following fact.

\begin{theorem}[\cite{IK06}, \cite{Nee08}, and \cite{Nee11}]
\label{NIK}
Let $R$ be a commutative noetherian ring.
\leavevmode
\begin{enumerate}[label=(\arabic*), font=\normalfont]
\item \label{Neeman}The canonical functor
\[\K(\Proj R)\to \D(\Flat R)\]
is a triangulated equivalence.
\item \label{Iyengar-Krause} Assume that $R$ admits a dualizing complex $D$ and that $D$ is a bounded below complex of injective $R$-modules. Then the functor
	\[D \otimes_R -: \D(\Flat R) \rightarrow \K(\Inj R)\]
is a triangulated equivalence whose quasi-inverse is given by 
	\[\Hom_R(D,-): \K(\Inj R) \rightarrow \D(\Flat R).\]
\end{enumerate}
\end{theorem}

\begin{remark}\label{Murfet}
There is a fully faithful functor $\D(R)\hookrightarrow \D(\Flat R)$, which is the composition of the canonical embedding $\D(R)\hookrightarrow \K(\Proj R)$ with the equivalence $\K(\Proj R)\isoto \D(\Flat R)$ by \cref{NIK}\cref{Neeman}.
Given any K-flat complex $X$ of flat $R$-modules and a K-projective complex $P\in \K(\Proj R)$ with a quasi-isomorphism $P\to X$, the functor $\D(R)\hookrightarrow \K(\Proj R)$ replaces $X$ by $P$, and the mapping cone of the quasi-isomorphism $P\to X$ is pure acyclic. Hence $P$ and $X$ are identified in $\D(\Flat R)$. Then we see that the image of $X$ by the functor $\D(R)\hookrightarrow \D(\Flat R)$ is $X$ up to isomorphism.

On the other hand, the canonical functor $\K(\Flat R)\to \D(R)$ naturally induces a triangulated functor $\D(\Flat R)\to \D(R)$, which yields the canonical map $\Hom_{\D(\Flat R)}(X,Y)\to \Hom_{\D(R)}(X,Y)$,
where $X$ and $Y$ are complexes of flat $R$-modules.
Suppose that $X$ and $Y$ are K-flat. Then this canonical map is bijective (as observed in the proof of \cite[Theorem 5.5]{Mur07} for schemes).
Indeed, by the previous paragraph, we have the canonical bijection $\Hom_{\D(R)}(X,Y)\isoto \Hom_{\D(\Flat R)}(X,Y)$, and, composing it with the map $\Hom_{\D(\Flat R)}(X,Y)\to \Hom_{\D(R)}(X,Y)$, we obtain the identity map on $\Hom_{\D(R)}(X,Y)$; the last fact can be easily verified because we may assume that $X$ and $Y$ are K-projective complexes of projective $R$-modules.
\end{remark}

\subsection{Lukas lemma for complexes}
We here recall the Lukas lemma, which makes first extension groups vanish under some conditions.
We use this technical lemma to handle arbitrary commutative noetherian rings of possibly infinite Krull dimension and their arbitrary sp-filtrations; cf. \cref{cdc-shortcut} and \cref{footnote-Kdim}.

Let $R$ be an arbitrary ring and $X\in \C(R)$, where $\C(R)$ is the category of complexes of right $R$-modules. Let $\cC$ be a subcategory of $\C(R)$. A \emph{$\cC$-cofiltration} of $X$ is an inverse system
$(X_\alpha, \pi_{\alpha, \beta} \mid \alpha\leq \beta \leq \delta)$ formed by an ordinal $\delta$ and epimorphisms $\pi_{\alpha, \beta}: X_{\beta}\to X_\alpha$ in $\C(R)$ indexed by $\alpha\leq \beta\leq \delta$ such that the following conditions are satisfied:
\begin{enumerate}
	\item[(i)] $X_\delta = X$,
	\item[(ii)] $\varprojlim_{\alpha<\beta}X_\alpha=X_\beta$ for each limit ordinal $\beta\leq \delta$, and
	\item[(iii)] $X_0\in \cC$ and $\Ker(\pi_{\alpha,\alpha+1}) \in \cC$ for every $\alpha <\delta$.
\end{enumerate}
See \cite[Definition 6.34(i)]{GT12}.

\begin{lemma}[Lukas lemma]\label{Lukas-lemma}
Let $X, Y\in \C(R)$ and 
\[\cC:=\{Z\in \C(R) \mid \Ext^1_{\C(R)}(X,Z)=0\}.\]
If $Y$ has a $\cC$-cofiltration, then $Y\in \cC$. 
\end{lemma}

\begin{proof}
See \cite[Lemma 6.37]{GT12}, which is proved for modules, but the proof also works for complexes.
\end{proof}

We close this section by giving an interpretation of this lemma from the view point of the derived category. What we need for this purpose is the equivalence $\Hom_{\D(R)}(X,Y)=0\Leftrightarrow \Ext^1_{\C(R)}(\Sigma P,Y)=0$ for $X,Y\in \C(R)$ and a K-projective resolution $P\to X$ such that $P$ consists of projective modules (cf. \cite[Corollary 3.3]{EJX96}). We state this classical fact in a better way, which is also well known to specialists.

Given $X,Y\in \C(R)$, there is a canonical morphism
$\Hom_{\C(R)}(X,Y)\to \Ext^1_{\C(R)}(\Sigma X,Y)$ of abelian groups, which assigns to each chain map $f: X\to Y$ the equivalence class of the canonical exact sequence $0\to Y\to \Cone(f)\to \Sigma X\to 0$. The exact sequence splits if and only if $f$ is null-homotopic (see \cite[Lemma 3.2]{EJX96}), and the image of the canonical morphism $\Hom_{\C(R)}(X,Y)\to \Ext^1_{\C(R)}(\Sigma X,Y)$ is the subgroup $\Ext^1_{\C(R), \text{dw}}(\Sigma X,Y)$ formed by the equivalence classes of degreewise split exact sequences. Thus, there is a canonical isomorphism $\Hom_{\K(R)}(X,Y)\isoto \Ext^1_{\C(R), \text{dw}}(\Sigma X,Y)$, where $\K(R)$ stands for the homotopy category of complexes of $R$-modules.
For a complex $P$ of projective $R$-modules, we have $\Ext^1_{\C(R), \text{dw}}(\Sigma P,Y)=\Ext_{\C(R)}^1(\Sigma P,Y)$, so if we further assume that $P$ is a K-projective resolution of $X$, then there are canonical isomorphisms
\begin{align}\label{Hom-Ext}
\Hom_{\D(R)}(X,Y)\cong \Hom_{\K(R)}(P,Y)\cong \Ext_{\C(R)}^1(\Sigma P,Y).
\end{align}

\begin{corollary}\label{cofiltration}
Let $X, Y \in \D(R)$, and consider the subcategory $X\Perp{0}$ of $\D(R)$. Let $\cC$ be the subcategory of $\C(R)$ formed by all objects in $X\Perp{0}$, that is, $\cC:=\{Z\in \C(R) \mid \Hom_{\D(R)}(X,Z)=0 \}$. If Y has a $\cC$-cofiltration, then $Y\in X\Perp{0}$. 
\end{corollary}
\begin{proof}
Let $P$ be a K-projective resolution of $X$ such that $P$ consists of projective $R$-modules. By definition, we have $\Hom_{\D(R)}(X,Z)\cong \Hom_{\K(R)}(P,Z)$ for all complexes $Z$, so $\cC=\{Z\in \C(R) \mid \Hom_{\K(R)}(P,Z)=0 \}$. It follows from \cref{Hom-Ext} that $\cC=\{Z\in \C(R) \mid \Ext^1_{\C(R)}(\Sigma P,Z)=0 \}$, and hence, if $Y$ has a $\cC$-cofiltration, then the Lukas lemma yields $Y\in \cC$, or equivalently, $Y\in X\Perp{0}$.
\end{proof}

\section{Slice sp-filtrations and codimension filtrations}\label{slice-section}
In the rest of this paper, $R$ denotes a commutative noetherian ring unless otherwise specified.
For a subset $W$ of $\Spec R$, $\dim W$ denotes the supremum of lengths of strict chains of prime ideals in $W$. In particular, $\dim (\Spec R)$ is the Krull dimension of $R$, which is denoted by $\dim R$ as before.
Note that if $W$ is empty, then $\dim W=-\infty$ by definition.

\begin{definition}\label{slice-sp}
Let $\Phi$ be an sp-filtration of $\Spec R$. We call $\Phi$ a \emph{slice sp-filtration} of $\Spec R$ if it is non-degenerate and $\dim(\Phi(n) \setminus \Phi(n+1)) \leq 0$ for all $n \in \ZZ$.
\end{definition}

Slice sp-filtrations are precisely the filtrations used in \cite[Chapter IV, \S 3]{Har66} for the definition of Cousin complexes.

\begin{remark}\label{slice-remark}
The authors of the present paper would like to adopt ``slice''\footnote{This terminology was initially brought from \cite[Definition 7.6]{NY18b}, which introduces a family of zero-dimensional subsets of $\Spec R$ satisfying some conditions for a given finite-dimensional subset $W$ to calculate the Bousfield localization functor $\lambda^W$ on $\D(R)$, and the authors of \cite{NY18b} called it a \emph{system of slices} of $W$. The conditions of this notion also imitate the definition of the filtrations in \cite[Chapter IV, \S 3]{Har66}.
}
(rather than ``Cousin'') to avoid a possible confusion with other notions of sp-filtrations; see \cref{w-s-Cousin} below.

After this choice, the authors noticed that the terminology ``slice filtration'' is used in the context of equivariant or motivic homotopy theory; see \cite{GRSO12}, \cite{Hea19}, \cite{Hil12}, and \cite{Ull13} for example. A basic idea behind slice filtrations is very similar to our motivation of \cref{slice-sp}; see \cref{slice-functor}. 
However, our definition is not precisely compatible with the slice filtrations in that area. We leave the prefix ``sp-'' so that the reader can distinguish ours from the other.

On the other hand, we remove the prefix from the filtrations considered in \cref{hight-codim} for simplicity because we think there is no worry of confusion.
\end{remark}

Recall that there is a bijective correspondence between the sp-filtrations of $\Spec R$ and the order-preserving functions $\Spec R \to \ZZ \cup\{\infty, -\infty\}$, and this restricts to a bijective correspondence between the non-degenerate sp-filtrations of $\Spec R$ and the order-preserving functions $\Spec R \to \ZZ$; see \cref{order-preserv}.
In fact, a non-degenerate sp-filtration $\Phi$ is a slice sp-filtration if and only if its corresponding function $\f_\Phi$
is strictly increasing, that is, $\pp \subsetneq \qq$ implies $\f_\Phi(\pp) < \f_\Phi(\qq)$. 
Hence we have the following commutative diagram:
\begin{equation}\label{slice-strict}
\begin{tikzcd}[every label/.append style = {font = \small}]
\begin{Bmatrix}
	\text{non-degenerate}\\
	\text{sp-filtrations}\\
	\text{of $\Spec R$} 
\end{Bmatrix}
	\ar[d,phantom,"\rotatebox{-90}{$\supseteq$}"]
	\ar[r,equal,"\sim"]&
\begin{Bmatrix}
	\text{order-preserving}\\
	\text{functions}\\
	\Spec R \to \ZZ 
\end{Bmatrix}
	\ar[d,phantom,"\rotatebox{-90}{$\supseteq$}"]\\
\begin{Bmatrix}
    \text{slice}\\
	\text{sp-filtrations}\\
	\text{of $\Spec R$} 
	\end{Bmatrix}
	\ar[r,equal,"\sim"]&
\begin{Bmatrix}
	\text{strictly increasing}\\
	\text{functions}\\
	\Spec R \to \ZZ 
\end{Bmatrix}\rlap{.}
\end{tikzcd}
\end{equation}

As observed in \cref{cosilt-cof-sp,silt-fin-sp}, the silting t-structures of finite type and the cosilting t-structures of cofinite type are classified by using the non-degenerate sp-filtrations.
We will explicitly construct a (co)silting object that induces  the (co)silting t-structure corresponding to each slice sp-filtration.

\begin{example}\label{hight-codim}
A typical example of a slice sp-filtration is $\Phi_{\height}$ defined by the height function $\height:\Spec R\to \ZZ$ as in \cref{order-preserv}. We call $\Phi_{\height}$ the \emph{height filtration} of $\Spec R$.
If $R$ admits a codimension function $\cd: \Spec R\to \ZZ$, then $\Phi_{\cd}$ is also a slice sp-filtration, which we call a \emph{codimension filtration} of $\Spec R$.

The grade function $\grade: \Spec R \to \ZZ$ is a natural example of an order-preserving function that may not be strictly increasing, where $\grade(\pp) := \depth(\pp,R)$. We call the sp-filtration $\Phi_{\grade}$ the \emph{grade filtration} of $\Spec R$. Note that $\grade \leq \height$ in general (\cite[Proposition 1.2.14]{BH98}), and the equality holds (or equivalently, $\grade$ is strictly increasing) if and only if $R$ is Cohen--Macaulay (\cite[Proposition 1.2.10(a)]{BH98}).
\end{example}

\begin{remark}\label{w-s-Cousin}
An sp-filtration $\Phi$ of $\Spec R$ is said to satisfy the \emph{weak Cousin condition} if the following condition holds:
For any $n\in \ZZ$ and any strict inclusion $\pp\subsetneq \qq$ in $\Spec R$ such that $\pp$ is maximal under $\qq$ (i.e., $\pp\subsetneq \qq$ is a saturated chain), $\qq \in \Phi(n)$ implies $\pp\in \Phi(n-1)$. See \cite[p.~332, Definition]{ATJLS10}.
If the converse implication also holds, then $\Phi$ is said to satisfy the \emph{strong Cousin condition}. See 
\cite[p.~331, Remark]{ATJLS10}.
 
Under the existence of a classical dualizing complex (resp. a pointwise dualizing complex), the weak Cousin condition characterizes the case when the t-structure corresponding to $\Phi$ via \cref{ATJLS-bijec} restricts to a t-structure in $\D^{\bd}_{\fg}(R)$ (resp. $\D_{\fg}(R)$); see \cite[Theorem 6.9 and Corollaries 3.12 and 6.10]{ATJLS10}. Here, $\D_{\fg}(R)$ stands for the subcategory of $\D(R)$ formed by complexes with finitely generated cohomology modules.

In general, a non-degenerate sp-filtration satisfying the weak Cousin condition need not be a slice sp-filtration. Conversely, a slice sp-filtration need not satisfy the weak Cousin condition.
In fact, a slice sp-filtration satisfies the weak Cousin condition if and only if it is a codimension filtration. One can also directly check that a non-degenerate sp-filtration is a codimension filtration if and only if it satisfies the strong Cousin condition. In conclusion:
\begin{equation*}
	\begin{Bmatrix}
		\text{non-degenerate sp-filtrations}\\
		\text{of $\Spec R$ satisfying}\\
		\text{the strong Cousin condition}
		\end{Bmatrix}
	=
	\begin{Bmatrix}
	\text{slice sp-filtrations}\\
	\text{of $\Spec R$ satisfying the}\\
	\text{weak Cousin condition}
	\end{Bmatrix}
=
	\begin{Bmatrix}
	\text{codimension filtrations}\\
	\text{of $\Spec R$}
	\end{Bmatrix}\rlap{.}
\end{equation*}
\end{remark}

\begin{definition}\label{Tphi}
Let $\Phi$ be a non-degenerate sp-filtration of $\Spec R$. We define 
\begin{equation*} 
T_\Phi := \bigoplus_{\pp \in \Spec R} \Sigma^{\f_\Phi(\pp)}\RGamma_{\pp} R_{\pp},
\end{equation*}
which is an object in $\D(R)$.
\end{definition}

\begin{remark}\label{RHom-TPhi}
Let $\pp\in\Spec R$.
It follows from \cref{GM-dual}, \cref{Cech-comp}, and tensor-hom adjunction that
there is an isomorphism
\[\RHom_R(\RGamma_{\pp} R_{\pp},-)\cong \LLambda^{\pp}\RHom_R(R_{\pp},-)\] of functors $\D(R)\to \D(R)$.
If $\mm$ is a maximal ideal, then  $\RGamma_{\mm} R_{\mm}\cong \RGamma_{\mm} R$ in $\D(R)$ (see \cref{Gamma-Inj} for example), and so $\LLambda^{\mm}\cong \RHom_R(\RGamma_{\mm} R,-)\cong\LLambda^{\mm}\RHom_R(R_{\mm},-)$ as functors $\D(R)\to \D(R)$.

For each $X\in \D(R)$, the inclusion $\supp_R\RGamma_{\pp} X_{\pp}\subseteq \{\pp\}$ holds by  \cref{Cech-comp}. In particular, we have $\supp_R\RGamma_{\pp} R_{\pp}= \{\pp\}$.
It then follows from the previous paragraph that we also have $\cosupp_{R}\LLambda^{\pp}\RHom_R(R_{\pp},X)\subseteq \{\pp\}$.
\end{remark}

We will show that $T_\Phi$ is a silting object inducing the t-structure $(\cY_\Phi, \cY\Perp{0}_\Phi)$ in $\D(R)$ for every slice sp-filtration $\Phi$ of $\Spec R$.
For this, we at least need to show that $T\Perp{>0}_\Phi=\cY_\Phi$, whence it follows that $(T\Perp{>0}_\Phi, (T\Perp{>0}_\Phi)\Perp{0})$ is a t-structure. However, it is still unclear if the equality $(T\Perp{>0}_\Phi)\Perp{0}=T\Perp{\leq 0}_\Phi$ holds, so we have to further show that $T_\Phi$ is silting in $\D(R)$ after all. 
By \cref{silting-cond}, this is equivalent to showing the following conditions: (1) $T_\Phi \in T\Perp{>0}_\Phi$; (2) $T\Perp{\ZZ}_\Phi = 0$; (3) $T\Perp{>0}_\Phi$ is closed under coproducts.
The second condition is not difficult to check thanks to a result of Neeman.

\begin{lemma}\label{generate}
Let $\Phi$ be a non-degenerate sp-filtration.
The object $T_\Phi$ generates $\D(R)$, that is, $T\Perp{\ZZ}_\Phi= 0$.
\end{lemma}
\begin{proof}
Since $\supp_R T_\Phi=\Spec R$ by \cref{RHom-TPhi}, it follows from \cite[Theorem 2.8]{Nee92} that the smallest localizing subcategory containing $T_\Phi$ is $\D(R)$, or equivalently, $T\Perp{\ZZ}_\Phi = 0$.
\end{proof}

Showing the remaining two conditions (for $T_\Phi$ with a slice sp-filtration $\Phi$) requires technical arguments. Since we have $\cS_{\Phi}^*\Perp{0}=\cY_\Phi$ by \cref{comp-gen-co-t}, the third condition follows once we verify the equality $T\Perp{>0}_\Phi=\cY_\Phi$, and this step is necessary in any case as mentioned above. To compare these two subcategories, let us here give useful descriptions of $T\Perp{>0}_\Phi$ and $\cY_\Phi$. For every object $X\in \D(R)$ and every integer $n$, we have natural isomorphisms
\begin{align*}
\Hom_{\D(R)}(T_\Phi, \Sigma^nX)&\cong \prod_{\pp\in\Spec R}\Hom_{\D(R)}(\Sigma^{\f_\Phi(\pp)}\RGamma_{\pp} R_{\pp}, \Sigma^nX)\\
&\cong \prod_{\pp\in\Spec R}H^0\RHom_R(\Sigma^{\f_\Phi(\pp)}\RGamma_{\pp} R_{\pp}, \Sigma^nX)\\
&\cong \prod_{\pp\in\Spec R}H^{n-\f_\Phi(\pp)} \LLambda^\pp \RHom_R(R_\pp, X)
\end{align*}
by \cref{RHom-TPhi}.
Hence $\Hom_{\D(R)}(T_\Phi, \Sigma^n X)=0$ for all $n>0$ if and only if it holds that $H^{n-\f_\Phi(\pp)} \LLambda^\pp \RHom_R(R_\pp, X)=0$ for all $\pp\in \Spec R$ and all $n>0$, or equivalently, we have $\sup \LLambda^\pp \RHom_R(R_\pp, X)\leq -\f_\Phi(\pp)$ for all $\pp\in \Spec R$.
Therefore
\begin{align}\label{TPhi-posit}
T\Perp{>0}_\Phi = \{X \in \D(R) \mid \width_{R_\pp} \RHom_R(R_\pp, X)\geq \f_\Phi(\pp)~\forall \pp \in \Spec R\}.
\end{align}
On the other hand, for any $X\in \D(R)$, the following bi-implications hold:
\begin{align*}
&\width_R(\pp, X)>n ~\forall n \in \ZZ~\forall \pp \in \Phi(n)\\
&\Leftrightarrow
\width_R(\pp, X)>\sup\{n\in \ZZ\mid \pp \in \Phi(n)\}~\forall \pp\in \Spec R\\
&\Leftrightarrow
\width_R(\pp, X)\geq \f_\Phi(\pp)~\forall \pp \in \Spec R.
\end{align*}
As a consequence,
\begin{equation}\label{cYPhi-posit}
	\cY_\Phi=\{X \in \D(R) \mid \width_R(\pp, X)\geq \f_\Phi(\pp)~\forall \pp \in \Spec R\},
\end{equation}
It will be shown in \cref{classes-equal} that $T\Perp{>0}_\Phi$ coincides with $\cY_\Phi$, but this is not easy. An essential difficulty comes from non-exactness of the colocalization functors $\Hom_R(R_\pp,-)$.

It is also a delicate problem to verify the first condition $T_\Phi \in T\Perp{>0}_\Phi$. We provide two approaches for this.
To get the full generality, we use the Lukas lemma (\cref{Lukas-lemma}) and combine it with various facts on Bousfield localization functors; see \cref{proof-theorem}.
If $R$ admits a dualizing complex, then \cref{NIK} enables us to efficiently prove the first condition as explained below.

Let $\Phi$ be a non-degenerate sp-filtration of $\Spec R$ and assume $R$ admits a dualizing complex $D$, which we may regard as a bounded below complex of injective $R$-modules. By \cref{Cech-comp}, and using \cref{Cech-notation}, we have 
\begin{align}\label{TPhi-check}
T_\Phi\cong \bigoplus_{\pp\in \Spec R} \Sigma^{\f_\Phi(\pp)} \check{C}(\pp)_\pp
\end{align} in $\D(R)$, where $\check{C}(\pp)_\pp=R_\pp\otimes_R\check{C}(\pp)$.
Moreover, it follows from \cref{Cech-comp,dc-loc-coh} that 
\begin{equation}\label{dc-loc-coh1}
D\otimes_R \Sigma^{\f_\Phi(\pp)}\check{C}(\pp)_\pp\cong \Sigma^{\f_\Phi(\pp)-\cd_{D}(\pp)} E_R(R/\pp)
\end{equation}
in $\D(R)$ for every $\pp\in \Spec R$. Since \cref{dc-loc-coh1} is an isomorphism between bounded below complexes of injective modules, it can be realized as an isomorphism in $\K(\Inj R)$. Collecting the isomorphisms in $\K(\Inj R)$ for all $\pp\in \Spec R$, and taking the coproduct of them, we obtain an isomorphism 
\begin{equation}\label{dc-loc-coh2}
\bigoplus_{\pp\in \Spec R} D \otimes_R \Sigma^{\f_\Phi(\pp)} \check{C}(\pp)_\pp\cong \bigoplus_{\pp\in \Spec R}\Sigma^{\f_\Phi(\pp)-\cd_{D}(\pp)} E_R(R/\pp)
\end{equation}
in $\K(\Inj R)$.

\begin{proposition}\label{P:dualizing}
		Let $R$ be a commutative noetherian ring with a dualizing complex and $\Phi$ a non-degenerate sp-filtration of $\Spec R$. Then the following hold.
	\begin{enumerate}[label=(\arabic*), font=\normalfont]
		\item \label{posit-ext-dc} $\Add(T_\Phi) \subseteq T\Perp{>0}_\Phi$ if and only if $\Phi$ is a slice sp-filtration.
		\item \label{negat-ext-dc} $\Add(T_\Phi)\subseteq T_\Phi\Perp{\neq 0}$ if and only if $\Phi$ is a codimension filtration.
	\end{enumerate}
\end{proposition}

\begin{proof}
Let $D$ be a dualizing complex for $R$, and regard $D$ as a bounded below complex of injective $R$-modules.
For the sake of brevity, we write simply $T$, $\f$ and $\cd$ instead of $T_\Phi$, $\f_\Phi$ and $\cd_{D}$, respectively. 
		
Let $\varkappa$ be a cardinal and $n$ be an integer. By definition, we have 
\[\Hom_{\D(R)}(T,\Sigma^n T^{(\varkappa)}) \cong \prod_{\pp \in \Spec R}\Hom_{\D(R)}(\Sigma^{\f(\pp)}\RGamma_{\pp}R_{\pp}, \Sigma^nT^{(\varkappa)}).\]
		 Fix $\pp \in \Spec R$. 
		 It follows from \cref{Cech-comp} and \cref{Murfet} that 
		 \begin{align*}
		 \Hom_{\D(R)}(\Sigma^{\f(\pp)}\RGamma_{\pp}R_{\pp},\Sigma^nT^{(\varkappa)})
		 &\cong \Hom_{\D(R)}\Big(\Sigma^{\f(\pp)}\check{C}(\pp)_{\pp},\bigoplus_{\qq \in  \Spec R}\Sigma^{n+\f(\qq)}\check{C}(\qq)_\qq^{(\varkappa)}\Big),\\
		 &\cong \Hom_{\D(\Flat R)}\Big(\Sigma^{\f(\pp)}\check{C}(\pp)_{\pp},\bigoplus_{\qq \in  \Spec R}\Sigma^{n+\f(\qq)}\check{C}(\qq)_\qq^{(\varkappa)}\Big).
		 \end{align*}
Furthermore, using \cref{NIK}\cref{Iyengar-Krause}, \cref{dc-loc-coh1}, and \cref{dc-loc-coh2}, we have
		\begin{align*}
		&\Hom_{\D(\Flat R)}\Big(\Sigma^{\f(\pp)}\check{C}(\pp)_{\pp},\bigoplus_{\qq \in \Spec R}\Sigma^{n+\f(\qq)}\check{C}(\qq)_\qq^{(\varkappa)}\Big)\\ 
		&\cong \Hom_{\K(\Inj R)}\Big(\Sigma^{\f(\pp)}D \otimes_R \check{C}(\pp)_{\pp},D \otimes_R \big(\bigoplus_{\qq \in \Spec R}\Sigma^{n+\f(\qq)}\check{C}(\qq)_\qq^{(\varkappa)}\big)\Big)\\
		&\cong  \Hom_{\K(\Inj R)}\Big(\Sigma^{\f(\pp)-\cd(\pp)}E_R(R/\pp),\bigoplus_{\qq \in \Spec R}\Sigma^{n+\f(\qq)-\cd(\qq)}E_R(R/\qq)^{(\varkappa)}\Big)\\
		&\cong  \Hom_{\K(\Inj R)}\Big(E_R(R/\pp),\bigoplus_{\qq \in \Spec R}\Sigma^{n+\f(\qq)-\f(\pp)-\cd(\qq)+\cd(\pp)}E_R(R/\qq)^{(\varkappa)}\Big)\\
		&\cong \Hom_{\K(\Inj R)}\Bigg(E_R(R/\pp),\bigoplus_{i\in \ZZ}\bigg(\bigoplus_{\substack{\qq \in \Spec R\\ n+\f(\qq)-\f(\pp)-\cd(\qq)+\cd(\pp)=i}}\Sigma^{i}E_R(R/\qq)^{(\varkappa)}\bigg)\Bigg)\\
		&\cong \Hom_{\K(\Inj R)}\Bigg(E_R(R/\pp),\bigoplus_{\substack{\qq \in \Spec R\\ n+\f(\qq)-\f(\pp)-\cd(\qq)+\cd(\pp)=0}}E_R(R/\qq)^{(\varkappa)}\Bigg)\\
		&\cong \Hom_{\K(\Inj R)}\Bigg(E_R(R/\pp),\bigoplus_{\substack{\qq \in V(\pp)\\ n+\f(\qq)-\f(\pp)-\cd(\qq)+\cd(\pp)=0}}E_{R}(R/\qq)^{(\varkappa)}\Bigg),
		\end{align*}
where the last isomorphism follows from \cref{comparison}\cref{Gamma-adjoint} below.

As a consequence, there is an isomorphism
\[\Hom_{\D(R)}(T,\Sigma^n T^{(\varkappa)})\cong \prod_{\pp \in \Spec R}\Hom_{\K(\Inj R)}\Bigg(E_R(R/\pp),\bigoplus_{\substack{\qq \in V(\pp)\\ n+\f(\qq)-\f(\pp)=\cd(\qq)-\cd(\pp)}}E_{R}(R/\qq)^{(\varkappa)}\Bigg).\]
Then we see from \cref{comparison}\cref{Gamma-adjoint} that
\begin{align*}
\text{$\Add(T)\subseteq T\Perp{n}$}\Leftrightarrow  \text{there is no $\pp\subseteq \qq$ in $\Spec R$ such that $n +\f(\qq)-\f(\pp)=\cd(\qq)-\cd(\pp)$}.
\end{align*}
Now let us complete the proof.

\cref{posit-ext-dc}: Suppose that $\Phi$ is a slice sp-filtration, or equivalently, $\f$ is strictly increasing; see \cref{slice-strict}. Take any inclusion $\pp\subseteq \qq$ in $\Spec R$.
Then $\f(\qq) - \f(\pp)\geq \cd(\qq) - \cd(\pp)$ by \cref{comparison}\cref{slice-vs-codim} below. Therefore, $n+\f(\qq) - \f(\pp)> \cd(\qq) - \cd(\pp)$ for every $n>0$. This implies $\Add(T) \subseteq T\Perp{>0}$ by the above equivalence.

Conversely, suppose that $\Add(T) \subseteq T\Perp{>0}$. If the order-preserving function $\f:\Spec R\to \ZZ$ is not strictly increasing, then there is a strict inclusion $\pp\subsetneq \qq$ in $\Spec R$ such that $\f(\qq)-\f(\pp)<\cd(\qq)-\cd(\pp)$ by \cref{comparison}\cref{slice-vs-codim}. So, setting $n:=-\f(\qq)+\f(\pp)+\cd(\qq)-\cd(\pp) >0$, we have  the equality $n+\f(\qq) - \f(\pp)= \cd(\qq) - \cd(\pp)$, and then 
$\Add(T) \not\subseteq T\Perp{n}$ by the above equivalence. This is a contradiction. Hence $\f$ must be strictly increasing.

		\cref{negat-ext-dc}: Suppose that $\Phi$ is a codimension filtration, or equivalently, $\f$ is a codimension function. 
Take an inclusion $\pp\subseteq \qq$ in $\Spec R$.
Then $\f({\qq})-\f({\pp})=\cd(\qq)-\cd(\pp)$ by \cref{comparison}\cref{slice-vs-codim}. Thus $\Add(T) \subseteq T\Perp{\neq 0}$ by the above equivalence. 

Conversely, suppose that $\Add(T) \subseteq T\Perp{\neq 0}$. 
By \cref{posit-ext-dc}, 
the function $\f:\Spec R\to \ZZ$ is strictly increasing, so 
$\f(\qq)-\f(\pp)\geq  \cd(\qq) -\cd(\pp)$ for every inclusion $\pp\subsetneq \qq$ in $\Spec R$ by \cref{comparison}\cref{slice-vs-codim}.
If the inequality $\f(\qq)-\f(\pp)> \cd(\qq) -\cd(\pp)$ occurs for some inclusion $\pp\subsetneq \qq$ in $\Spec R$, then 
setting $n:=-\f(\qq)+\f(\pp)+\cd(\qq)-\cd(\pp) <0$, we have the equality $n+\f(\qq) - \f(\pp)= \cd(\qq) - \cd(\pp)$. Hence $\Add(T) \not\subseteq T\Perp{n}$ by the above equivalence, but this can not happen as we supposed that $\Add(T) \subseteq T\Perp{\neq 0}$. Thus we have $\f(\qq)-\f(\pp)= \cd(\qq) -\cd(\pp)$  for every inclusion $\pp\subsetneq \qq$ in $\Spec R$, that is, $\f$ is a codimension function.
	\end{proof}

\begin{remark}\label{comparison}
We here summarize two elementary facts used in the above proof.

\begin{enumerate}[label=(\arabic*), font=\normalfont]
\item \label{Gamma-adjoint}
Let $\pp, \qq\in \Spec R$. Then $\Hom_R(E_R(R/\pp), E_R(R/\qq))\neq 0$ if and only if $\pp\subseteq \qq$. The ``if'' part follows because there is a nonzero $R$-linear map $E_R(R/\pp)\to E_R(R/\qq)$
that extends the composition of the canonical map $R/\pp\twoheadrightarrow R/\qq$ and the given map $R/\qq \ha E_R(R/\qq)$. The ``only if'' part follows because $E_R(R/\qq)$ is $\qq$-local (i.e., the canonical map $E_R(R/\qq)\to E_R(R/\qq)_{\qq}$ is bijective) and $E_R(R/\pp)$ is $\pp$-torsion (i.e., the canonical map $\Gamma_{\pp} E_R(R/\pp)\to E_R(R/\pp)$ is bijective); see \cite[Theorem 18.4(v) and (vi)]{Mat89}.
Indeed, this fact implies that there is no nonzero $R$-linear map from $E_R(R/\pp)$ to any coproduct of $\qq$-local $R$-modules for all $\qq\in \Spec R$ with $\qq\notin V(\pp)$.

\item \label{slice-vs-codim} Let $\pp, \qq\in \Spec R$ with $\pp\subsetneq \qq$. Since $R$ is noetherian, there is a saturated chain \[\pp = \qq_0 \subsetneq \qq_1 \subsetneq \qq_2 \subsetneq \cdots \subsetneq \qq_k = \qq\]
in $\Spec R$. Then the inequality $\f(\qq) - \f(\pp)\geq k$ holds for each strictly increasing function  $\f:\Spec R\to \ZZ$. Moreover, if a codimension function $\cd$ exists, then $\cd(\qq) - \cd(\pp)=k$ by definition. 
In fact, under the existence of a codimension function $\cd$, an order-preserving function $\f:\Spec R\to \ZZ$ is strictly increasing if and only if $\f(\qq) - \f(\pp)\geq \cd(\qq) - \cd(\pp)$ for every inclusion $\pp\subseteq \qq$ in $\Spec R$. Note also that we simply have $\f(\qq) - \f(\pp)= \cd(\qq) - \cd(\pp)$ when $\pp= \qq$.
\end{enumerate}

\end{remark}

We say that an sp-filtration $\Phi$ of $\Spec R$ is \emph{bounded} if there are integers $s$ and $t$ such that $\Phi(s)=\Spec R$ and $\Phi(t)=\emptyset$ (cf. \cite[\S 5.1, Definition]{ATJLS10}), or equivalently, such that $s< \f_\Phi(\pp)\leq t$ for all $\pp\in \Spec R$. Every bounded sp-filtration is non-degenerate by definition.

In the rest of this section, using \cref{P:dualizing}, we establish a short proof for the claims in \cref{intro-slice,intro-tilt}, provided that $R$ admits a classical dualizing complex and $\Phi$ is a bounded slice sp-filtration; see \cref{cdc-shortcut} below.

\begin{lemma}\label{fin-proj-dim}
Assume that $R$ has finite Krull dimension. Let $\Phi$ be a bounded sp-filtration of $\Spec R$.
Then $T_\Phi$ is isomorphic to a bounded complex of projective $R$-modules in $\D(R)$.
\end{lemma}

\begin{proof}
For each $\pp\in \Spec R$, $\RGamma_{\pp} R_{\pp}$ is isomorphic to a complex of flat $R$-modules concentrated in degrees from $0$ to $\height(\pp)$; see \cref{radical-torsion-sop} below. Since $\height(\pp)\leq \dim R<\infty$ and $\Phi$ is bounded, $T_{\Phi}$ is isomorphic to a bounded complex of flat $R$-modules.
Then the lemma follows from \cite[Part II, Corollary 3.2.7]{RG71}, which states that every flat $R$-module has projective dimension at most $\dim R$. 
\end{proof}

\begin{remark}\label{radical-torsion-sop}
Let $(R,\mm)$ be a commutative noetherian local ring, and take a system of parameters $\boldsymbol{x}=x_1,\ldots,x_{d}$ of $R$, where $d=\dim R$; see \cite[\S 14]{Mat89}.
If $I$ is the ideal generated by $\boldsymbol{x}$, it has the same radical as the maximal ideal $\mm$, so that $\Gamma_I=\Gamma_\mm$; see \cite[Proposition 7.3]{ILL+07}.
Thus we have a natural isomorphism
\[\RGamma_\mm X\cong \check{C}(\boldsymbol{x})\otimes_R X\]
for every $X\in \D(R)$ by \cref{Cech-comp}, where $\check{C}(\boldsymbol{x})$ is a complex of flat $R$-modules concentrated in degrees from 0 to $d$.

Let $R$ be an arbitrary commutative noetherian ring, and take $\pp\in \Spec R$.
Regard $\RGamma_{\pp R_\pp} R_\pp\in \D(R_\pp)$ as an object in $\D(R)$ via the (fully faithful) scalar restriction functor $\D(R_\pp)\to \D(R)$. 
Then we have a natural isomorphism
\[\RGamma_{\pp} R_\pp\cong \RGamma_{\pp R_\pp} R_\pp\]
 in $\D(R)$; cf. \cite[Proposition 7.15(3)]{ILL+07}.
Moreover, by the first paragraph, $\RGamma_{\pp R_\pp} R_\pp$ is isomorphic in $\D(R_\pp)$ to a bounded complex of flat $R_\pp$-modules concentrated in degrees from 0 to $\dim R_\pp=\height(\pp)$.
Therefore $\RGamma_{\pp} R_\pp$ is isomorphic in $\D(R)$ to a bounded complex of flat $R$-modules concentrated in degrees from 0 to $\height(\pp)$.
\end{remark}

\begin{example}
$\Spec \ZZ$ has a slice sp-filtration which is not bounded.\footnote{This does not conflict with \cite[Corollary 4.8(3)]{ATJLS10}, which implicitly assumes the weak Cousin condition. See also \cref{w-s-Cousin}.}
For this, choose a function $\f: \Spec \ZZ\to \ZZ$ such that $\f((0))= 0$ and $\f((p))= p$ for every prime $p$. Then $\f$ is strictly increasing, as $\f((0))<\f((p))$ for every prime $p$, so 
the sp-filtration $\Phi_\f$ corresponding to $\f$ is a slice sp-filtration, but this is not bounded.
\end{example}

\begin{lemma}\label{Kb(ProjR)}
Assume that $R$ admits a classical dualizing complex. Let $\Phi$ be a bounded sp-filtration of $\Spec R$.
Then $R\in \thick(\Add(T_\Phi))$.
Therefore, all bounded complexes of projective $R$-modules belong to $\thick(\Add(T_\Phi))$.
\end{lemma}
\begin{proof}
We first show that $\thick(\Add(T_\Phi))=\thick(\Add(X))$ for $X:=\bigoplus_{\pp\in \Spec R} \RGamma_\pp R_\pp$.
Let $I$ be the image of the function $\f_\Phi:\Spec R\to \ZZ$.
By assumption, $I$ is a finite set.
Setting $X_i:=\bigoplus_{\f_{\Phi}(\pp)=i}\RGamma_\pp R_\pp$, we can write $T_{\Phi}=\bigoplus_{i\in I} \Sigma^iX_i$.
Then it follows that $T_{\Phi}\in \thick(\Add(X))$ and $X\in \thick(\Add(T_{\Phi}))$.

Next, let $D$ be a classical dualizing complex and assume that $D$ is a bounded complex of injective $R$-modules. Let $\cd$ be the codimension function associated to $D$, that is, $\cd:=\cd_D$; see \cref{dc-loc-coh}. 
The above argument implies that $\thick(\Add(T_\Phi))=\thick(\Add(X))=\thick(\Add(T_{\Phi_\cd}))$, where $T_{\Phi_\cd}=\bigoplus_{\pp\in \Spec R} \Sigma^{\cd(\pp)}\RGamma_\pp R_\pp$.
Hence it suffices to show $R\in \thick(\Add(T_{\Phi_\cd}))$.

In the rest of the proof, using $\cref{TPhi-check}$, we regard $T_{\Phi_\cd}$ as $ \bigoplus_{\pp\in \Spec R} \Sigma^{\cd(\pp)} \check{C}(\pp)_\pp$. 
By \cref{Murfet}, we have the canonical embedding $\D(R)\hookrightarrow \D(\Flat R)$ and this sends $T_{\Phi_\cd}$ to $T_{\Phi_\cd}$.
In addition, the canonical embedding $\D(R)\hookrightarrow \K(\Inj R)$ sends $E:=\bigoplus_{\pp\in \Spec R} E_R(R/\pp)$ to $E$.
By \cref{dc-loc-coh2}, the equivalence $D\otimes_R-: \D(\Flat R)\isoto \K(\Inj R)$ in \cref{NIK}\cref{Iyengar-Krause} sends $T_{\Phi_\cd}$ to $E$. Hence this equivalence induces an equivalence  $\thick(\Add(T_{\Phi_{\cd}}))\isoto \thick(\Add(E))$ of the subcategories of $\D(R)$.
Since  $\thick(\Add(E))$ contains all bounded complexes of injective $R$-modules, we have $D\in \thick(\Add(E))$. 
The quasi-inverse $\Hom_R(D,-): \K(\Inj R)\isoto \D(\Flat R)$ sends $D\in \K(\Inj R)$ to $\Hom_R(D,D) \in \D(\Flat R)$, so it follows that $R\cong \Hom_R(D,D) \in \thick(\Add(T_{\Phi_{\cd}}))$.

The second claim follows from \cref{thick-closure}. 
\end{proof}

\begin{remark}\label{cdc-shortcut}
Assume $R$ admits a classical dualizing complex. Then \cref{P:dualizing} implies, in conjunction with \cref{silt-bound} and \cref{fin-proj-dim,Kb(ProjR)}, that $T_\Phi$ is a silting object in $\D(R)$ for every bounded slice sp-filtration $\Phi$, and that $T_\Phi$ is a tilting object in $\D(R)$ for every codimension filtration $\Phi$. Hence \cref{P:dualizing} gives a quite reasonable approach to know whether $T_\Phi$ can be silting and tilting. Having established the silting property of $T_\Phi$ for a bounded slice sp-filtration $\Phi$, it is also not difficult to show the equality $T\Perp{>0}_\Phi = \cY_\Phi$ by a direct argument. Indeed, since $T_\Phi$ is a bounded silting object, it is of finite type (see \cref{of-finite-type}), and therefore $T\Perp{>0}_\Phi = \cY_{\Phi'}$ for some non-degenerate sp-filtration $\Phi'$ by \cref{silt-fin-sp}. It remains to show that $\Phi = \Phi'$, or equivalently, $\f_\Phi=\f_{\Phi'}$. 
Fix $\qq\in \Spec R$. Since $\RHom_{R}(R_\qq, \kappa(\qq))\cong \kappa(\qq)$, it follows that $\width_{R_\qq} \RHom_{R}(R_\qq, \Sigma^{i}\kappa(\qq))=-i$ for all $i\in \ZZ$.
In addition, $\width_{R_\pp} \RHom_{R}(R_\pp, \kappa(\qq))=\infty$ whenever $\qq\neq \pp\in \Spec R$.
Thus we have $\Sigma^{i}\kappa(\qq)\in T\Perp{>0}_\Phi$ if and only if $-i\geq \f_\Phi(\qq)$ by \cref{TPhi-posit}.
On the other hand, given $\pp\in \Spec R$ with $\pp\subseteq \qq$, we have $\width_{R}(\pp,\Sigma^i\kappa(\qq))=-i$ for all $i\in \ZZ$, and if $\pp\not\subseteq \qq$, then $\width_{R}(\pp,\Sigma^i\kappa(\qq))=\infty$. 
Hence, it follows from \cref{cYPhi-posit} that $\Sigma^{i}\kappa(\qq)\in \cY_{\Phi'}$ if and only if $-i\geq \f_{\Phi'}(\qq)$. Since $T\Perp{>0}_\Phi = \cY_{\Phi'}$, it has been shown that $-i\geq \f_{\Phi}(\qq)$ if and only if $-i\geq \f_{\Phi'}(\qq)$ for all $i\in \ZZ$. In other words, $\f_{\Phi}(\qq) =\f_{\Phi'}(\qq)$. Therefore $\Phi=\Phi'$.
\end{remark}

In the next section, for any slice sp-filtration $\Phi$ of $\Spec R$, we will prove the equality $T\Perp{>0}_\Phi = \cY_{\Phi}$ and that $T_\Phi$ is silting in $\D(R)$ without any additional assumption; see \cref{slice-silting}.
Its proof is independent from \cref{P:dualizing}, while we will need it and \cref{P:dualizing}\cref{negat-ext-dc} both to prove \cref{tilt-dc}. 
We also mention that, by \cref{slice-silting} and \cref{necessity}\cref{silt-necessity}, the assumption on the existence of a dualizing complex in \cref{P:dualizing} can be removed from \cref{posit-ext-dc} and from the ``only if'' part of \cref{negat-ext-dc}.

	
\section{Proof of \cref{intro-slice}}
\label{proof-theorem}	

Let $R$ be a commutative noetherian ring.
To prove \cref{intro-slice} in full generality, we need to treat a possibly larger dimension (as defined below) than $\dim R$. 
Given a subset $W\subseteq \Spec R$, we denote by $\max W$ (resp. $\min W$) the set of maximal (resp. minimal) elements in the poset $W$.

We recursively define specialization subsets $W_\alpha$ of $\Spec R$ for each ordinal $\alpha$ in the following way. 
First, let $W_0$ be the set of maximal ideals of $R$.
If $W_\alpha$ is defined for some ordinal $\alpha$, and $W_\alpha\subsetneq \Spec R$, then 
we put $W_{\alpha+1}:=W_\alpha\cup \max(\Spec R\sm W_\alpha)$. 
If $\beta$ is a limit ordinal and $W_\alpha$ is defined for every $\alpha<\beta$, then we put $W_\beta := \bigcup_{\alpha < \beta} W_\alpha$. 
In this way, we obtain the least ordinal $\delta$ satisfying $W_\delta=\Spec R$. Writing $\Kdim R$ for $\delta$, we call it the \emph{large Krull dimension} of $R$.

By construction, $W_\alpha$ is specialization closed for every $\alpha\leq \Kdim R$, and $W_\alpha\subsetneq W_\beta$ whenever $\alpha<\beta\leq \Kdim R$. Moreover, $\Kdim R=\delta$ is always a successor ordinal because $R$ has only finitely many minimal prime ideals.

If $\dim R$ is infinite, we may interpret it as the first infinite ordinal, in terms of the definition of $\dim R$. 
Under this convention, we always have $\dim R\leq \Kdim R$, and the equality holds if $\dim R$ is finite.
We remark that $\Kdim R$ could be uncountable in general; see \cite[Corollary 8.14, and Theorem 9.8]{GR73}.\footnote{
The large Krull dimension of $R$ is greater than or equal to the ``classical Krull dimension'' of $R$ in the sense of \cite[p.~48]{GR73}. To obtain the latter ordinal, we only need to choose $(\bigcup_{\alpha <\beta} W_{\alpha})\cup \max(\Spec R\sm \bigcup_{\alpha<\beta} W_{\alpha})$ instead of $\bigcup_{\alpha <\beta} W_{\alpha}$
when $W_{\beta}$ is defined for each limit ordinal $\beta$.
Then we see that the difference of the two ordinals is up to one.
We define $\Kdim R$ as above because it is slightly better for our transfinite induction in the proof of \cref{classes-equal}.
}

The main task of this section is to show the equality $T\Perp{>0}_\Phi=\cY_\Phi$ for a slice sp-filtration $\Phi$ of $\Spec R$.
 To this end, for an ordinal $\alpha \leq \Kdim R$,  we define the following subcategories:
	\begin{align*}
	\cY_{(\Phi, \alpha)} &:= \{X \in \D(R) \mid \width_R(\pp, X)>n ~\forall n \in \ZZ ~\forall \pp \in \Phi(n)\cap W_\alpha\},\\
	\cY'_{(\Phi, \alpha)} &:= \{X \in \D(R) \mid \width_{R_\pp}\RHom_R(R_{\pp},X) \geq \f_{\Phi}(\pp) ~\forall \pp \in W_\alpha \},\\
	\cY''_{(\Phi, \alpha)} &:= \{X \in \D(R) \mid \width_R(V,X)>n  ~\forall n \in \ZZ ~\forall V \subseteq \Phi(n)\cap W_\alpha\}.
\end{align*}
Note that $V$ in the last definition is assumed to be specialization closed.
We will show $\cY_{(\Phi, \alpha)}=\cY'_{(\Phi, \alpha)}=\cY''_{(\Phi, \alpha)}$ for every ordinal $\alpha \leq \Kdim R$
by using transfinite induction (\cref{classes-equal}). 
The only trivial inclusion is: 
$\cY''_{(\Phi, \alpha)}\subseteq \cY_{(\Phi, \alpha)}$. Notice also from \cref{cYPhi-posit} that
\begin{equation}\label{cY-f}
\cY_{(\Phi, \alpha)} = \{X \in \D(R) \mid \width_R(\pp, X)\geq\f_\Phi(\pp) ~\forall \pp \in W_\alpha\}.
\end{equation}

We start with the following lemma.
For simplicity, we will write a triangle $X\to Y\to Z\to \Sigma X$ in $\D(R)$ as $X\to Y\to Z\xrightarrow{+}$ throughout the rest of the paper.

\begin{lemma}\label{dim0-triangle}
Let $V_0\subseteq V$ be specialization closed subsets of $\Spec R$. If $\dim (V\sm V_0)\leq 0$, then there is a triangle of the form
	\[\prod_{\pp \in V \setminus V_0} \LLambda^{\pp}\RHom_R(R_{\pp},X) \rightarrow \lambda^{V} X \rightarrow \lambda^{V_0} X \xrightarrow{+}\]
for every $X\in \D(R)$.
\end{lemma}

\begin{proof}
We first remark that the canonical morphism $\RGamma_{V_0}R\cong \RGamma_{V_0}\RGamma_{V}R\to \RGamma_{V}R$ yields a triangle
\begin{equation*}\label{relative-lc}
\RGamma_{V_0}R\to \RGamma_{V}R\to \bigoplus_{\pp\in V\sm V_0}\RGamma_{\pp}R_\pp\xrightarrow{+}.
\end{equation*}
See \cref{Gamma-Inj} below.
Next, let $X\in \D(R)$, and apply $\RHom_R(-,X)$ to the above triangle. Then we obtain the following triangle:
\[\prod_{\pp\in V\sm V_0}\RHom_R(\RGamma_{\pp}R_\pp
,X)\to \RHom_R(\RGamma_{V}R,X)\to \RHom_R(\RGamma_{V_0}R,X) \xrightarrow{+}.\]
Hence we have the triangle in the lemma by \cref{(co)smash-iso,RHom-TPhi}.
\end{proof}

\begin{remark}\label{Gamma-Inj}
For the reader's convenience, we here summarize some basic facts on a complex $I$ of injective $R$-modules; part of them are used in the proofs of \cref{dim0-triangle,classes-equal}.

For each $n\in \ZZ$, we may write $I^n=\bigoplus_{\pp\in \Spec R} E_R(R/\pp)^{(\mu^n_\pp)}$ by some cardinals $\mu^n_\pp$ (see \cite[Theorem 18.5]{Mat89}).
Given a multiplicatively closed subset $S\subset R$,
the canonical chain map $I\to I\otimes_R S^{-1}R$ is degreewise split surjective, where each map $I^n\to I^n\otimes_R S^{-1}$ can be identified with the canonical surjection $\bigoplus_{\pp\in \Spec R} E_R(R/\pp)^{(\mu^n_\pp)}\to \bigoplus_{\pp \cap S=\emptyset} E_R(R/\pp)^{(\mu^n_\pp)}$.
Moreover, for a specialization closed subset $V$, the $n$th component of the canonical chain map $\Gamma_VI\to I$ can be identified with the canonical injection $\bigoplus_{\pp\in V} E_R(R/\pp)^{(\mu^n_\pp)}\to \bigoplus_{\pp\in\Spec R} E_R(R/\pp)^{(\mu^n_\pp)}$.

If $V_0\subseteq V$ are specialization closed subsets of $\Spec R$, then the canonical morphism $\Gamma_{V_0}\cong \Gamma_{V_0}\Gamma_V\to \Gamma_V$ of functors induces a degreewise split exact sequence 
\begin{equation}\label{relative-lc1}
0\to \Gamma_{V_0}I\to \Gamma_VI\to \Gamma_VI/\Gamma_{V_0}I\to 0,
\end{equation}
where $(\Gamma_VI/\Gamma_{V_0}I)^n\cong \bigoplus_{\pp\in V\sm V_0} E_R(R/\pp)^{(\mu^n_\pp)}$.
Assume $\dim (V/V_0)\leq 0$. Then we have a decomposition
$\Gamma_VI/\Gamma_{V_0}I\cong \bigoplus_{\pp\in V/V_0}\Gamma_\pp I_\pp$ as complexes; see \cref{comparison}\cref{Gamma-adjoint}. 
If $I$ is K-injective, then \cref{relative-lc1} clearly realizes the triangle 
\begin{equation}\label{relative-lc2}
\RGamma_{V_0}X\to \RGamma_{V}X\to \bigoplus_{\pp\in V\sm V_0}\RGamma_{\pp}X_\pp\xrightarrow{+}
\end{equation}
in $\D(R)$ for every complex $X$ quasi-isomorphic to $I$. 
In particular, letting $I$ be an injective resolution of $R$, we may put $X:=R$ in \cref{relative-lc2}.

In fact, the triangle \cref{relative-lc2} exists for any complex $X\in \D(R)$ by the isomorphism $(\RGamma_{W}R)\LotimesR X\cong \RGamma_{W}X$ for a specialization closed subset $W$ (\cite[Proposition 3.5.5(ii)]{Lip02}). More precisely, \cref{relative-lc1} realizes \cref{relative-lc2} whenever $I$ is a complex of injective modules that is quasi-isomorphic to a given complex $X$, because $\RGamma_{W}I\cong \Gamma_{W}I$ holds without K-injectivity on $I$ (\cite[Lemma 3.5.1]{Lip02}).
\end{remark}

\begin{remark}\label{slice-functor}
Given specialization closed subsets $V_0\subseteq V$ of $\Spec R$, we have the triangle 
\[\gamma_{V_0}\gamma_VX\to \gamma_VX\to \lambda^{V_0^{\cp}}\gamma_VX \xrightarrow{+}\]
by \cref{stable-approx}, where $\gamma_{V_0}\gamma_VX\cong \gamma_{V_0}X$ by \cref{trivial-iso}.
If $\dim (V/V_0)\leq 0$, we see from  \cref{(co)smash-iso} that the above triangle coincides with \cref{relative-lc2}.
It is also possible to deduce the isomorphism $\lambda^{V_0^{\cp}}\gamma_VX\cong \bigoplus_{\pp\in V\sm V_0}\RGamma_{\pp}X_\pp$ from a more formal argument. 

We may regard the composite functor $\lambda^{V_0^{\cp}}\gamma_V$ as an analogue of a ``slice functor'' in \cite[\S 2.2]{GRSO12}; which is used to decompose spectra in that context into computable pieces called slices; see \cite[\S 1]{GRSO12} and also \cite[\S 4]{HHR16}.
In our case, the condition $\dim (V\sm V_0)\leq 0$ yields such a computable situation, and this fact leads to the definition of our slice sp-filtrations.
\end{remark}

For the proof of the next proposition, we make a remark.

\begin{remark}\label{for-Lukas}
Let $X\in \D(R)$.
For an integer $n$ and a specialization closed subset $V\subseteq \Spec R$, $\width_R(V,X)>n$ if and only if $\sup\RHom_R(\RGamma_VR,X)<-n$, that is, $H^i\RHom_R(\RGamma_VR,X)=0$ for every $i\geq -n$. The last condition is equivalent to that 
$\Hom_{\D(R)}(\RGamma_VR,\Sigma^{i-n}X)=0$ for every $i\geq 0$, and we can rewrite this as the single equality:
$$\prod_{i\geq 0}\Hom_{\D(R)}(\Sigma^{n-i} \RGamma_VR,X)=0.$$
Hence it holds that 
\[\width_R(V,X)>n\Leftrightarrow X\in \Big(\bigoplus_{i\geq 0}\Sigma^{n-i} \RGamma_VR\Big)\Perp{0}.\]

Now, let $\Phi$ be a slice sp-filtration of $\Spec R$.
It then follows, for an ordinal $\alpha\leq \Kdim R$, that
\begin{align*}
X\in \cY''_{(\Phi, \alpha)} &\Leftrightarrow 
X\in \Big(\bigoplus_{i\geq 0}\Sigma^{n-i} \RGamma_VR\Big)\Perp{0} ~\forall V \subseteq \Phi(n)\cap W_\alpha ~\forall n \in \ZZ\\
&\Leftrightarrow 
X\in 
\Bigg(\bigoplus_{n \in \ZZ}
\bigg(\bigoplus_{V \subseteq \Phi(n)\cap W_\alpha}\Big(\bigoplus_{i\geq 0}\Sigma^{n-i} \RGamma_VR\Big)\bigg)\Bigg)\Perp{0}.
\end{align*}
This will allow us to use the Lukas lemma via  \cref{cofiltration}.
\end{remark}
	
\begin{proposition}\label{classes-equal}
Let $\Phi$ be a slice sp-filtration of $\Spec R$.
	For each ordinal $\alpha \leq \Kdim(R)$, we have the equalities
		\[\cY_{(\Phi, \alpha)} = \cY'_{(\Phi, \alpha)} = \cY''_{(\Phi, \alpha)}.\]
\end{proposition}
\begin{proof}
For brevity, we simply write $\cY_\alpha$, $\cY'_\alpha$, and $\cY''_\alpha$ instead of $\cY_{(\Phi, \alpha)}$, $\cY'_{(\Phi, \alpha)}$, and $\cY''_{(\Phi, \alpha)}$ respectively.
We proceed by transfinite induction on $\alpha\leq \Kdim R$. 
If $\alpha=0$, then $W_0$ consists of the maximal ideals of $R$. 
Hence, for every $V\subseteq W_0$, we have
\[\prod_{\mm \in V} \LLambda^\mm \cong\prod_{\pp \in V} \LLambda^{\mm}\RHom_R(R_{\mm},-) \cong \lambda^{V}\] 
by \cref{RHom-TPhi} and \cref{dim0-triangle} applied to $V$ and $V_0:=\emptyset$.
Thus the equality $\cY_0=\cY''_0$ holds. The equality $\cY_0=\cY'_0$ follows from the first isomorphism above and \cref{cY-f}.

We next establish the successor step. Assume that $\cY_\alpha = \cY'_\alpha = \cY''_\alpha$ for some ordinal $\alpha<\Kdim R$. 
For the desired equality $\cY_{\alpha+1} = \cY'_{\alpha+1} = \cY''_{\alpha+1}$, it suffices to show that $\cY_{\alpha + 1} \subseteq \cY'_{\alpha + 1} \subseteq \cY''_{\alpha + 1}$ because   
 $\cY''_{\alpha+1}\subseteq \cY_{\alpha+1}$ by definition.
Let $X\in \cY_{\alpha +1}$. Note that $\cY_{\alpha +1} \subseteq \cY_\alpha=\cY'_\alpha=\cY''_\alpha$ by definition and assumption.
Moreover, since $X\in \cY_{\alpha +1}\subseteq \cY'_{\alpha}$,
\begin{align*}
X\in \cY'_{\alpha +1}
&\Leftrightarrow \width_{R_\pp}\RHom_{R}(R_\pp,X)\geq \f_\Phi(\pp)~\forall\pp \in W_{\alpha+1}=W_{\alpha}\cup (W_{\alpha+1}\sm W_{\alpha})\\
&\Leftrightarrow \width_{R_\pp}\RHom_{R}(R_\pp,X)\geq \f_\Phi(\pp)~\forall\pp \in W_{\alpha+1}\sm W_{\alpha}.
\end{align*}
We have $\dim (W_{\alpha+1}\sm W_{\alpha})=0$ by construction, so $\dim (V \sm W_{\alpha})\leq 0$ for any integer $n$ and any specialization closed subset $V \subseteq \Phi(n)\cap W_{\alpha+1}$. Thus \cref{dim0-triangle} yields the triangle
	\begin{align}\label{ind-successor}
	\prod_{\pp \in V \setminus W_\alpha} \LLambda^{\pp}\RHom_R(R_{\pp},X) \rightarrow \lambda^{V}X \rightarrow \lambda^{V\cap W_\alpha}X \xrightarrow{+}.
	\end{align}

Now, take $\qq \in W_{\alpha + 1} \setminus W_\alpha$ and put $n: = \f_\Phi(\qq) -1$, so that $V(\qq) \cap W_\alpha = V(\qq) \setminus \{\qq\}$ and $\qq \in \Phi(n)$. 
Then $\qq\in \Phi(n) \cap W_{\alpha+1}$, and hence $\width_{R}(\qq,X)>n$ since $X\in \cY_{\alpha+1}$. Furthermore, $\Phi$ is a slice sp-filtration, so $\dim (\Phi(n)\sm \Phi(n+1))\leq 0$, and this along with the inclusion $V(\qq)\subseteq \Phi(n)$ implies  that $V(\qq) \setminus \{\qq\}\subseteq \Phi(n+1)$. In particular, we have $V(\qq) \cap W_\alpha=V(\qq) \setminus \{\qq\}\subseteq \Phi(n+1)\cap W_\alpha$, so $\width_R(V(\qq) \cap W_\alpha,X)>n+1$ since $X \in \cY_{\alpha+1}\subseteq \cY''_\alpha$.
We have observed that 
\[\sup \lambda^{V(\qq)}X<-n ~\text{ and }~ \sup \lambda^{V(\qq)\cap W_\alpha}X<-n-1.\]
By \cref{ind-successor} applied to $V(\qq)\subseteq \Phi(n) \cap W_{\alpha+1}$, we obtain the triangle 
\begin{equation*}
\LLambda^{\qq}\RHom_R(R_{\qq},X) \rightarrow \lambda^{V(\qq)}X \rightarrow \lambda^{V(\qq)\cap W_\alpha}X \xrightarrow{+}.
\end{equation*}
It then follows that $\sup \LLambda^{\qq}\RHom_R(R_{\qq},X)<-n$, that is, $\width_{R_\qq}\RHom_R(R_\qq, X)\geq \f(\qq)$.
We have shown that $X\in \cY'_{\alpha+1}$ (see the previous paragraph).
Thus $\cY_{\alpha+1} \subseteq \cY'_{\alpha+1}$.

To prove the remaining inclusion $\cY'_{\alpha+1}\subseteq \cY''_{\alpha+1}$, take $X \in \cY'_{\alpha + 1}$. 
Let $V \subseteq \Phi(n) \cap W_{\alpha+1}$ be a specialization closed subset for an integer $n\in \ZZ$. Since $X\in \cY'_{\alpha+1}\subseteq \cY'_{\alpha}=\cY''_{\alpha}$ and  $V\cap W_\alpha\subseteq \Phi(n) \cap W_{\alpha}$, it follows that $\width_R(V\cap W_\alpha,X)>n$.
On the other hand, if $\pp \in V \setminus W_{\alpha}$, then $\pp\in \Phi(n)\cap W_{\alpha +1}$ as $V \subseteq \Phi(n) \cap W_{\alpha+1}$.
Thus we have $\width_{R_\pp}\RHom_R(R_\pp,X)\geq \f_{\Phi}(\pp)>n$ by definition and assumption.
We have observed that 
\[\sup \lambda^{V\cap W_\alpha}X<-n ~\text{ and }~ \sup \LLambda^\pp\RHom_R(R_\pp,X)<-n.\]
Thus we have $\sup \lambda^VX< -n$ by \cref{ind-successor}, that is, $\width_{R}(V,X)>n$, and hence $X\in \cY''_{\alpha+1}$. Therefore, the inclusion $\cY'_{\alpha+1}\subseteq \cY''_{\alpha+1}$ follows.

Finally, let us establish the limit step of the induction.\footnote{If $\dim R$ is finite, then $\dim R=\Kdim R$, so the limit step can be omitted. \label{footnote-Kdim}}
Let $\beta< \Kdim R$ be a limit ordinal, and suppose that $\cY_\alpha = \cY'_\alpha = \cY''_\alpha$ for all $\alpha < \beta$.
Since $W_{\beta}=\bigcup_{\alpha < \beta} W_\alpha$, it follows that $\cY_{\beta}=\bigcap_{\alpha<\beta} \cY_{\alpha}$ and $\cY'_{\beta}=\bigcap_{\alpha<\beta} \cY'_{\alpha}$. Thus we have $\cY_{\beta} = \cY'_{\beta}$.
For the remaining equality $\cY'_{\beta} =\cY''_{\beta}$, it suffices to show that $\cY'_{\beta} \subseteq \cY''_{\beta}$ because $\cY''_{\beta} \subseteq \cY_{\beta}$ by definition. Take $X\in \cY'_{\beta}$. We first observe that $X(\pp):=\LLambda^\pp \RHom_R(R_\pp, X)\in \cY''_\beta$ for an arbitrary $\pp\in\Spec R$.
Let  $n$ be an integer and $V \subseteq \Phi(n)\cap W_{\beta}$ be a specialization closed subset. 
Recall that $\cosupp_RX(\pp)\subseteq \{\pp\}$ (\cref{RHom-TPhi}). If $\pp\notin V$, \cref{bi-loc} implies that $\lambda^VX(\pp)=0$, that is, $\width_R(V,X(\pp))=\infty>n$.
If $\pp\in V$, then $\lambda^VX(\pp)\cong X(\pp)$, so
\[\width_R(V,X(\pp))=-\sup X(\pp)= \width_{R_\pp} \RHom_R(R_\pp,X)\geq \f_\Phi(\pp)>n.\]
Thus $X(\pp)\in  \cY''_\beta$.

We next show that there is a complex $Y$ such that $X\cong Y$ in $\D(R)$ and $Y$ has a $\cC$-cofiltration, where $\cC$ denotes the subcategory of $\C(R)$ formed by all objects in $\cY''_\beta$.
If this is done, then \cref{cofiltration} implies that $X\cong Y\in \cY''_\beta$ (see \cref{for-Lukas}), so $\cY'_{\beta} \subseteq \cY''_{\beta}$. Hence the proof will be completed.

To this end, take K-injective resolutions $R\to I$ and $X\to J$ such that $I$ and $J$ consist of injective $R$-modules.
Then $X\cong \RHom_R(R, X)\cong \Hom_R(I, J)$ in $\D(R)$.
We put 
\[Y:=\Hom_R(I, J) ~\text{ and }~ Y_\alpha:=\Hom_R(\Gamma_{W_\alpha}I, J)\]
for each $\alpha\leq \Kdim R$, where 
$\RHom_R(\RGamma_{W_\alpha}R, X)\cong Y_\alpha$ in $\D(R)$.
Consider a direct system formed by the canonical chain maps $\iota_{\alpha,\nu}:\Gamma_{W_\alpha}I\to \Gamma_{W_{\nu}}I$ for all $\alpha\leq \nu\leq \delta=\Kdim R$, where every $\iota_{\alpha,\nu}$ is degreewise split injective (see \cref{Gamma-Inj}) and $\varinjlim_{\alpha\leq \delta} \Gamma_{W_\alpha}I \cong \Gamma_{W_\delta}I\cong I$ in $\C(R)$.
Then we obtain the inverse system $(Y_\alpha,\pi_{\alpha, \nu} \mid \alpha\leq \nu\leq \delta)$ formed by $\pi_{\alpha,\nu}:= \Hom_R(\iota_{\alpha,\nu},J)$ for all $\alpha\leq \nu \leq \delta$, where every $\pi_{\alpha,\nu}$ is degreewise split surjective and 
\[Y=\Hom_R(I, J)\cong \varprojlim_{\alpha\leq \delta}  \Hom_R(\Gamma_{W_\alpha}I,J)=\varprojlim_{\alpha\leq \delta} Y_\alpha\] in $\C(R)$.
If $\nu$ is a limit ordinal, then $W_\nu=\bigcup_{\alpha<\nu}{W_\alpha}$ by definition, so $\Gamma_{W_\nu}I\cong \varinjlim_{\alpha<\nu}\Gamma_{W_\alpha}I$, and hence $Y_\nu\cong \varprojlim_{\alpha<\nu}Y_\alpha$.

We remark that $\cY''_{\beta}$ is closed under products since $\lambda^{V}\cong \RHom_R(\RGamma_VR,-)$ for any specialization closed subset $V\subseteq \Spec R$. 
Moreover, 
\[\prod_{\mm\in W_0} X(\mm) \cong \prod_{\mm\in W_0} \RHom_R(\RGamma_\mm R,X) \cong \prod_{\mm\in W_0}\Hom_{R}(\Gamma_{\mm}I,J)\cong \Hom_{R}(\Gamma_{W_0}I,J)= Y_0\] in $\D(R)$ by \cref{RHom-TPhi,Gamma-Inj}.
Since we know that $X(\pp)\in \cY''_\beta$ for all $\pp\in \Spec R$, it follows that $Y_0 \cong \prod_{\mm\in W_0} X(\mm) \in \cY''_\beta$, that is, $Y_0\in \cC$.
To observe that the kernel of $\pi_{\alpha,\alpha+1}$ belongs to $\cC$ for each $\alpha<\delta$,  
consider the degreewise split exact sequence 
\[0\to \Gamma_{W_\alpha}I\xrightarrow{\iota_{\alpha,\alpha+1}} \Gamma_{W_{\alpha+1}}I\to \Gamma_{W_{\alpha+1}}I/\Gamma_{W_\alpha}I\to 0.\] 
Applying $\Hom_R(-,J)$ to it, we obtain the exact sequence 
\begin{equation}\label{Lukas-kernel}
0\to \Hom_R(\Gamma_{W_{\alpha+1}}I/\Gamma_{W_\alpha}I,J) \to Y_{\alpha+1} \xrightarrow{\pi_{\alpha,\alpha+1}} Y_{\alpha}\to 0.
\end{equation}
Sending this to $\D(R)$, we obtain the triangle in \cref{dim0-triangle} applied to $W_{\alpha}\subseteq W_{\alpha+1}$; see the proof of the lemma.
Furthermore, putting $Y(\pp):=\Hom_R(\Gamma_\pp I_\pp,J)$,
we have 
\[\Hom_R(\Gamma_{W_\alpha+1}I/\Gamma_{W_\alpha}I,J)\cong \prod_{\pp\in W_{\alpha+1}\sm W_\alpha}\Hom_R(\Gamma_\pp I_\pp,J)=\prod_{\pp\in W_{\alpha+1}\sm W_\alpha}Y(\pp)\]
in $\C(R)$ (see \cref{Gamma-Inj}). 
As a consequence, there are isomorphisms
\[\prod_{\pp\in W_{\alpha+1}\sm W_\alpha}Y(\pp)\cong \prod_{\pp\in W_{\alpha+1}\sm W_\alpha}\LLambda^{\pp} \RHom_R(R_\pp,X)=\prod_{\pp\in W_{\alpha+1}\sm W_\alpha}X(\pp)\]
in $\D(R)$. Since the last one belongs to $\cY''_\beta$, so does the first, and hence $\prod_{\pp\in W_{\alpha+1}\sm W_\alpha}Y(\pp)\cong \Ker\pi_{\alpha,\alpha+1}\in \cC$.
Therefore, we have shown that the inverse system $(Y_\alpha,\pi_{\alpha, \nu} \mid \alpha\leq \nu \leq \delta)$ is a $\cC$-cofiltration, as desired.
\end{proof}

We now prove one of the main results of this paper (\cref{intro-slice}).

\begin{theorem}\label{slice-silting}
	Let $R$ be a commutative noetherian ring and $\Phi$ a slice sp-filtration of $\Spec R$. Then $T_\Phi$ is a silting object in $\D(R)$ such that $(T\Perp{>0}_\Phi,T\Perp{\leq 0}_\Phi)=(\cY_{\Phi},\cY\Perp{0}_{\Phi})$.
\end{theorem}

\begin{proof}
Since $W_\delta=\Spec R$, we have $\cY_{(\Phi,\delta)}=\cY_{\Phi}$ for $\delta=\Kdim R$, and $T\Perp{>0}_\Phi=\cY'_{(\Phi,\delta)}$ by \cref{TPhi-posit}.
Thus $T\Perp{>0}_\Phi=\cY_{\Phi}$ by \cref{classes-equal}. In particular, $T\Perp{>0}_\Phi$ is closed under coproducts by  \cref{comp-gen-co-t}.
By \cref{generate}, we have $T_\Phi\Perp{\ZZ} = 0$. Hence it remains to show that $T_\Phi \in T\Perp{>0}_\Phi$; see \cref{silting-cond}.

Since $T\Perp{>0}_\Phi=\cY_{\Phi}$ is closed under coproducts, it suffices to show that $\Sigma^{\f_\Phi(\pp)}\RGamma_{\pp}R_{\pp} \in \cY_{\Phi}$ for each $\pp \in \Spec R$.
Thus, if we observe that 
\begin{equation*}
\width_R(\qq, \Sigma^{\f_\Phi(\pp)}\RGamma_{\pp}R_{\pp})>n
\end{equation*}
for any $n\in \ZZ$ and $\qq\in \Phi(n)$,
then the proof will be completed.

It follows from \cref{FI2} and \cref{Cech-comp} that
\begin{equation}\label{width-compute}
\width_R(\qq, \Sigma^{\f_\Phi(\pp)}\RGamma_{\pp}R_{\pp})=-\sup 
(R/\qq)\LotimesR \Sigma^{\f_\Phi(\pp)}\RGamma_{\pp}R_{\pp}=-\sup 
\Sigma^{\f_\Phi(\pp)}\RGamma_{\pp}(R_\pp/\qq_\pp).
\end{equation}
Moreover, we have $\sup\RGamma_{\pp}(R_\pp/\qq_\pp) =\sup\RGamma_{\pp R_\pp}(R_\pp/\qq_\pp)= \dim R_{\pp}/\qq_{\pp}$; see \cite[Theorem 3.5.7 and p.~413]{BH98}. Since $\Phi$ is a slice sp-filtration, we also have $\dim(R_{\pp}/\qq_{\pp}) = \height(\pp/\qq) \leq\f_\Phi(\pp)-\f_\Phi(\qq)$; see \cref{cd-func} and \cref{comparison}\cref{slice-vs-codim}. 
Then
\[\sup 
\Sigma^{\f_\Phi(\pp)}\RGamma_{\pp}(R_\pp/\qq_\pp)=\sup \RGamma_{\pp}(R_\pp/\qq_\pp)-\f_\Phi(\pp)\leq -\f_\Phi(\qq).\]
Therefore
\[\width_R(\qq, \Sigma^{\f_\Phi(\pp)}\RGamma_{\pp}R_{\pp})\geq \f_\Phi(\qq)>n\]
by \cref{width-compute}, as desired.
\end{proof}

The next theorem is also part of our main results (\cref{intro-tilt}\cref{intro-tilt-dc}).
Recall that a dualizing complex in our sense may have infinite injective dimension (see \cref{dc}). Moreover, if $D$ is a dualizing complex, then $\cd_D$ is a codimension function (see \cref{cd-func}), so $\Phi_{\cd_D}$ is a codimension filtration of $\Spec R$.

\begin{theorem}\label{tilt-dc}
	Let $R$ be commutative noetherian ring with a dualizing complex and $\Phi$ a non-degenerate sp-filtration of $\Spec R$. Then $T_{\Phi}$ is a tilting object in $\D(R)$ if and only if $\Phi$ is a codimension filtration of $\Spec R$.
\end{theorem}
\begin{proof}
	This follows from \cref{silting-cond}, \cref{slice-silting}, and \cref{P:dualizing}\cref{negat-ext-dc}.
\end{proof}

As mentioned in \cref{cdc-shortcut}, \cref{slice-silting,tilt-dc} can be proved in a simpler way if $\Phi$ is bounded and $R$ admits a classical dualizing complex.

\begin{remark}\label{CM-remark}
If $(R,\mm)$ is a $d$-dimensional Cohen--Macaulay local ring, then $\RGamma_\mm R\cong \Sigma^{-d} H^{d}_\mm R$ by \cite[Theorem 3.5.7]{BH98}, where $H^{d}_\mm R:= H^d \RGamma_\mm R$ is called the \emph{$d$-th local cohomology module of $R$ with respect to $\mm$}. 

Suppose that $R$ is a (possibly non-local) Cohen--Macaulay ring. Then  $R_\pp$ is a Cohen--Macaulay local ring  for each $\pp\in \Spec R$ by definition. 
Thus $\RGamma_\pp R_{\pp}\cong \Sigma^{-\height(\pp)}H_\pp^{\height(\pp)} R_\pp$ for every $\pp\in \Spec R$,
where $H_\pp^{\height(\pp)} R_\pp:=H^{\height(\pp)}\RGamma_\pp R_{\pp}= H^{\dim R_\pp}_{\pp R_\pp} R_{\pp}$.
Therefore, for every non-degenerate sp-filtration $\Phi$ of $\Spec R$, we have
\[T_\Phi =\bigoplus_{\pp \in \Spec R} \Sigma^{\f_\Phi(\pp)}\RGamma_\pp R_{\pp} \cong \bigoplus_{\pp \in \Spec R}\Sigma^{\f_\Phi(\pp)-\height(\pp)}H_\pp^{\height(\pp)} R_\pp\]
in $\D(R)$.
\end{remark}

The following is a corollary of \cref{slice-silting}. Notice that it does not assume the existence of a dualizing complex.

\begin{corollary}\label{tilt-CM}
	Let $R$ be a Cohen--Macaulay ring and $\Phi$ a codimension filtration. Then $T_{\Phi}$ is a tilting object in $\D(R)$.

	In particular, $T_{\Phi_{\height}} \cong  \bigoplus_{\pp \in \Spec R}H_\pp^{\height(\pp)} R_\pp$ is a tilting object.
\end{corollary}
\begin{proof}
First, $T:= T_{\Phi}$ is a silting object by \cref{slice-silting}. Thus it remains to show that $\Add(T) \in T\Perp{<0}$.
Write $\Spec R$ as a disjoint union $S_1 \sqcup \ldots \sqcup S_n$ of connected components. Letting 
$T_i:=\bigoplus_{\pp \in S_i}\Sigma^{\f_\Phi(\pp)}\RGamma_\pp R_\pp$, we can write $T=\bigoplus_{1\leq i\leq n} T_i$.
By assumption, $\f_{\Phi}$ is a codimension function on $\Spec R$, while the height function $\height$ is a codimension function on $\Spec R$ as well since $R$ is Cohen--Macaulay; see \cref{cd-func}.
Therefore, $\height - \f_\Phi$ is constant on each component $S_i$. Then we see from \cref{CM-remark} that each $T_i$ is isomorphic to a stalk complex, i.e., a complex concentrated in some degree. Since modules do not admit non-trivial negative self-extensions, it follows that $\Add(T_i) \subseteq T_i\Perp{<0}$. 
Moreover, we also have $\Add(T_i)\subseteq T_j\Perp{\ZZ}$ whenever $i \neq j$; see \cref{non-connected} below.
Hence we have $\Add(T) \in T\Perp{<0}$, as desired.
\end{proof}

\begin{remark}\label{non-connected}
For any commutative noetherian ring $R$, $\Spec R$ decomposes into a finite disjoint union $\Spec R_1\sqcup \Spec R_2 \sqcup\cdots \sqcup  \Spec R_n$ of connected components $\Spec R_i$ and we have an isomorphism $R \isoto \prod_{1\leq i\leq n} R_i$ of rings (cf. \cite[Chapter II, Exercise 2.19]{Har77}). This isomorphism yields an equivalence from $\Mod R$ is to the product category $\prod_{1\leq i\leq n}\Mod R_i$, and so $\D(R)$ is equivalent to the the product category $\prod_{1\leq i\leq n} \D(R_i)$ as well.
In particular, if  $X_i\in \D(R_i)$ and $X_j\in \D(R_j)$ and $i \neq j$, regarding $X_i$ and $X_j$ as objects in $\D(R)$, we have $\Add(X_i)\subseteq X_j\Perp{\ZZ}$. See also \cref{bi-loc}.

If $\cd$ and $\cd'$ are different codimension functions on $\Spec R$, then the two t-structures $(\cY_{\Phi_{\cd}}, \cY\Perp{0}_{\Phi_{\cd}})$ and $(\cY_{\Phi_{\cd'}}, \cY\Perp{0}_{\Phi_{\cd'}})$ in $\D(R)$ are different. However, if we restrict them to each derived category $\D(R_i)$, then this difference can be omitted up to shift, since $\cd-\cd'$ is constant on the connected component $\Spec R_i$. In other words, we have $\cY_{\Phi_{\cd}}\cap \D(R_i)= \Sigma^m(\cY_{\Phi_{\cd'}}\cap \D(R_i))=(\Sigma^m\cY_{\Phi_{\cd'}})\cap \D(R_i)$ for some integer $m$.
An important consequence is that 
the heart $\cY_{\Phi_{\cd}}\cap \Sigma (\cY\Perp{0}_{\Phi_{\cd}})$ of the t-structure $(\cY_{\Phi_{\cd}}, \cY\Perp{0}_{\Phi_{\cd}})$ is equivalent to the heart $\cY_{\Phi_{\cd'}}\cap \Sigma (\cY\Perp{0}_{\Phi_{\cd'}})$ of the other t-structure $(\cY_{\Phi_{\cd'}}, \cY\Perp{0}_{\Phi_{\cd'}})$. Therefore, if $R$ admits a codimension function $\cd$, then $(\cY_{\Phi_{\cd}}, \cY\Perp{0}_{\Phi_{\cd}})$ is the t-structure (induced by $T_{\Phi_\cd}$) such that its heart does not depend on the choice of $\cd$, up to equivalence. 

By the same token, although the silting objects $T_{\Phi_{\cd}}$ and $T_{\Phi_{\cd'}}$ are not equivalent in general, $\RHom_R(T_{\Phi_{\cd}},T_{\Phi_{\cd}}^{(\varkappa)})$ and $\RHom_R(T_{\Phi_{\cd'}},T_{\Phi_{\cd'}}^{(\varkappa)})$ are isomorphic as objects of $\D(R)$ for any cardinal $\varkappa$. Consequently, the silting object $T_{\Phi_\cd}$ is tilting if and only if $T_{\Phi_{\cd'}}$ is tilting. If this is the case, then $\End_{\D(R)}(T_{\Phi_{\cd}}) \cong \End_{\D(R)}(T_{\Phi_{\cd'}})$ as rings.

These facts can be also stated focusing on the cosilting side.
\end{remark}

As essentially observed in  the proof of \cref{tilt-CM}, once a module $M$ (over any ring $R$) is silting in the derived category $\D(R)$, then $M$ is tilting in $\D(R)$ (cf. \cite[Corollary 3.7]{Wei13}). We should remark that, although tilting modules in the sense of \cite{AHC01} are tilting objects in the derived category $\D(R)$ (\cref{tilt-module}), silting modules in the sense of \cite{AHMV16a} need not be silting objects in $\D(R)$; a \emph{silting module} is defined to be the cokernel of a morphism $P^{-1}\to P^0$ of projective modules such that the 2-term complex $(0\to P^{-1}\to P^0\to )$ is a silting object in $\D(R)$ (\cite[Theorem 4.9.]{AHMV16a}).

We will later observe that every 2-term silting complex over a commutative noetherian ring is a tilting object in the derived category (\cref{2-term}).

\begin{remark}\label{good-tilting}
Let $R$ be a Cohen--Macaulay ring of finite Krull dimension.
Then the $R$-module $T:=\bigoplus_{\pp \in \Spec R}H_\pp^{\height(\pp)} R_\pp$ is a bounded tilting object in $\D(R)$ by \cref{fin-proj-dim,tilt-CM}, and therefore it is a tilting module (\cref{tilt-module}).
In fact, the tilting module $T$ is \emph{good} in the sense of \cite[Definition 1.2]{BMT11}, that is, there is an exact sequence of the form $0 \to R \to T_0 \to T_1 \to \cdots \to T_n \to 0$ where $T_i$ is a direct summand of a finite coproduct of copies of $T$ for all $i=0,\ldots,n$.
Indeed, the \emph{Cousin complex} $C(R)$ for $R$ (in the sense of Sharp \cite{Sha69}) gives such an exact sequence (see \cite[Theorems 3.5 and 4.7]{Sha69} and \cite[Theorem]{Sha77b}), where $C(R)^{-1}=R$ and $C(R)^i\cong \bigoplus_{\height(\pp)=i}H_\pp^{\height(\pp)} R_\pp$ for $0\leq i\leq d$. See also \cite[Proposition 2.8.2(5)]{GN02}.

As a consequence, \cite[Theorem 2.2]{BMT11} is available for $T = \bigoplus_{\pp \in \Spec R}H_\pp^{\height(\pp)} R_\pp$. In particular, letting $S:=\End_{R}(T)$, the functor $\RHom_{R}(T,-):\D(R) \to \D(S)$ is fully faithful (\cite[Theorem 2.2(2)]{BMT11}), and the essential image of this functor is $\cE\Perp{0}\subseteq \D(S)$, where $\cE$ is the kernel of the functor $-\otimes^{\mathbf{L}}_ST:\D(S)\to \D(R)$ (\cite[Corollary 2.4]{BMT11}).

If $\dim R=1$, the projective dimension of $T$ is equal to one by \cref{fin-proj-dim} and \cite[Part II, Theorem 3.2.6]{RG71} (see also \cite[Theorem 4.9]{NT20}), so the result \cite[Theorem 1.1]{CX12} is available, producing a \emph{recollement} linking $\D(S)$ with $\D(R)$ and the derived module category of a suitable ring. In addition, we can compute $S=\End_{R}(T)$ explicitly as the lower triangular matrix ring in \cref{1-dim-example}.
More generally, when $R$ has any finite Krull dimension, 
there is a recollement linking $\D(S)$ with $\D(R)$ and the derived category of dg modules over a suitable dg algebra; see \cite[Proposition 5.2]{BP13} and the four paragraphs preceding it.

It is also possible to describe derived equivalences induced by good tilting modules by using techniques of \emph{contramodules} over a topological ring (\cite{PS21}).
The endomorphism ring $S=\End_{R}(T)$ admits a naturally induced linear topology called the \emph{finite topology}, and the category $\Ctra S$ of right contramodules over the topological ring $S$
is an abelian category; see \cite[\S 6.2 and \S 7.1]{PS21}.
There is a natural forgetful functor $\Ctra S \to \Mod S$, and $T$ being good implies that the functor induces a fully faithful functor $\D(\Ctra S)\to \D(S)$ between the unbounded derived categories. 
Moreover, the functor $\RHom_R(T,-):\D(R)\to \D(S)$ naturally restricts to a triangulated equivalence $\D(R) \isoto \D(\Ctra S)$, and this further restricts to an equivalence $\cH_T\isoto \Ctra S$, where $\cH_T$ is the heart of the tilting t-structure induced by $T$.
See \cite[Proposition 8.2]{PS21}; also cf. \cite[Corollary 6.3, Theorem 7.1, and Proposition 7.3]{PS21}. Note that the equivalence $\cH_T\isoto \Ctra S$ induces the triangulated equivalence $\D (\cH_T)\isoto\D(\Ctra S)$.

We now have the triangulated equivalence $\RHom_R(T,-):\D(R) \isoto \D(\Ctra S)$, while we also have the realization functor $\D(\cH_T) \to \D(R)$ (due to Virili) that is a triangulated equivalence as $T$ is a tilting module (\cref{subsec-real-func,subsec-der-eq}).
At least when restricted to the respective bounded derived categories, we can choose a realization functor $\D^\bd (\cH_T) \isoto \D^\bd (R)$ so that it is compatible with the other equivalence $\RHom_R(T,-):\D^\bd (R) \isoto \D^\bd (\Ctra S)$, under the identification $\D^\bd (\cH_T)\cong\D^\bd(\Ctra S)$.
This fact is explained in the recent preprint \cite{Hrb22}; see the proofs of \cite[Proposition 4.6 and Theorem 4.7]{Hrb22} in particular.
 
In \cite{Hrb22}, the above-mentioned results of \cite{PS21} have been generalized from good tilting modules to a large class of tilting complexes, which includes our tilting complexes available by \cref{cdc-shortcut,finite-theorem}; see \cite[\S\S 3--4]{Hrb22}.
\end{remark}

We close this section by attempting to explain how \cref{Tphi} is natural in terms of a more general approach that can yield silting objects under suitable assumptions. 
The reader may freely skip this part.
We also remark that the explanation below does not mean that our results in this section can be deduced from other existing results.

Let $R$ be any ring and assume that there is a decreasing chain of smashing subcategories of $\D(R)$ indexed by the integers: $\cdots \supseteq \cL_{n-1} \supseteq \cL_n \supseteq \cL_{n+1} \supseteq \cdots$.
Denote by $\lambda_n$ the localization functor $\D(R)\to \D(R)$ with $\Ker \lambda_n=\cL_{n}$ for each $n \in \ZZ$. Since application of $\lambda_n $ makes the canonical morphism $\Id_{\D(R)}\to \lambda_{n+1}$ invertible, we can identify $S_n:=\lambda_nR$ with $\lambda_n S_{n+1}$. Then we have the canonical morphism $S_{n+1} \to \lambda_n S_{n+1}= S_n$. Embed this morphism to a triangle $S_{n+1} \to S_n \to T_n \xrightarrow{+}$ in $\D(R)$, and 
put $T := \bigoplus_{n \in \ZZ}\Sigma^n T_n$. 
If, for each $n$, the morphism $R \to S_n$ can be regarded as a homological ring epimorphism, $S_n$ is of projective dimension at most one over $R$, and the chain of smashing subcategories is non-degenerate in the sense that $\bigcap_{n \in \ZZ}\cL_n = 0$ and $(\bigcup_{n \in \ZZ}\cL_n)\Perp{\ZZ} = 0$, then $T$ is a silting object in $\D(R)$; see \cite[Theorem 3.5]{AHS11} and \cite[Proposition 5.15]{AHH21}.

Now let $R$ be a commutative noetherian ring. 
Given a silting t-structure of finite type, we have its corresponding non-degenerate sp-filtration $\Phi$ by \cref{silt-fin-sp}.
Then we obtain a decreasing chain of smashing subcategories by putting $\cL_n:=\cL_{\Phi(n)}$, and this chain is non-degenerate in the above sense.
Assume $\Phi$ is a slice sp-filtration, and consider the object $T:=\bigoplus_{n \in \ZZ}\Sigma^n T_n$ of $\D(R)$ constructed as above. We can observe that $T_n \cong \Sigma \gamma_{\Phi(n)}\lambda_{\Phi(n+1)}R=\Sigma \gamma_{\Phi(n)}\lambda^{\Phi(n+1)^\cp}R\cong \Sigma \lambda^{\Phi(n+1)^\cp}\gamma_{\Phi(n)}R$; see \cref{stable-approx} and \cite[Proposition 6.1(3)]{BIK08}. 
Since $\dim(\Phi(n) \setminus \Phi(n+1)) \leq 0$ for each $n \in \ZZ$, it follows from \cref{slice-functor} that $\Sigma \lambda^{\Phi(n+1)^\cp}\gamma_{\Phi(n)}R\cong \bigoplus_{\pp \in U_n} \Sigma \RGamma_{\pp}R_\pp$, where  $U_n := \Phi(n) \setminus \Phi(n+1)$.
Therefore, we have $T_n \cong \bigoplus_{\pp \in U_n}\Sigma\RGamma_{\pp}R_\pp$. Since $\Phi$ is non-degenerate, each $\pp \in \Spec R$ belongs to $U_n$ if and only if $n+1 = \f_\Phi(\pp)$, and thus we arrive to the expression $T = \bigoplus_{\pp \in \Spec R}\Sigma^{\f_\Phi(\pp)}\RGamma_\pp R_\pp$ of \cref{Tphi}.

\section{Cosilting objects corresponding to slice sp-filtrations}\label{cosilting-section}

Let $R$ be a commutative noetherian ring and $\Phi$ be a slice sp-filtration of $\Spec R$. In this section, we give an explicit description of a cosilting object inducing the t-structure $(\cU_\Phi,\cV_\Phi)$.
Although \cref{(co)silt-dual} can translate \cref{slice-silting} into the cosilting case, this is not the best formulation for us. We more carefully dualize each summand of $T_\Phi = \bigoplus_{\pp \in \Spec R} \Sigma^{\f_\Phi(\pp)}\RGamma_{\pp} R_{\pp}
$. We start with an observation on local duality and Matlis duality.

We first treat a (commutative noetherian) local ring $R$ with maximal ideal $\mm$ and residue field $k$, and assume that $R$ admits a dualizing complex $D_{R}$ (see \cref{notation-dc}). Set $d:=\dim R$.
There is a canonical isomorphism 
\begin{align}\label{local-dual-1}
\RGamma_{\mm}R\cong \Hom_{R}(D_R,\Sigma^{-d}E_R({k}))
\end{align}
in $\D(R)$ by \emph{local duality} (\cite[Chapter V, Theorem 6.2]{Har66}).
Moreover, there is a canonical isomorphism $\Hom_R(E_R({k}),E_R({k}))\cong \widehat{R}:=\Lambda^\mm R$ by \emph{Matlis duality} (see \cite[Theorem A.31]{ILL+07}).
Since $D_{R}\in \D^\bd_\fg(R)$ and $D_{R}\otimes_R\widehat{R}\cong D_{\widehat{R}}$ in $\D(R)$ by \cref{dc-facts}\cref{complete-dc}, application of $\RHom_R(-,\Sigma^{-d}E_{R}(k))$ to  \cref{local-dual-1} yields a natural isomorphism \begin{align}\label{local-dual-2}
\RHom_R(\RGamma_{\mm}R, \Sigma^{-d}E_{R}(k)) \cong D_{\widehat{R}},
\end{align}
in $\D(R)$ by a standard argument (see \cite[(1.4)]{Fox79}). In fact, \cref{local-dual-2} holds without the existence of $D_R$, as explained below.

Let $(R,\mm,k)$ be a local ring that may not admit a dualizing complex. The $\mm$-adic completion $\widehat{R}$ of $R$ is a d-dimensional local ring with maximal ideal $\widehat{\mm}=\mm \widehat{R}$ and residue field $\widehat{R}/\widehat{\mm}\cong k$ (\cite[p.63]{Mat89}).
By \cite[Corollary 3.4.4]{Lip02}, there is a canonical isomorphism
\begin{equation}\label{base-change}
(\RGamma_{\mm}R)\LotimesR\widehat{R}\isoto \RGamma_{\widehat{\mm}}\widehat{R}
\end{equation}
in $\D(\widehat{R})$. Since $E_{R}(k)\cong E_{\widehat{R}}(k)=\Hom_{\widehat{R}}(\widehat{R}, E_{\widehat{R}}(k))$ in $\Mod R$ and $\Mod \widehat{R}$ (\cite[Theorem A.31]{ILL+07}), \cref{base-change} and tensor-hom adjunction yields the natural isomorphisms
\begin{equation}\label{tensor-hom}
\RHom_{R}(\RGamma_{\mm}R, E_{R}(k)) \cong \RHom_{\widehat{R}}((\RGamma_{\mm}R)\LotimesR\widehat{R}, E_{\widehat{R}}(k))\cong \RHom_{\widehat{R}}(\RGamma_{\widehat{\mm}}\widehat{R},E_{\widehat{R}}(k)),
\end{equation}
in $\D(R)$. Recall that $\widehat{R}$ admits a dualizing complex $D_{\widehat{R}}$, so the isomorphism \cref{local-dual-2} holds for $\widehat{R}$. Combining this fact with \cref{tensor-hom}, we see that \cref{local-dual-2} holds for $R$.

Let us next give a version of \cref{local-dual-1} that also works without a dualizing complex.
Regard $\RGamma_{\widehat{\mm}}\widehat{R}$ as an object in $\D(R)$ by the scalar restriction functor $\D(\widehat{R})\to \D(R)$. We have canonical isomorphisms
\begin{equation}\label{comp-loc}
\RGamma_{\mm}R\isoto \RGamma_{\mm}\widehat{R}\cong (\RGamma_{\mm}R)\LotimesR\widehat{R}\isoto \RGamma_{\widehat{\mm}}\widehat{R}
\end{equation}
in $\D(R)$, where the first is an isomorphism induced by the canonical morphism $R\to \widehat{R}$ (see \cite[Corollary 3.4.5.]{Lip02} or \cite[Proposition 3.5.4(d)]{BH98}), the second is given by \cref{Cech-comp}, and the third is \cref{base-change} sent to $\D(R)$. 
On the other hand, since $\widehat{R}$ admits a dualizing complex $D_{\widehat{R}}$, the isomorphism \cref{local-dual-1} holds for $\widehat{R}$, and we can send it to $\D(R)$. Combining this fact with \cref{comp-loc}, we obtain a canonical isomorphism
\begin{align}\label{local-dual-3}
\RGamma_{\mm}R\cong \Hom_{\widehat{R}}(D_{\widehat{R}},\Sigma^{-d}E_{\widehat{R}}({k}))
\end{align}
in $\D(R)$. 

We finally treat an arbitrary commutative noetherian ring $R$. Let $\pp\in \Spec R$ and write $\widehat{R_{\pp}}:=\Lambda^{\pp} R_{\pp}$, which is a complete local ring with maximal ideal $\pp \widehat{R_{\pp}}$ and residue field $\widehat{R_{\pp}}/\pp \widehat{R_{\pp}}\cong\kappa(\pp)$, where $\height(\pp)=\dim R_\pp=\dim \widehat{R_{\pp}}$.
Recall that $E_{R}(R/\pp)$ coincides with $E_{R_\pp}(\kappa(\pp))$ (\cite[Theorem 18.4(vi)]{Mat89}). 
Regard $\RGamma_{\pp R_\pp}R_\pp$ and 
$\RHom_{R_{\pp}}(\RGamma_{\pp R_\pp} R_{\pp}, E_{R_\pp}(\kappa(\pp)))$ as objects in $\D(R)$ by 
the (fully faithful) scalar restriction functor $\D(R_{\pp})\to \D(R)$. They naturally coincide with $\RGamma_{\pp} R_\pp$ and $\RHom_{R}(\RGamma_{\pp} R_{\pp}, E_{R}(R/\pp))$ respectively.
Therefore, we can deduce from \cref{local-dual-3,local-dual-2} that there are canonical isomorphisms in $\D(R)$:
\begin{equation}\label{local-dual-4}
\RHom_R(\RGamma_{\pp}{R_\pp}, E_{R}(R/\pp)) \cong \Sigma^{\height(\pp)}D_{\widehat{R_{\pp}}},
\end{equation}
\begin{equation}
\RGamma_{\pp}R_{\pp} \cong \Hom_{\widehat{R_{\pp}}}(\Sigma^{\height(\pp)}D_{\widehat{R_{\pp}}},E_{R}(R/\pp)).\label{local-dual-5}
\end{equation}
These isomorphisms lead us to the following definition.

\begin{definition}\label{CPhi}
Let $\Phi$ be a non-degenerate sp-filtration of $\Spec R$.
We define 
\[C_\Phi := \prod_{\pp \in \Spec R}\Sigma^{\height(\pp)-\f_\Phi(\pp)}D_{\widehat{R_\pp}},\]
which is an object in $\D(R)$.
\end{definition}

\begin{theorem}\label{slice-cosilt}
	Let $R$ be a commutative noetherian ring and $\Phi$ a slice sp-filtration of $\Spec R$. Then $C_\Phi$ is a cosilting object in $\D(R)$, which corresponds to the silting object $T_\Phi$ via \cref{(co)silt-dual}, so that $(\Perp{\leq 0}C_\Phi, \Perp{>0}C_\Phi)=(\cU_\Phi, \cV_\Phi)$.
	\end{theorem}
\begin{proof}
Let $E$ be an injective cogenerator in $\Mod R$. By \cref{slice-silting}, $T_\Phi$ is a silting object of finite type in $\D(R)$, and hence its dual $T_\Phi^+:=\RHom_{R}(T_\Phi, E)\in \D(R)$ is a cosilting object by \cref{(co)silt-dual}.
What remains to show is that $\Prod(T_\Phi^+) = \Prod(C_\Phi)$; see \cref{silt-equiv}.
Since $T_\Phi^+ \cong \prod_{\pp \in \Spec R}\Sigma^{-\f_\Phi(\pp)}(\RGamma_\pp R_\pp)^+$,
it suffices to prove that
$\Prod ((\RGamma_\pp R_\pp)^+)= \Prod (\Sigma^{\height(\pp)}D_{\widehat{R_\pp}})$ for each $\pp\in\Spec R$.
This follows from \cref{local-dual-4} and \cref{elementary} below. 
\end{proof}

\begin{remark}\label{elementary}
Let $\pp\in\Spec R$ and $X$ a complex of $R_\pp$-modules.
Then $X^+\cong \Hom_R(X,E)$ in $\D(R)$, and 
$\Hom_R(X,E)\cong \Hom_{R_\pp}(X,\Hom_R(R_\pp,E))$ in $\C(R)$.
Notice that $\Hom_R(R_\pp,E)$ is an injective cogenerator in $\Mod R_\pp$, while $E_R(R/\pp)\cong E_{R_\pp}(\kappa(\pp))$ is an injective cogenerator in $\Mod R_\pp$ as well (\cite[Theorem A.20 and Lemma A.27]{ILL+07}). Hence the product closure of $\Hom_R(R_\pp,E)$ and that of $E_R(R/\pp)$, taken in $\Mod R_{\pp}\subseteq \Mod R$, are the same. 
Then it is easily seen that the product closure of $X^+\cong \Hom_{R_\pp}(X,\Hom_R(R_\pp,E))$ and that of $\Hom_{R}(X,E_{R}(R/\pp))\cong \Hom_{R_\pp}(X,E_{R}(R/\pp))$, taken in $\D(R_{\pp})\subseteq \D(R)$, are the same.
\end{remark}

We do not know if the correspondence \cref{(co)silt-dual} can translate the tilting condition into the cotilting condition; indeed, there is no way to check these conditions by using sets of compact objects.
However, we can directly prove a dual result to \cref{tilt-dc}.

\begin{theorem}\label{dc-cotilt} 
Let $R$ be commutative noetherian ring with a dualizing complex and $\Phi$ a non-degenerate sp-filtration of $\Spec R$. Then $C_{\Phi}$ is a cotilting object in $\D(R)$ if and only if $\Phi$ is a codimension filtration of $\Spec R$.
\end{theorem}
\begin{proof}
Let $D$ be a dualizing complex for $R$. We may assume that $D$ is a bounded below complex of injective modules. 
Moreover, for each $\pp\in \Spec R$, we may interpret $D_{\widehat{R_\pp}}$ as a complex of injective $\widehat{R_\pp}$-modules concentrated in degrees from $0$ to $\height(\pp)$; see \cref{minimal}. 
By \cref{dc-facts}\cref{localize-dc} and \cref{complete-dc}, $\widehat{R_\pp}\otimes_RD=\widehat{R_\pp}\otimes_{R_\pp}D_\pp$ is also a dualizing complex for $\widehat{R_\pp}$. 
Recall that a dualizing complex for the local ring $\widehat{R_\pp}$ is unique up to shift and isomorphism in $\D(\widehat{R_\pp})$. In addition, since $\widehat{R_\pp}$ is flat over $R$, $\widehat{R_\pp}\otimes_RD$ consists of injective $R$-modules by \cite[Theorem 3.2.16]{EJ11}, and so does $D_{\widehat{R_\pp}}=\Hom_{\widehat{R_\pp}}(\widehat{R_\pp},D_{\widehat{R_\pp}})$ by tensor-hom adjunction. Note that $\inf(\widehat{R_\pp}\otimes_RD) = \inf(D_\pp) = \cd(\pp) - \height(\pp)$, where $\cd:=\cd_{D}$, and thus $\inf(\Sigma^{\cd(\pp)-\height(\pp)} \widehat{R_\pp}\otimes_RD) = 0$.
Hence we can check using \cref{dc-loc-coh} and \cref{comp-loc} that there is a quasi-isomorphism 
\begin{equation*}
\Sigma^{\cd(\pp)-\height(\pp)} \widehat{R_\pp}\otimes_RD\to D_{\widehat{R_\pp}}
\end{equation*}
in $\C(R)$.
Applying $\Hom_{\widehat{R_\pp}}(-,D_{\widehat{R_\pp}})$ to this, we get the natural quasi-isomorphism
\[\Hom_{\widehat{R_\pp}}(D_{\widehat{R_\pp}},D_{\widehat{R_\pp}})\to \Hom_{\widehat{R_\pp}}(\Sigma^{\cd(\pp)-\height(\pp)} \widehat{R_\pp}\otimes_RD,D_{\widehat{R_\pp}})\]
since $D_{\widehat{R_\pp}}$ is K-injective. Furthermore, we have the standard isomorphism in $\C(R)$:
\[\Hom_{\widehat{R_\pp}}(\Sigma^{\cd(\pp)-\height(\pp)} \widehat{R_\pp}\otimes_RD,D_{\widehat{R_\pp}})
\cong \Hom_{R}(\Sigma^{\cd(\pp)-\height(\pp)} D,D_{\widehat{R_\pp}}).\]
Composing the last two chain maps and the canonical quasi-isomorphism $\widehat{R_\pp}\to \Hom_{\widehat{R_\pp}}(D_{\widehat{R_\pp}},D_{\widehat{R_\pp}})$,
we have the quasi-isomorphism
\begin{equation}\label{susp-Hom}
\widehat{R_\pp}\to \Hom_{R}(\Sigma^{\cd(\pp)-\height(\pp)} D,D_{\widehat{R_\pp}}).
\end{equation}
The both sides are bounded above complexes of flat modules, so the mapping cone of this morphism is a pure acyclic complex of flat modules, i.e., an acyclic complex of flat modules whose cycle modules are flat. Let $\f:=\f_\Phi$. By \cref{susp-Hom}, we have the quasi-isomorphism $\Sigma^{\cd(\pp)-\f(\pp)}\widehat{R_\pp}\to \Hom_{R}(D, \Sigma^{\height(\pp)-\f(\pp)}D_{\widehat{R_\pp}})$ for each $\pp\in \Spec R$.
Hence we obtain the induced quasi-isomorphism
\begin{align*}
\prod_{\pp\in \Spec R} \Sigma^{\cd(\pp)-\f(\pp)}\widehat{R_\pp}\to
\prod_{\pp\in \Spec R}\Hom_R(D,\Sigma^{\height(\pp)-\f(\pp)}D_{\widehat{R_\pp}})
\end{align*}
whose mapping cone is a pure acyclic complex of flat modules.
Set $C:= C_\Phi$, and note that $C$ is a bounded below complex of injective $R$-modules (since $\f=\f_\Phi$ has a lower bound by non-degeneracy of $\Phi$; see the first paragraph of \cref{sp-filt}).
We can interpret the right-hand side of the above quasi-isomorphism as $\Hom_R(D,C)$, so that we have the quasi-isomorphism
\begin{align}\label{susp-Hom-2}
\prod_{\pp\in \Spec R} \Sigma^{\cd(\pp)-\f(\pp)}\widehat{R_\pp}\to \Hom_R(D,C),
\end{align}
which can be regarded as an isomorphism in the pure derived category $\D(\Flat R)$.

Now, since $C$ is K-injective, there is a canonical isomorphism 
\begin{equation}\label{K-inj}
\Hom_{\D(R)}(C^\varkappa, C)\cong \Hom_{\K(\Inj R)}(C^\varkappa, C)
\end{equation}
for any cardinal $\varkappa$. Furthermore, by \cref{K-inj}, \cref{NIK}\cref{Iyengar-Krause}, \cref{susp-Hom-2}, \cref{Murfet}, and \cref{bi-loc}, we have the following isomorphisms:
\begin{align*}
\Hom_{\D(R)}(C^\varkappa, C)
&\cong \Hom_{\K(\Inj R)}(C^\varkappa, C)\\
&\cong \Hom_{\D(\Flat R)}(\Hom_R(D,C^\varkappa), \Hom_R(D,C))\\
&\cong \Hom_{\D(\Flat R)}(\Hom_R(D,C)^\varkappa, \Hom_R(D,C))\\
&\cong \Hom_{\D(\Flat R)}\bigg(\prod_{\qq \in \Spec R}\Sigma^{\cd(\qq)-\f(\qq)}\widehat{R_\qq}^\varkappa, \prod_{\pp \in \Spec R}\Sigma^{\cd(\pp)-\f(\pp)}\widehat{R_\pp}\bigg)\\
&\cong \prod_{\pp \in \Spec R} \Hom_{\D(\Flat R)}\bigg(\prod_{\qq \in \Spec R}\Sigma^{\cd(\qq)-\f(\qq)}\widehat{R_\qq}^\varkappa, \Sigma^{\cd(\pp)-\f(\pp)}\widehat{R_\pp}\bigg)\\
&\cong \prod_{\pp \in \Spec R} \Hom_{\D(R)}\bigg(\prod_{\qq \in \Spec R}\Sigma^{\cd(\qq)-\f(\qq)}\widehat{R_\qq}^\varkappa, \Sigma^{\cd(\pp)-\f(\pp)}\widehat{R_\pp}\bigg)\\
&\cong \prod_{\pp \in \Spec R} \Hom_{\D(R)}\bigg(\prod_{\qq \in V(\pp)}\Sigma^{\cd(\qq)-\f(\qq)}\widehat{R_\qq}^\varkappa, \Sigma^{\cd(\pp)-\f(\pp)}\widehat{R_\pp}\bigg),
\end{align*}
where $\widehat{R_{\qq}}\in \cC^{\{\qq\}}= \cC^{U(\qq)}\cap \cC^{V(\qq)}$ for every $\qq\in \Spec R$ by \cref{typic-(co)loc,closed-iso}; note also that $\LLambda^{\pp}\widehat{R_\pp}\cong \widehat{R_\pp}$ (see \cite[p. 69]{Lip02}).

Let us prove the ``if'' part of the theorem.
Suppose that $\Phi$ is a codimension filtration. Then it is a slice sp-filtration, so $C_\Phi$ is a cosilting object in $\D(R)$ by \cref{slice-cosilt}. 
Further, $\f$ is a codimension function, so we have $\cd(\qq)-\cd(\pp)=\f(\qq) -\f(\pp)$ for any inclusion $\pp\subseteq \qq$ in $\Spec R$.
Thus
\begin{align*}
\Hom_{\D(R)}(C^\varkappa, \Sigma^nC)
&\cong \prod_{\pp \in \Spec R} \Hom_{\D(R)}\bigg(\prod_{\qq \in V(\pp)}\widehat{R_\qq}^\varkappa, \Sigma^n\widehat{R_\pp}\bigg)\\
&\cong \prod_{\pp \in \Spec R} \Hom_{\D(R)}\bigg(\prod_{\qq \in V(\pp)}\widehat{R_\qq}^\varkappa, \Sigma^n\RHom_R\big(E_R(R/\pp),E_R(R/\pp)\big)\bigg)\\
&\cong \prod_{\pp \in \Spec R} \Hom_{\D(R)}\bigg((\prod_{\qq \in \Spec R}\widehat{R_\qq}^\varkappa)\LotimesR E_R(R/\pp), \Sigma^nE_R(R/\pp)\bigg)=0
\end{align*}
whenever $n\neq 0$, where $\widehat{R_\pp}\cong \Hom_R\big(E_R(R/\pp),E_R(R/\pp)\big)$ by Matlis duality.
Therefore $\Prod(C)\subseteq \Perp{<0}C$, so $C$ is cotilting in $\D(R)$.
The ``only if'' part can be shown by modifying this argument (cf. the proof of \cref{P:dualizing}) or follows from \cref{necessity}\cref{cotilt-neces} below.
\end{proof}

We remark that $\QQ\oplus (\QQ/\ZZ)$ is a cotilting object of cofinite type in $\D(\ZZ)$ such that it induces the standard t-structure in $\D(\ZZ)$, but its character dual $\Hom_\ZZ (\QQ\oplus (\QQ/\ZZ), \QQ/\ZZ)$ is not a silting object in $\D(\ZZ)$; see \cite[\S 3.3]{AH13}. Thus, the assignment $C\mapsto C^+$ does not provide a formal way to obtain a silting object from  a cosilting object of finite type $C\in \D(R)$ for a commutative noetherian ring $R$. 
This is the reason why we first studied the silting case and deduced the cosilting case, even though the latter is often more tractable than the former when we work with large modules.

In fact, the $\ZZ$-module $\QQ\oplus (\QQ/\ZZ)$ is (not only cotilting but) tilting of finite type in $\D(\ZZ)$ (see \cref{ex-(co)tilt} below), and hence $\Hom_\ZZ (\QQ\oplus (\QQ/\ZZ), \QQ/\ZZ)$ is cotilting of cofinite type by \cref{(co)silt-dual}.

\begin{example}\label{ex-(co)tilt}
Suppose that $R$ is a Gorenstein ring with possibly infinite Krull dimension. Then $R$ itself is a dualizing complex; see \cite[Chapter V, Theorem 9.1]{Har66} and \cref{strong-pw}. The codimension function $\cd_R$ associated to the dualizing complex $R$ coincides with the height function (see \cite[Theorem 18.8]{Mat89}).
For the height filtration $\Phi_{\height}$, we have
\[T_{\Phi_{\height}}\cong \bigoplus_{\pp\in \Spec R} E_R(R/\pp)~\text{ and }~C_{\Phi_{\height}}\cong \prod_{\pp\in \Spec R} \widehat{R_\pp}\]
in $\D(R)$.
The first isomorphism follows from \cref{dc-loc-coh}, and the second follows from \cref{dc-facts}\cref{localize-dc} and \cref{complete-dc}. As we explain in the two remarks below, $T_{\Phi_{\height}}$ (resp. $C_{\Phi_{\height}}$) naturally induces the subcategory of all Gorenstein injective (resp. Gorenstein flat) modules as the tilting (resp. cotilting) class. This was well known in the case when the Krull dimension is finite (\cref{rmk-gorenstein}), but it turns out the finiteness assumption assumption is not necessary (\cref{rmk-gorenstein-2}).
\end{example}

\begin{remark}\label{rmk-gorenstein}
	Let $R$ be an \emph{Iwanaga-Gorenstein} ring, i.e., a (possibly non-commutative) noetherian ring with finite injective dimension as a left and right $R$-module. Let $0\to R\to I^0\to \cdots \to I^n\to 0$ be a minimal injective resolution of $R$ as a right $R$-module. It is well known that the coproduct $T:=\bigoplus_{0\leq i\leq n}I^i$ is an $n$-tilting right module; see \cite[Example 5.7]{AH13}.
	Thus $C:=\Hom_\ZZ(T,\QQ/\ZZ)$ is an $n$-cotilting left module (see \cref{of-finite-type,charact-dual}).
	The \emph{tilting class} $\cT:= \{M \in \Mod{R} \mid \Ext_R^i(T,M) = 0 ~\forall i>0\}$ induced by $T$ coincides with the class of \emph{Gorenstein injective} right modules (see \cite[Corollary 4.7]{Miy96} and \cite[Corollary 11.2.2]{EJ11}), and the \emph{cotilting class} $\cC:= \{M \in \Mod{R^{\op}} \mid \Ext_{R^{\op}}^i(M,C) = 0 ~\forall i>0\}$ induced by $C$ coincides with the class of \emph{Gorenstein flat} left modules (see \cite[Theorem 10.3.8]{EJ11}).
	If $R$ is commutative (so $R$ is a Gorenstein ring of Krull dimension $n$), then the tilting module $T$ is nothing but $\bigoplus_{\pp\in \Spec R} E_R(R/\pp)$.
\end{remark}

\begin{remark}\label{rmk-gorenstein-2}	
Let $R$ be a (commutative) Gorenstein ring of infinite Krull dimension, and set $T:=\bigoplus_{\pp\in \Spec R} E_R(R/\pp)$ and $C:=\prod_{\pp\in \Spec R} \widehat{R_\pp}$.
	It follows from  \cref{tilt-CM} (resp. \cref{cotilt-CM}) and \cref{ex-(co)tilt} that $T$ (resp. $C$) is tilting (resp. cotilting) in $\D(R)$, but this has infinite projective (resp. injective) dimension. Therefore $T$ (resp. $C$) is not a tilting (resp. cotilting) module in the classical sense (see \cref{tilt-module}). However, $T$ (resp. $C$) is an \emph{$\infty$-tilting} (resp. \emph{$\infty$-cotilting}) object in $\Mod R$ in the sense of Positselski and {\v{S}}{\v{t}}ov\'{\i}\v{c}ek (see \cite[\S 2]{PS19}). 
	More precisely, setting $\cT:=T\Perp{>0}\cap \Mod R$ and $\cC:=\Perp{>0}C\cap \Mod R$, we can show that $(T, \cT)$ (resp. $(C, \cC)$) is an \emph{$\infty$-tilting pair} (resp. an \emph{$\infty$-cotilting pair}) in $\Mod R$ (see \cite[\S 3]{PS19}).
	Let us briefly explain how this fact can be checked. 
	
	We first show that $\cT$ (resp. $\cC$) coincides with $\mathcal{GI}$ (resp. $\mathcal{GF}$), the class of Gorenstein injective (resp. Gorenstein flat) $R$-modules. Recall from \cite[Theorem 3.6]{Hol04} that $M\in \Mod R$ is Gorenstein flat if and only if $M^+:=\Hom_{\ZZ}(M,\QQ/\ZZ)$ is Gorenstein injective. On the other hand, it follows from \cite[Theorem 1.1]{CK18} that if $M$ is Gorenstein injective, then $M^+$ is Gorenstein flat. Hence \cite[Lemma 1]{EI15} implies that $M$ is Gorenstein injective if and only if $M^+$ is Gorenstein flat. 
	As a consequence, $(\mathcal{GI}, \mathcal{GF})$ is a \emph{duality pair} over $R$ in the sense of \cite[Definition 2.1]{HJ09} (see also \cite[Theorem 3.7]{Hol04}), so that $\mathcal{GI}$ is closed under pure submodules by \cite[Theorem 3.1]{HJ09}.  
	Further, $\mathcal{GI}$ is closed under taking direct products and direct limits by \cite[Theorem 1.1]{CK18} and \cite[Theorem 1]{EI15}. Thus $\mathcal{GI}$ is a definable class (see \cite[\S 3.4.1]{Pre09}), and so the class $\{M\in \Mod R \mid M^+\in \mathcal{GI}\}$ is definable as well (cf. \cite[Remark 2.4]{AHH21} and \cite[Theorem
	1.3.15]{Pre09}). Moreover, $\mathcal{GF}=\{M\in \Mod R \mid M^+\in \mathcal{GI}\}$ and  $\mathcal{GI}=\{M\in \Mod R \mid M^+\in \mathcal{GF}\}$ by the above observation.
	Since $\mathcal{GI}$ (resp. $\mathcal{GF})$ is definable, a standard argument (cf. \cite[Lemma 3.5]{HHZ21}) shows that $M \in \Mod R$ is Gorenstein injective (resp. Gorenstein flat) if and only if $M_\mm$ is Gorenstein injective (resp. Gorenstein flat) in $\Mod R_\mm$ for all maximal ideals $\mm$ of $R$. Hence, noting that $E_R(R/\pp)\cong E_{R_\pp}(\kappa(\pp))$ and $C=\prod_{\pp\in \Spec R} \widehat{R_\pp}$ is equivalent to $T^+\cong \Hom_R(T, \Hom_\ZZ (R,\QQ/\ZZ))$ as cotilting objects in $\D(R)$ (see the proof of \cref{slice-cosilt} and \cref{ex-(co)tilt}), we can deduce from \cite[Theorem 10.3.8]{EJ11} and tensor-hom adjunction that $M \in \Mod R$ belongs to $\cC$ if and only if $M$ is Gorenstein flat, i.e., $\cC=\mathcal{GF}$. 
	We have $\cT=\{M\in \Mod R\mid M^+\in \cC \}$ by \cref{compact-gen-t}, \cref{comp-gen-co-t}, and \cite[Lemma 2.5]{AHH21}, so it follows that $\cT=\{M\in \Mod R\mid M^+\in \mathcal{GF}\}=\mathcal{GI}$.
	
	Now, $(T,\cT)=(T,\mathcal{GI})$ is an $\infty$-tilting pair in $\Mod R$ by the above argument and \cite[Example 6.4]{PS19}. Moreover, we can directly show that $(C,\cC)=(C,\mathcal{GF})$ satisfies ($\mathrm{i}^*$)-($\mathrm{iii}^*$) in \cite[p.~311]{PS19}, and $(C,\mathcal{GF})$ also satisfies ($\mathrm{iv}^*$) by \cite[Theorem 3.7]{Hol04}.
	Thus, to see that $(C,\cC)$ is an $\infty$-cotilting pair in $\Mod R$, it remains to show that $(C,\cC)$ satisfies $(\text{v}^*)$.
	For this purpose, it suffices to show that, for any $M\in \mathcal{GF}$, there exists an injection $M\to N$ such that $N\in \Prod(C)\cap \Mod R$. This follows because there exists an injection from $M$ to a flat $R$-module $F$ by definition and the pure-injective envelope of $F$ is a direct summand of some object in $\Prod(C)\cap \Mod R$; see \cite[Theorem 5.3.28 and Corollary 6.7.2]{EJ11} and \cite[Remark 7.8]{KN22}.
\end{remark}

\begin{remark}\label{canonical-module}
By local duality \cref{local-dual-1}, any commutative noetherian local ring admitting a dualizing complex $D_R$ (\cref{notation-dc}) is Cohen--Macaulay if and only if $D_R$ is quasi-isomorphic to an $R$-module, which is called a \emph{canonical module} of $R$ and denoted by $\omega_R$; cf. \cite[Chapter V, Proposition 3.4]{Har66} and \cite[\S 3.3]{BH98}. The canonical module is unique up to isomorphism.

Suppose that $R$ is a (possibly non-local) Cohen--Macaulay ring. Then $R_\pp$ is a Cohen--Macaulay local ring  for each $\pp\in \Spec R$ by definition, and so is $\widehat{R_\pp}$ (\cite[Theorem 17.5]{Mat89}). Hence $D_{\widehat{R_\pp}}$ is quasi-isomorphic to $\omega_{\widehat{R_\pp}}$.
Therefore
\[C_\Phi \cong \prod_{\pp \in \Spec R}\Sigma^{\height(\pp)-\f_\Phi(\pp)}\omega_{\widehat{R_\pp}}\]
in $\D(R)$ for every non-degenerate sp-filtration $\Phi$ of $\Spec R$.
\end{remark}

The following is a corollary of \cref{slice-cosilt}.

\begin{corollary}\label{cotilt-CM}
Let $R$ be a Cohen--Macaulay ring and $\Phi$ a codimension filtration. Then $C_{\Phi}$ is a cotilting object in $\D(R)$.

In particular, $C_{\Phi_{\height}} \cong \prod_{\pp \in \Spec R}\omega_{\widehat{R_\pp}}$ is a cotilting object in $\D(R)$.
\end{corollary}

\begin{proof}
By \cref{slice-cosilt}, $C_{\Phi}$ is a cosilting object in $\D(R)$, so it only remains to check that $\Prod(C_{\Phi})\subseteq \Perp{<0}C_{\Phi}$. This follows from a parallel argument to the proof of \cref{tilt-CM}.
\end{proof}

\begin{remark}\label{bd-comp-inj}
Let $R$ be a commutative noetherian ring. For each $\pp\in \Spec R$, $D_{\widehat{R_\pp}}$ is quasi-isomorphic to a complex of injective $\widehat{R_\pp}$-modules concentrated in degrees between $0$ and $\height(\pp)$ (see \cref{minimal}). 
The canonical ring homomorphism $R\to R_\pp\to \widehat{R_\pp}$ is flat, and then it easily follows that every injective $\widehat{R_\pp}$-module is injective over $R$.
Therefore, if $\dim R<\infty$ and $\Phi$ is a bounded sp-filtration of $\Spec R$, then $C_{\Phi}$ is isomorphic in $\D(R)$ to a bounded complex of injective $R$-modules. In particular, $C_{\Phi_{\height}}$  isomorphic in $\D(R)$ to a bounded complex of injective $R$-modules concentrated in degrees from $0$ to $\dim R$.
\end{remark}

\begin{remark}\label{charact-seq}
Let $R$ be a commutative noetherian ring and $n$ be an integer with $n \geq 1$. By \cite[Theorem 4.2]{AHPST14}, there are bijections among the equivalence classes of $n$-tilting modules, those of $n$-cotilting modules, and the \emph{characteristic sequences} of length $n$ (\cite[Definition 3.1]{AHPST14}).
Each characteristic sequence $\mathbf{Y}:=(Y_1,\ldots, Y_n)$ constitutes the sp-filtration $\Phi_{\mathbf{Y}}$ with $\Phi_{\mathbf{Y}}(-1)=\Spec R$, $\Phi_{\mathbf{Y}}(0)=Y_{1}$,\ldots, $\Phi_{\mathbf{Y}}(n-1)=Y_{n}$, and $\Phi_{\mathbf{Y}}(n)=\emptyset$. By \cite[Corollary 3.5]{AHS14} and \cref{(co)silt-dual}, an $n$-tilting (resp. $n$-cotilting) module corresponding to the characteristic sequence $\mathbf{Y}$ induces the t-structure in $\D(R)$ corresponding to the sp-filtration $\Phi_{\mathbf{Y}}$ via \cref{silt-fin-sp} (resp. \cref{cosilt-cof-sp}); see also the first sentence of \cite[\S 4]{AHPST14}.
In general, an sp-filtration $\Phi$ comes from a characteristic sequence $\mathbf{Y}=(Y_1,\ldots, Y_n)$ if and only if the order-preserving function $\f_{\Phi}:\Spec R\to \ZZ\cup\{\infty, -\infty\}$ corestricts to $\f_{\Phi}:\Spec R \to \{0,1,\ldots,n\}$ and $\f_{\Phi} \leq \grade$ (see \cref{hight-codim}). This fact essentially follows from \cite[Lemma 4.1]{AHPST14}; see also \cite[Corollary 3.5]{AHS14} and \cite[Definition 1.1 and Theorem 1.2]{DT15}.

As mentioned in \cite[p. 3499]{AHPST14}, the classification \cite[Theorem 4.2]{AHPST14} does not concretely describe (co)tilting modules corresponding to each characteristic sequence. However, if $R$ is a Cohen--Macaulay ring of finite Krull dimension $d$, then our \cref{tilt-CM,cotilt-CM} can explicitly give the $d$-tilting module $\bigoplus_{\pp\in \Spec R}H^{\height(\pp)}_\pp R_\pp$ and the $d$-cotilting module $ \prod_{\pp \in \Spec R}\omega_{\widehat{R_\pp}}$. The tilting (resp. cotilting) module is a new example as far as $d>0$ (resp. $d>1$) and it corresponds to the characteristic sequence $(Y_1,\ldots, Y_{d})$ given by $Y_i:=\Phi_{\height}(i-1)$. The cotilting module for the case $d=1$ is available in \cite[Example 6.12]{CS20}, which also covers non-affine Noetherian schemes.

In fact, given a commutative noetherian ring $R$ and a slice sp-filtration $\Phi$ of $\Spec R$, it directly follows from the definitions of depth and Cohen--Macaulay rings that the silting object $T_{\Phi}$ is isomorphic in $\D(R)$ to an $R$-module if and only if $R$ is Cohen--Macaulay and $\Phi$ is the height filtration.
By \cref{(co)silt-dual}, the same fact holds true for the cosilting object $C_{\Phi}$. As a consequence, $T_{\Phi}$ (resp. $C_{\Phi}$) is equivalent to a tilting (resp. cotilting) module if and only if $R$ is Cohen--Macaulay of finite Krull dimension and $\Phi$ is the height filtration.
\end{remark}

\section{Endomorphism rings of (co)tilting objects induced by codimension functions}\label{end-section}
In this section, we study the endomorphism rings of the tilting object and the cotilting object corresponding to a codimension filtration.

Let $R$ be a commutative noetherian ring.
Let $\pp\in \Spec R$ and $\varkappa$ be any cardinal.
There are standard isomorphisms
\[\RHom_{R}(\RGamma_{\pp}R_\pp,\RGamma_{\pp}R_\pp^{(\varkappa)})\cong \RHom_{R}(\RGamma_{\pp}R_\pp,R_\pp^{(\varkappa)})\cong \RHom_{R_\pp}(\RGamma_{\pp R_\pp}R_\pp,R_\pp^{(\varkappa)}).\]
in $\D(R)$, and the last object coincides with $\LLambda^{\pp R_\pp} (R_\pp^{(\varkappa)})= \Lambda^{\pp R_\pp} (R_\pp^{(\varkappa)})=\Lambda^{\pp} (R_\pp^{(\varkappa)})$ by \cref{GM-dual}.

Hence we have a natural isomorphism
\begin{equation}\label{p-complete}
\RHom_{R}(\RGamma_{\pp}R_\pp,\RGamma_{\pp}R_\pp^{(\varkappa)})\cong  \widehat{R_\pp^{(\varkappa)}}.
\end{equation}
in $\D(R)$. 
Since the right-hand side has neither positive nor negative cohomologies, \cref{p-complete} can be naturally identified with an isomorphism
\begin{align}\label{p-complete-2}
\Hom_{\D(R)}(\RGamma_{\pp}R_\pp,\RGamma_{\pp}R_\pp^{(\varkappa)})\cong \widehat{R_\pp^{(\varkappa)}}
\end{align}
in $\D(R)$. If $\varkappa=1$, \cref{p-complete-2} is of the form \begin{equation}\label{ring-hom}
\End_{\D(R)}(\RGamma_{\pp}R_\pp)\cong \widehat{R_\pp},
\end{equation}
which is an isomorphism of rings.
For later use, let us verify the last fact in more detail.
It suffices to treat the case when $R$ is local and $\pp$ is the maximal ideal $\mm$. 
Consider the canonical ring homomorphism $f:R\to \End_{\D(R)}(\RGamma_{\mm}R)$, and compose it with the standard isomorphisms 
\begin{align*}
\End_{\D(R)}(\RGamma_{\mm}R)&=\Hom_{\D(R)}(\RGamma_{\mm}R,\RGamma_{\mm}R)\\
&\cong\Hom_{\D(R)}(\RGamma_{\mm}R,R)\\
&\cong\RHom_{R}(\RGamma_{\mm}R,R)
\end{align*}
in $\D(R)$.
Then we obtain a natural morphism $R\to \RHom_{R}(\RGamma_{\mm}R,R)$, which is in fact induced by the canonical morphism $\RGamma_{\mm}R\to R$ under the identification $R=\RHom_R(R,R)$.
Thus the morphism $R\to  \RHom_{R}(\RGamma_{\mm}R,R)$ coincides with the canonical morphism $R\to  \LLambda^\mm R=\widehat{R}$ via \cref{GM-dual} applied to $I:=\mm$ (see,  e.g.,  \cite[Proposition 7.5.16]{SS18}).
We see from this observation that the composition of $f:R\to \End_{\D(R)}(\RGamma_{\mm}R)$ and the isomorphism $g:\End_{\D(R)}(\RGamma_\mm R)\isoto \widehat{R}$ of $R$-modules given by \cref{ring-hom} is just the completion map $c:R\to \widehat{R}$, i.e., $c=gf$.
Moreover, we can naturally extend $f$ to a ring homomorphism $h:\widehat{R}\to \End_{\D(R)}(\RGamma_{\mm}R_\mm)\cong \End_{\D(R)}(\RGamma_{\mm}\widehat{R})$, i.e., $f=hc$; see \cref{comp-loc}.
Since $c=gf$, we have $c=ghc$. Hence 
$\Id_{\widehat{R}}=\Lambda^\mm (c)= \Lambda^\mm(gh) \Lambda^\mm(c)= \Lambda^\mm(gh)$, but clearly $\Lambda^\mm(g)=g$ and $\Lambda^\mm(h)=h$ because the domains and the codomains are $\mm$-adically complete. Consequently $\Id_{\widehat{R}}=gh$.
This means that the ring homomorphism $h:\widehat{R}\to \End_{\D(R)}(\RGamma_{\mm}R_\mm)$ is the inverse map to the isomorphism 
$g:\End_{\D(R)}(\RGamma_\mm R)\isoto \widehat{R}$ of $R$-modules. 
Therefore $g$ is an isomorphism of rings, and so is \cref{ring-hom}, as desired.

Let $\pp\subsetneq \qq$ be a chain in $\Spec R$ and $\varkappa$ a cardinal.
Although computing $\RHom_{R}(\RGamma_{\pp}R_\pp,\RGamma_{\qq}R_\qq^{(\varkappa)})$ is not easy in general, we can explicitly do this if the chain is saturated by 
the theorem below, which is formulated in a more general situation.
To prove it, we recall the natural isomorphism
 \begin{equation}\label{lambda-zero}
\lambda^{\{\pp\}}\cong \LLambda^\pp (-\LotimesR R_\pp).
\end{equation}
of functors $\D(R)\to \D(R)$ for any $\pp\in \Spec R$; see \cite[Corollary 3.7]{NY18b}.

\begin{theorem}\label{minimal-case}
Let $\pp\in \Spec R$ and $W\subseteq \min (V(\pp)\sm \{\pp\})$. Let $\varkappa$ be any cardinal. Then there is a natural isomorphism
\[\RHom_{R}(\RGamma_{\pp}R_{\pp},\bigoplus_{\qq\in W} \Sigma\RGamma_\qq R_\qq^{(\varkappa)})\cong \Lambda^\pp ((\prod_{\qq\in W} \widehat{R_\qq^{(\varkappa)}})_\pp).\]
in $\D(R)$.
\end{theorem}

\begin{proof}
Put $V_0:=V(\pp)\sm \{\pp\}$, and consider the approximation triangle 
\[\gamma_{V_0} R\to \gamma_{V(\pp)}R\to \lambda^{V_0^\cp}\gamma_{V(\pp)}R\xrightarrow{+},\]
where $\gamma_{V_0} \gamma_{V(\pp)}R\cong \gamma_{V_0} R$ by \cref{trivial-iso} and $\lambda^{V_0^\cp}\gamma_{V(\pp)}R\cong \RGamma_{\pp} R_\pp$ by  \cref{slice-functor}.
Putting $X:=\bigoplus_{\qq\in W} \RGamma_\qq R_\qq^{(\varkappa)}$ and applying $\RHom_R(-,\Sigma X)$ to the above triangle, we obtain the following triangle
\begin{equation}\label{minimal-1}
\RHom_R(\Sigma\gamma_{V_0} R,\Sigma X)\to\RHom_R(\RGamma_{\pp} R_\pp, \Sigma X) \to \RHom_R(\gamma_{V(\pp)}R,\Sigma X)\xrightarrow{+}.
\end{equation}
Let $\overline{W}^{\mathrm{g}}$ be the generalization closure of $W$ (i.e., the smallest generalization closed subset containing $W$).
Since $X\in \cL_{\overline{W}^{\mathrm{g}}}=\cC^{\overline{W}^{\mathrm{g}}}$ 
by \cref{bi-loc}, there is a canonical isomorphism
\begin{align}\label{minimal-2}
\RHom_R(\gamma_{V_0} R, X)\cong \RHom_R(\lambda^{\overline{W}^{\mathrm{g}}}\gamma_{V_0} R, X).
\end{align}
Consider the approximation triangle 
\[\gamma_{(\overline{W}^{\mathrm{g}})^\cp}\gamma_{V_0} R \to \gamma_{V_0} R\to \lambda^{\overline{W}^{\mathrm{g}}}\gamma_{V_0} R\xrightarrow{+},\]
where 
$\gamma_{(\overline{W}^{\mathrm{g}})^\cp}\gamma_{V_0} R\cong \gamma_{(\overline{W}^{\mathrm{g}})^\cp\cap V_0} R$; see, e.g., \cite[Proposition 6.1]{BIK08}. Since the composition $\gamma_{(\overline{W}^{\mathrm{g}})^\cp\cap V_0} R\cong \gamma_{(\overline{W}^{\mathrm{g}})^\cp}\gamma_{V_0} R\to \gamma_{V_0} R$ is the canonical one, letting $V_1:=(\overline{W}^{\mathrm{g}})^\cp\cap V_0$, we may regard the above triangle as the approximation triangle 
\[\gamma_{V_1} R \to \gamma_{V_0} R\to \lambda^{V_1^\cp}\gamma_{V_0} R\xrightarrow{+}.\]
Note that $V_0\backslash V_1=V_0\cap V_1^\cp=V_0\cap \overline{W}^{\mathrm{g}}=W$ and $\dim W\leq 0$ by assumption.
Hence $\lambda^{V_1^\cp}\gamma_{V_0} R\cong \bigoplus_{\qq\in W} \RGamma_\qq R_\qq$ by \cref{slice-functor}.
This fact and \cref{minimal-2} yield the natural isomorphisms
\begin{align}\label{minimal-2-1}
\RHom_R(\gamma_{V_0} R, X)&\cong \RHom_R(\bigoplus_{\qq\in W} \RGamma_\qq R_\qq, X)\cong\prod_{\qq\in W}\RHom_R(\RGamma_\qq R_\qq, X).
\end{align}
Since $\dim W\leq 0$, we have
\begin{equation}\label{minimal-3-1}
\RHom_R(\RGamma_\qq R_\qq, X)\cong \RHom_R(\RGamma_\qq R_\qq, \RGamma_\qq R_\qq^{(\varkappa)})
\end{equation}
for each $\qq\in W$ by \cref{bi-loc} and the second paragraph of \cref{RHom-TPhi}.
Therefore
\begin{align}\label{minimal-3}\RHom_R(\gamma_{V_0} R, X)\cong \prod_{\qq\in W}\RHom_R(\RGamma_\qq R_\qq, \RGamma_\qq R_\qq^{(\varkappa)})\cong \prod_{\qq\in W}\widehat{R_\qq^{(\varkappa)}},
\end{align}
where the first isomorphism holds by \cref{minimal-2,minimal-2-1,minimal-3-1} and the second holds by \cref{p-complete}.

Now, we remark that
\[\lambda^{\{\pp\}}\RHom_R(\gamma_{V(\pp)}R,\Sigma X)\cong \lambda^{\{\pp\}} \lambda^{V(\pp)}\Sigma X\cong \lambda^{\{\pp\}} \Sigma X \cong\LLambda^{\pp}(\bigoplus_{\qq\in W} \Sigma (\RGamma_\qq R_\qq^{(\varkappa)})_\pp) =0
\]
by \cref{trivial-iso}, \cref{(co)smash-iso}, and \cref{lambda-zero}, where the last equality holds because $(\RGamma_\qq R_\qq^{(\varkappa)})_\pp=0$ for every $\qq\in W$.
Hence, applying $\lambda^{\{\pp\}}$ to \cref{minimal-1}, we obtain the isomorphism
\begin{equation}\label{minimal-4}
\lambda^{\{\pp\}}\RHom_R(\Sigma \gamma_{V_0} R, \Sigma X)\isoto\lambda^{\{\pp\}}\RHom_R(\RGamma_{\pp} R_\pp,\Sigma X)
\end{equation}
induced by the canonical morphism 
$\RGamma_{\pp} R_\pp\cong \lambda^{V_0^\cp}\gamma_{V(\pp)}R \to \Sigma \gamma_{V_0}R$.
Notice that 
\begin{equation}\label{minimal-5}
\lambda^{\{\pp\}}\RHom_R(\RGamma_{\pp} R_\pp,\Sigma X)\cong\LLambda^{\pp}\RHom_R(\RGamma_{\pp} R_\pp,\Sigma X)\cong\RHom_R(\RGamma_{\pp} R_\pp,\Sigma X),
\end{equation}
where the first isomorphism follows from \cref{lambda-zero} along with the isomorphism $\RHom_R(\RGamma_{\pp} R_\pp,\Sigma X)\cong \RHom_R(\RGamma_{\pp} R_\pp,\Sigma X)_\pp$
and the second follows from \cref{RHom-TPhi}.
Thus combining \cref{minimal-3,minimal-4,minimal-5}, we obtain a natural isomorphism 
\[\RHom_R(\RGamma_{\pp} R_\pp,\Sigma X)\cong \Lambda^\pp ((\prod_{\qq\in W} \widehat{R_\qq^{(\varkappa)}})_\pp).\]
\end{proof}

\begin{notation}\label{adelic}
For each $\pp\in \Spec R$, we write $\Ad^{\pp}=\Lambda^{\pp}(-\otimes_RR_\pp)$, which is a functor $\Mod R\to \Mod R$. Moreover, for a subset $W\subseteq \Spec R$ with $\dim W\leq 0$, we write $\Ad^{W}=\prod_{\pp\in W}\Ad^\pp$.
If $W=\emptyset$, then $\Ad^{W}$ is the zero functor.
\end{notation}

\begin{corollary}\label{X(n)-X(n+1)}
Assume that $R$ admits a codimension function $\cd$.
For each $n\in \ZZ$, set $W_n:=\{\pp\in \Spec R \mid \cd(\pp)=n\}$,  
$T(n):=\bigoplus_{\pp\in W_n}\Sigma^n\RGamma_{\pp}R_\pp$, and $Y(n):=T(n)\oplus T(n+1)$. 
Let $\varkappa$ be any cardinal. Then there is an isomorphism
\begin{equation}\label{X(n)}
\RHom_{R}(Y(n), Y(n)^{(\varkappa)})\cong 
\Ad^{W_n}(R^{(\varkappa)})\oplus \Ad^{W_{n+1}}(R^{(\varkappa)})\oplus \Ad^{W_{n}}\Ad^{W_{n+1}}(R^{(\varkappa)})
\end{equation}
in $\D(R)$ for any $n\in \ZZ$, and this is induced by the following collection of natural isomorphisms:
\begin{equation}\label{T(i)-T(j)}
\RHom_{R}(T(i), T(j)^{(\varkappa)})
\cong 
\left\{
\begin{array}{ll}
\Ad^{W_i}(R^{(\varkappa)})& (n\leq i=j\leq n+1)\\
\Ad^{W_{n}}\Ad^{W_{n+1}}(R^{(\varkappa)})& (i=n,\ j=n+1)\\
0&(i=n+1,\ j=n)
\end{array}
\right.
\end{equation}
\end{corollary}

\begin{proof}
Let $i$ and $j$ be integers. 
There are natural isomorphisms
\begin{align*}
\RHom_{R}(T(i), T(j)^{(\varkappa)})
&\cong \prod_{\pp\in W_i}\RHom_{R}(\RGamma_{\pp}R_\pp,\bigoplus_{\qq\in W_j}\Sigma^{j-i}\RGamma_{\qq}R_\qq^{(\varkappa)})\\
&\cong \prod_{\pp\in W_i}\RHom_{R}(\RGamma_{\pp}R_\pp,\bigoplus_{\qq\in V(\pp)\cap W_j}\Sigma^{j-i}\RGamma_{\pp}R_\pp^{(\varkappa)})
\end{align*}
in $\D(R)$, where the second isomorphism follows from \cref{bi-loc}.

If $i=j$ and $\pp\in W_i$, then $V(\pp)\cap W_i=\{\pp\}$, so we have
\[\RHom_{R}(T(i), T(i)^{(\varkappa)})
\cong \prod_{\pp\in W_i}\RHom_{R}(\RGamma_{\pp}R_\pp,\RGamma_{\pp}R_\pp^{(\varkappa)})
\cong \prod_{\pp\in W_i}\Ad^{\pp}(R^{(\varkappa)})
= \Ad^{W_i} (R^{(\varkappa)})\]
by \cref{p-complete}. 
This shows the first isomorphism of \cref{T(i)-T(j)}.

If $(i,j)=(n,n+1)$ for some integer $n$, then we have 
\begin{align*}
\RHom_{R}(\RGamma_{\pp}R_\pp,\bigoplus_{\qq\in V(\pp)\cap W_{n+1}}\Sigma \RGamma_{\qq}R_\qq^{(\varkappa)})
&\cong \Lambda^\pp ((\prod_{\qq\in V(\pp)\cap W_{n+1}} \widehat{R_\qq^{(\varkappa)}})_\pp)\\
&\cong \Lambda^\pp((\prod_{\qq\in W_{n+1}} \widehat{R_\qq^{(\varkappa)}})_\pp)\\
&=\Ad^{\pp} \Ad^{W_{n+1}}(R^{(\varkappa)})
\end{align*}
for every $\pp\in W_n$ by \cref{minimal-case}, where the second isomorphism holds because \[(\prod_{\qq\in W_{n+1}\sm V(\pp)} \widehat{R_\qq^{(\varkappa)}})_\pp\otimes_RR/\pp^t\cong (\prod_{\qq\in W_{n+1}\sm V(\pp)} (\widehat{R_\qq^{(\varkappa)}}\otimes_RR/\pp^t))_\pp=0\] for all $t\geq 1$. Hence there is a natural isomorphism
\[\RHom_{R}(T(n), T(n+1)^{(\varkappa)})\cong \prod_{\pp\in W_n} \Ad^{\pp}\Ad^{W_{n+1}}(R^{(\varkappa)})=\Ad^{W_n} \Ad^{W_{n+1}}(R^{(\varkappa)})\]
in $\D(R)$, from which the second isomorphism of  \cref{T(i)-T(j)} follows.

If $i>j$ and $\pp\in W_i$, then $V(\pp)\cap W_j=\emptyset$, so 
\[\RHom_{R}(T(i), T(j)^{(\varkappa)})= 0\]
in $\D(R)$.
This shows the third case of \cref{T(i)-T(j)}.
\end{proof}

\begin{corollary}\label{End-tilt}
Let $R$, $Y(n)$, and $\varkappa$ be as in \cref{X(n)-X(n+1)}.
Then there is a natural isomorphism
\[\RHom_R(Y(n),Y(n)^{(\varkappa)}) \cong \Hom_{\D(R)}(Y(n), Y(n)^{(\varkappa)})\]
in $\D(R)$ for any $n\in \ZZ$. The right-hand side is a flat $R$-module.
\end{corollary}

\begin{proof}
This follows from \cref{X(n)-X(n+1)}.
See also \cref{flatcotorsion} below for the second statement.
\end{proof}

If $\dim R\leq 1$, a codimension function $\cd$ on $\Spec R$ always exists; for example, choose $\cd$ as the height function $\height :\Spec R\to \ZZ$ or the function defined by the assignment $\pp\mapsto \dim R-\dim R/\pp$.

\begin{theorem}\label{1-dim}
	Assume that the Krull dimension of $R$ is at most one. Let $\Phi$ be a codimension filtration $\Phi$ of $\Spec R$. Then $T_\Phi$ is a tilting object in $\D(R)$.
\end{theorem}
\begin{proof}
By \cref{slice-silting}, $T_\Phi$ is a silting object in $\D(R)$. To show that $T_\Phi$ is tilting, we may assume that $\Phi$ is the height filtration; see the last paragraph of \cref{non-connected}. Then the theorem follows from \cref{End-tilt}.
\end{proof}

For the proof of the next theorem, we make a remark.
\begin{remark}\label{A-ring}
Let $R$ be any commutative noetherian ring and $W\subseteq \Spec R$ with $\dim W\leq 0$.
For each $\qq\in W$, let $h_\qq: \widehat{R_\qq}\isoto\End_{\D(R)}(\RGamma_\qq R_\qq)$ be  the canonical ring isomorphism given by \cref{ring-hom}.
Let $n$ be an integer and $X:=\bigoplus_{\pp\in W}\Sigma^n\RGamma_{\pp}R_\pp$.
Take any element $(r_\qq)_{\qq\in W}$  of $\prod_{\qq\in W}\widehat{R_\qq}=\Ad^{W} R$. 
Then we have the coproduct $\bigoplus_{\qq\in W}  \Sigma^n h_\qq(r_\qq): X\to X$ of morphisms $ \Sigma^n h_\qq(r_\qq):\Sigma^n\RGamma_\qq R_\qq \to \Sigma^n\RGamma_\qq R_\qq$. 
Since 
\[\End_{\D(R)}(X)\cong \prod_{\qq\in W}\End_{\D(R)}(\RGamma_\qq R_\qq)\] 
by \cref{bi-loc},
the assignment $(r_\qq)_{\qq\in W}\mapsto \bigoplus_{\qq\in W} \Sigma^n  h_\qq(r_\qq)$ induces a ring isomorphism
\begin{equation}\label{A-ring-1}
\Ad^{W} R=\prod_{\qq\in W}\widehat{R_\qq}\isoto \End_{\D(R)}(X).
\end{equation} 
This gives the first isomorphism of \cref{T(i)-T(j)} if $W=W_n$, $X=T(n)$, and $\varkappa=1$. Note that if $W=\emptyset$, then the both sides of \cref{A-ring-1} are the zero ring.

Let $U\subseteq \Spec R$ with $\dim U\leq 0$. 
Let us observe that $\Ad^U \Ad^WR$ has an $(\Ad^{W}R,\Ad^UR)$-bimodule structure. 
Regard the (commutative) $R$-algebra $\Ad^{W}R$ as a left $\Ad^{W}R$-module.
Each element $a\in \Ad^{W}R$ induces an $R$-homomorphism $\Ad^{W}R\xrightarrow{a \cdot} \Ad^{W}R$, and it is sent to an $R$-homomorphism by the functor $\Ad^{U}: \Mod R\to \Mod R$.
Hence we can define the $a$-action on $\Ad^U\Ad^{W}R$ as the induced map $\Ad^{U}(a \cdot): \Ad^U \Ad^WR\to \Ad^U \Ad^WR$, so $\Ad^U \Ad^WR$ is a left $\Ad^WR$-module. 
On the other hand, we may interpret $\Ad^\pp$ as a functor $\Mod R\to \Mod R_\pp\to \Mod \widehat{R_\pp}$  for each $\pp\in \Spec R$ (see, e.g., \cite[Remark A.11]{KN22}). Thus $\Ad^{U}=\prod_{\pp\in U}\Ad^\pp$ is a functor $\Mod R\to \Mod \Ad^{U}R$. In particular, the (commutative) ring $\Ad^{U}R$ naturally acts on $\Ad^{U}\Ad^{W}R$ from the right.
As a consequence, $\Ad^{U}\Ad^{W}R$ has an $(\Ad^{W}R,\Ad^{U}R)$-bimodule structure. Indeed, given $a\in \Ad^{W}R$, $b\in \Ad^{U}R$, and $x\in \Ad^{U}\Ad^{W}R$,  we have $(ax)b=(\Ad^{U}(a \cdot))(x)b=(\Ad^{U}(a \cdot))(xb)=a(xb)$, where the second equality holds because $\Ad^{U}(a \cdot): \Ad^U \Ad^WR\to \Ad^U \Ad^WR$ is a morphism in $\Mod \Ad^UR$, i.e., an $\Ad^{U}R$-homomorphism.
\end{remark}

\begin{theorem}\label{End-tilt-2}
Let $R$, $W_n$, and $Y(n)$ be as in \cref{X(n)-X(n+1)}.
For each $n\in \ZZ$, there is a natural isomorphism
\begin{equation*}
\End_{\D(R)}(Y(n))\cong 
\begin{pmatrix}
\Ad^{W_n}R & 0 \vspace{2mm}\\ 
	\Ad^{W_{n}}\Ad^{W_{n+1}}R & \Ad^{W_{n+1}}R
	\end{pmatrix}
\end{equation*}
of rings.
\end{theorem}

\begin{proof}
By definition and the third case of \cref{T(i)-T(j)}, we have the ring isomorphism
\begin{align*}
\End_{\D(R)}(Y(n))
&\cong 
\begin{pmatrix}
\End_{\D(R)}(T(n)) & 0 \vspace{2mm}\\ 
\Hom_{\D(R)}(T(n),T(n+1)) & \End_{\D(R)}(T(n+1))
	\end{pmatrix}.
\end{align*}
Moreover, by the first and the second cases of \cref{T(i)-T(j)}, we have the following isomorphisms of $R$-modules:
\begin{align}
\Ad^{W_{n}} R&\cong \End_{\D(R)}(T(n)),\label{ring-iso-1}\\
\Ad^{W_{n+1}} R&\cong \End_{\D(R)}(T(n+1)),\label{ring-iso-2}\\
\Ad^{W_n}\Ad^{W_{n+1}} R&\cong \Hom_{\D(R)}(T(n),T(n+1)).\label{AWAW-Hom}
\end{align}
Here, \cref{ring-iso-1,ring-iso-2} are isomorphisms of rings by \cref{A-ring-1}.
Thus it remains to show that, through \cref{ring-iso-1}, \cref{ring-iso-2} and \cref{AWAW-Hom}, 
the $(\Ad^{W_{n+1}}R,\Ad^{W_{n}}R)$-bimodule structure on $\Ad^{W_n}\Ad^{W_{n+1}} R$ agrees with the $(\End_{\D(R)}(T(n+1)),\End_{\D(R)}(T(n)))$-module structure on $\Hom_{\D(R)}(T(n),T(n+1))$.

Take $\pp\in W_n$. Let $W:=V(\pp) \cap W_{n+1}$ and $X:=\bigoplus_{\qq\in W} \RGamma_{\qq}R_\qq$.
By the proof of \cref{X(n)-X(n+1)}, we have a natural isomorphism
\begin{equation}\label{Pf-X(n)-X(n+1)}
\RHom_{R}(\RGamma_{\pp}R_\pp,\Sigma X)\cong\RHom_{R}(\Sigma^{n}\RGamma_{\pp}R_\pp,T(n+1)).
\end{equation}
in $\D(R)$.
Let $V_0:=V(\pp)\sm \{\pp\}$. Since $W\subseteq \min V_0$, we have a natural morphism
\begin{equation}\label{1st-minimal-1}
\RHom_R(\Sigma\gamma_{V_0} R,\Sigma X)\to\RHom_R(\RGamma_{\pp} R_\pp, \Sigma X),
\end{equation}
which is the first morphism of \cref{minimal-1}.
Moreover, there is a natural isomorphism
\begin{equation}\label{AW-RHom-1}
\Ad^W R\cong \RHom_R(\Sigma\gamma_{V_0} R,\Sigma X)
\end{equation}
by \cref{minimal-3}. By construction, the left $\Ad^W R$-action on $\Ad^W R$ agrees with the left $\End_{\D(R)}(X)$-action on $\RHom_R(\Sigma\gamma_{V_0} R,\Sigma X)$, through \cref{AW-RHom-1} and the ring isomorphism $\Ad^W R\isoto \End_{\D(R)}(X)$ given by  \cref{A-ring-1}.
Composing \cref{Pf-X(n)-X(n+1)} \cref{1st-minimal-1}, and \cref{AW-RHom-1}, we obtain a natural morphism
\begin{align}\label{AW-RHom-2}
\Ad^W R\to \RHom_{R}(\Sigma^{n}\RGamma_{\pp}R_\pp,T(n+1))
\end{align}
in $\D(R)$.
We remark that, by \cref{minimal-4,Pf-X(n)-X(n+1)},  the morphism \cref{1st-minimal-1} becomes an isomorphism upon application of $\lambda^{\{\pp\}}$.
Furthermore, we have $\lambda^{\{\pp\}}\Ad^{W} R\cong \Ad^{\pp}\Ad^{W} R\cong  \Ad^{\pp}\Ad^{W_{n+1}} R$ in $\D(R)$; see \cref{lambda-zero} and the proof of \cref{X(n)-X(n+1)}. Thus,
application of $\lambda^{\{\pp\}}$ to \cref{AW-RHom-2} induces
a natural isomorphism
\begin{equation}\label{AW-RHom-3}
\Ad^{\pp}\Ad^{W_{n+1}} R\isoto \RHom_{R}(\Sigma^{n}\RGamma_{\pp}R_\pp,T(n+1))
\end{equation}
in $\D(R)$; see also \cref{minimal-5}. By construction, the left $\Ad^{W_{n+1}} R$-action on $\Ad^{\pp}\Ad^{W_{n+1}} R$ agrees with the left $\End_{\D(R)}(T(n+1))$-action on $\RHom_{R}(\Sigma^{n}\RGamma_{\pp}R_\pp,T(n+1))$ through \cref{AW-RHom-3,ring-iso-2}.
Taking the product of \cref{AW-RHom-3} for all $\pp\in W_n$, and using \cref{T(i)-T(j)} (or \cref{End-tilt}), we obtain the following natural isomorphisms
\begin{align*}
\Ad^{W_n}\Ad^{W_{n+1}} R
&\cong 
\prod_{\pp\in W_n}\RHom_{R}(\Sigma^{n}\RGamma_{\pp}R_\pp,T(n+1))\\
&\cong \RHom_{R}(T(n),T(n+1))\\
&\cong \Hom_{\D(R)}(T(n),T(n+1))
\end{align*}
in $\D(R)$, whose composition is nothing but \cref{AWAW-Hom}.
By construction, the left $\Ad^{W_{n+1}} R$-action on $\Ad^{W_n}\Ad^{W_{n+1}} R$ agrees with the left $\End_{\D(R)}(T(n+1))$-action on $\Hom_{\D(R)}(T(n),T(n+1))$ through \cref{AWAW-Hom,ring-iso-2}.

We next show that the right $\Ad^{W_{n}}R$-action on $\Ad^{W_n}\Ad^{W_{n+1}}R$ agrees with the right $\End_{\D(R)}(T(n))$-action on $\Hom_{\D(R)}(T(n),T(n+1))$ through \cref{ring-iso-1,AWAW-Hom}.
Since the isomorphism \cref{AWAW-Hom} restricts to the isomorphism 
\begin{align}\label{restricted-iso}
\Ad^\pp \Ad^{W_{n+1}}R\cong \Hom_{\D(R)}(\Sigma^{n}\RGamma_{\pp}R_\pp,T(n+1))
\end{align}
of $R_\pp$-modules for each $\pp\in W_n$, it suffices to show that the right $\widehat{R_\pp}$-action on $\Ad^\pp \Ad^{W_{n+1}}R$ agrees with the right $\End_{\D(R)}(\Sigma^{n}\RGamma_{\pp}R_\pp)$-action on $M:=\Hom_{\D(R)}(\Sigma^{n}\RGamma_{\pp}R_\pp,T(n+1))$ through \cref{ring-hom,restricted-iso}. 

The right $\widehat{R_\pp}$-action on $\Ad^\pp \Ad^{W_{n+1}}R$ induces an $\widehat{R_\pp}$-action on $M$ through \cref{restricted-iso}.
Further, $M$ has the right $\End_{\D(R)}(\Sigma^{n}\RGamma_\pp R_\pp)$-action, which can be regarded as an $\widehat{R_\pp}$-action through \cref{ring-hom}.
Consequently, $M$ has two (right) $\widehat{R_\pp}$-module structures extending its $R_\pp$-module structure.
Then the two $\widehat{R_\pp}$-module structures coincide by \cite[Proposition A.15]{KN22}.
This argument shows that the right $\widehat{R_\pp}$-action on $\Ad^\pp \Ad^{W_{n+1}}R$ agrees with the right $\End_{\D(R)}(\Sigma^{n}\RGamma_\pp R_\pp)$-action on $M$ through \cref{ring-hom,restricted-iso}.
\end{proof}

We denote by $\mSpec R$ the set of maximal ideals of $R$.
The following partly generalizes \cite[Example 7.5 and Example 7.6(2)]{CX12} and \cite[Example 8.4]{PS21}. 

\begin{example}\label{1-dim-example}
In the setting of \cref{1-dim}, we can explicitly compute the endomorphism ring $\End_{\D(R)}(T_\Phi)$  of the tilting object $T_\Phi$. 
By \cref{non-connected}, we may assume that $\Phi$ is the sp-filtration $\Phi_{\cd}$ for the codimension function $\cd:\Spec R\to \ZZ$ given by $\pp\mapsto \dim R-\dim R/\pp$ (see \cref{order-preserv}).
If $\dim R=0$, then $T_\Phi \cong \prod_{\pp\in \Spec R} R_\pp \cong R$, so $\End_{\D(R)}(T_\Phi)\cong R$.
Hence assume $\dim R=1$. 
In the notation of \cref{X(n)-X(n+1)}, we have $T_\Phi = T(0) \oplus T(1) = Y(0)$, where $W_1=\mSpec R$ and $W_0=\Spec R\sm W_1$. Then $\Ad^{W_0}=\prod_{\pp \in W_0}(- \otimes_R R_\pp)$, and this functor coincides with $-\otimes_R S^{-1}R$, where $S=R\sm \bigcup_{\pp\in W_0} \pp$; see, e.g., \cite[\S8, Theorem 8.15 and Remark 2]{Mat89}. By \cref{End-tilt-2}, we obtain the description \cref{intro-End} of $\End_{\D(R)}(T_\Phi)$. If $R=\ZZ$, $(\prod_{\mm\in \mSpec \ZZ} \widehat{\ZZ_{\mm}})\otimes_{\ZZ}\QQ$ is known as the \emph{ring of finite adeles}.

The endomorphism ring $\End_{\D(R)}(T_\Phi)$ does not depend (up to isomorphism) on the choice of a codimension filtration, but if the height filtration is chosen as $\Phi$, then one would apply \cref{End-tilt-2} (and \cref{X(n)-X(n+1)}) to $\cd:=\height$. In this case, and if $\Spec R$ is not connected, \cref{End-tilt-2} could give a slightly different description of the endomorphism ring:
\[\End_{\D(R)}(T_{\Phi})\cong
	\begin{pmatrix}
		S^{-1}R &  0 \vspace{2mm}\\
		\displaystyle{(\prod_{\height(\mm)=1} \widehat{R_\mm})\otimes_RS^{-1}R} & \displaystyle{\prod_{\height(\mm)=1} \widehat{R_\mm}}
		\end{pmatrix},\]
where $S$ is the complement of the set of height-zero prime ideals.
If $R$ is Cohen--Macaulay, $S$ is the set of nonzero divisors (see \cite[Theorem 6.1(ii) and Theorem 17.3(i)]{Mat89}), so $S^{-1}R$ is the total quotient ring.
\end{example}

The next result allows us to partly dualize the tilting condition to the cotilting condition.

\begin{proposition}\label{vanishing}
	Let $\pp \subseteq \qq$ be prime ideals of $R$ and $\varkappa$ a nonzero cardinal. For each $i \in \ZZ$, the following conditions are equivalent:
\begin{enumerate}[label=(\arabic*), font=\normalfont]
\item \label{k-lc-vanish} $H^{i}\RHom_R(\RGamma_{\pp}R_\pp,\RGamma_{\qq}R_\qq^{(\varkappa)})=0$
\item \label{lc-vanish} $H^{i}\RHom_R(\RGamma_{\pp}R_\pp,\RGamma_{\qq}R_\qq)=0$
\item \label{lc-dc-vanish}  $H^{\height(\qq)-i}\RGamma_\pp(R_\pp\LotimesR D_{\widehat{R_\qq}})=0$
\item\label{k-dc-vanish}
 $H^{i-\height(\qq)+\height(\pp)}\RHom_R(D_{\widehat{R_\qq}}^\varkappa,D_{\widehat{R_\pp}}) = 0$
\item\label{dc-vanish} $H^{i-\height(\qq)+\height(\pp)}\RHom_R(D_{\widehat{R_\qq}},D_{\widehat{R_\pp}}) = 0$
\end{enumerate}
\end{proposition}
\begin{proof}
For any complex $X$ of $R_\pp$-modules and any complex $Y$ of $R_\qq$-modules, there is a standard isomorphism $\RHom_{R}(X,Y)\cong \RHom_{R_\qq}(X,Y)$. Hence it is enough to treat the case when $(R,\mm, k)$ is a local ring, $\qq$ is the maximal ideal $\mm$, and $\pp$ is a prime ideal of the local ring $R$. Then $\widehat{R_\qq}=\widehat{R}$.
In addition, we may regard $D_{\widehat{R}}$ as a bounded complex of finitely generated $\widehat{R}$-modules.
Noting that $E_{R}(k)\cong E_{\widehat{R}}(k)$ is an injective cogenerator in $\Mod \widehat{R}$ (\cite[Lemma A.27 and Theorem A.31]{ILL+07}), we have
\begin{align*}
&H^{i}\RHom_{R}(\RGamma_{\pp}R_\pp,\RGamma_\mm R^{(\varkappa)})=0\\
&\Leftrightarrow H^{i-\height(\mm)}\RHom_{R}(\RGamma_{\pp}R_\pp, \Sigma^{\height(\mm)}\RGamma_\mm R^{(\varkappa)})=0\\
&\Leftrightarrow H^{i-\height(\mm)}\RHom_{R}(\RGamma_{\pp}R_\pp, \Hom_{\widehat{R}}(D_{\widehat{R}}, E_R(k))^{(\varkappa)})=0\\
&\Leftrightarrow H^{i-\height(\mm)}\RHom_{R}(\RGamma_{\pp}R_\pp, \RHom_{\widehat{R}}(D_{\widehat{R}}, E_R(k)^{(\varkappa)}))=0\\
&\Leftrightarrow H^{i-\height(\mm)}\RHom_{\widehat{R}}(\RGamma_{\pp}R_\pp\LotimesR D_{\widehat{R}}, E_R(k)^{(\varkappa)})=0\\
&\Leftrightarrow H^{\height(\mm)-i}(\RGamma_{\pp}R_\pp\LotimesR D_{\widehat{R}})=0,
\end{align*}
where the second bi-implication follows from \cref{local-dual-5} as $\RGamma_\mm R^{(\varkappa)}\cong (\RGamma_\mm R)^{(\varkappa)}$, the third follows as $D_{\widehat{R}}$ consists of finitely generated $\widehat{R}$-modules and $E_R(k)^{(\varkappa)}$ is injective, and the fourth follows 
by the adjunction
\[\begin{tikzcd}[column sep=normal, row sep=small]
\D(R) \arrow[shift left, bend left=10]{r}[above]{- \LotimesR D_{\widehat{R}}} & \arrow[shift left, bend left=10]{l}{\RHom_{\widehat{R}}(D_{\widehat{R}},-)}  \D(\widehat{R}).\end{tikzcd}
\]
Then it is seen that the equivalences among \cref{lc-dc-vanish,lc-vanish,k-lc-vanish} hold.

Next, we recall that, for any finitely generated $\widehat{R}$-module $N$, the functor $N\otimes_{\widehat{R}}-: \Mod \widehat{R}\to \Mod \widehat{R}$ commutes with products because $\widehat{R}$ is noetherian. 
Thus we have a natural isomorphism 
\begin{align}\label{iso-kappa-product} D_{\widehat{R}}\otimes_{\widehat{R}} \widehat{R}^\varkappa\isoto D_{\widehat{R}}^\varkappa
\end{align}
in $\C(\widehat{R})$. Using this, \cref{local-dual-4}, and \cref{Cech-comp},  
we obtain natural isomorphisms
\begin{align*}
\RHom_R(D_{\widehat{R}}^\varkappa,D_{\widehat{R_\pp}}) 
&\cong \RHom_R(D_{\widehat{R}} \otimes_{\widehat{R}} \widehat{R}^\varkappa,D_{\widehat{R_\pp}})\\
&\cong \RHom_R(D_{\widehat{R}} \otimes_{\widehat{R}} \widehat{R}^\varkappa,\Sigma^{-\height(\pp)}\RHom_R(\RGamma_{\pp}{R_\pp}, E_{R}(R/\pp)))\\
&\cong 
\Sigma^{-\height(\pp)}\RHom_R(\RGamma_{\pp}{R_\pp}\LotimesR(D_{\widehat{R}}\otimes_{\widehat{R}}  \widehat{R}^\varkappa),E_{R}(R/\pp))\\
&\cong \Sigma^{-\height(\pp)}\Hom_R(\check{C}(\pp)_\pp\otimes_RD_{\widehat{R}}\otimes_{\widehat{R}}  \widehat{R}^\varkappa,E_{R}(R/\pp))
\end{align*}
in $\D(R)$, where $E_{R}(R/\pp)\cong E_{R_\pp}(\kappa(\pp))$ is an injective cogenerator in $\Mod R_\pp$ and 
the last $\Hom_{R}$ can be replaced by $\Hom_{R_\pp}$. 
Therefore we have
\begin{align*}
&H^{i-\height(\mm)+\height(\pp)}\RHom_R(D_{\widehat{R}}^\varkappa,D_{\widehat{R_\pp}})=0\\
&\Leftrightarrow H^{i-\height(\mm)}\Hom_{R_\pp}(\check{C}(\pp)_\pp\otimes_RD_{\widehat{R}}\otimes_{\widehat{R}}  \widehat{R}^\varkappa,E_{R}(R/\pp))=0\\ 
&\Leftrightarrow H^{\height(\mm)-i}(\check{C}(\pp)_\pp\otimes_RD_{\widehat{R}}\otimes_{\widehat{R}}  \widehat{R}^\varkappa)=0\\
&\Leftrightarrow H^{\height(\mm)-i}(\check{C}(\pp)_\pp\otimes_RD_{\widehat{R}})=0\\
& \Leftrightarrow H^{\height(\mm)-i}(\RGamma_\pp R_\pp\LotimesR D_{\widehat{R}})=0,
\end{align*}
where the third bi-implication follows since $\widehat{R}^\varkappa$ is a faithfully flat $\widehat{R}$-module.
This shows that the equivalences among \cref{k-dc-vanish,dc-vanish,lc-dc-vanish} hold.
\end{proof}

\begin{proposition}\label{necessity}
Let $\Phi$ be a non-degenerate sp-filtration of $\Spec R$. 
\begin{enumerate}[label=(\arabic*), font=\normalfont]
\item \label{silt-necessity} If $T_\Phi$ is a silting object in $\D(R)$, then $\Phi$ is a slice sp-filtration. If $T_\Phi$ is a tilting object in $\D(R)$, then $\Phi$ is a codimension filtration, and so $R$ admits a codimension function.

\item \label{cotilt-neces} If $C_\Phi$ is a cosilting object in $\D(R)$, then $\Phi$ is a slice sp-filtration. If $C_\Phi$ is a cotilting object in $\D(R)$, then $\Phi$ is a codimension filtration, and so $R$ admits a codimension function.
\end{enumerate}
\end{proposition}
\begin{proof}
\cref{silt-necessity}: Let $T:=T_{\Phi}$ and $\f:=\f_{\Phi}$, where  $\f_\Phi: \Spec R\to \ZZ$ is the order-preserving function corresponding to $\Phi$ (\cref{order-preserv}). 
Suppose that $T$ is silting. Then $T \in T\Perp{>0}$ by \cref{silting-cond}.
Take any saturated chain $\pp\subsetneq \qq$ in $\Spec R$ and consider 
the direct summand 
\[X:=\RHom_R(\Sigma^{\f(\pp)}\RGamma_{\pp}R_\pp,\Sigma^{\f(\qq)}\RGamma_{\qq}R_\qq)\]
of $\RHom_R(T,T)$.
If $\f(\pp) = \f(\qq)$, then \cref{minimal-case} implies that
$X \cong \Sigma^{-1} \Lambda^{\pp}(\widehat{R_\qq}\otimes_R R_\pp)\neq 0$ in $\D(R)$.
Thus $H^1\RHom_R(T,T)\neq 0$, but this contradicts 
that $T \in T\Perp{>0}$.
Therefore $\f=\f_\Phi$ is strictly increasing, that is, $\Phi$ is a slice sp-filtration; see \cref{slice-strict}.

Next, suppose that $T$ is tilting. Then $\f$ is strictly increasing by the above argument.
If $\f$ is not a codimension function, there is a saturated chain $\pp\subsetneq \qq$ such that $\f(\qq)-\f(\pp)>1$. Putting $n:=\f(\qq)-\f(\pp)$, we have
\[\RHom_R(\Sigma^{\f(\pp)}\RGamma_{\pp}R_\pp,\Sigma^{\f(\qq)} \RGamma_{\qq}R_\qq)=\RHom_R(\RGamma_{\pp}R_\pp,\Sigma^{n} \RGamma_{\qq}R_\qq)\cong \Sigma^{n-1} \Lambda^{\pp}(\widehat{R_\qq}\otimes_RR_\pp)\] by \cref{minimal-case}. Hence $H^{-n+1}\RHom_R(T,T)\neq 0$ and $-n+1<0$, but this contradicts that $T \in T\Perp{<0}$. Thus $\f=\f_{\Phi}$ is a codimension function, that is, $\Phi$ is a codimension filtration.

\cref{cotilt-neces}: Let $C:=C_{\Phi}$ and $\f:=\f_{\Phi}$. 
Suppose that $C$ is cosilting. 
Then $C\in \Perp{>0}C$; see the first paragraph of \cref{silt-equiv}.
Take any saturated chain $\pp\subsetneq \qq$ in $\Spec R$.
By \cref{minimal-case}, 
$\RHom_R(\RGamma_{\pp}R_\pp,\RGamma_{\qq}R_\qq)\cong \Sigma^{-1} \Lambda^{\pp}(\widehat{R_\qq}\otimes_RR_\pp)$ in $\D(R)$, and hence $H^{1}\RHom_R(\RGamma_{\pp}R_\pp,\RGamma_{\qq}R_\qq)\neq 0$.
It then follows from \cref{vanishing}, applied to $i=1$, that $H^{1-\height(\qq)+\height(\pp)}\RHom_R( D_{\widehat{R_\qq}}, D_{\widehat{R_\pp}}) \neq 0,$ i.e., $H^{1}\RHom_R(\Sigma^{\height(\qq)}
D_{\widehat{R_\qq}},\Sigma^{\height(\pp)}D_{\widehat{R_\pp}})\neq 0$. If $\f(\pp)=\f(\qq)$, $\RHom_R(C_{\Phi},C_{\Phi})$ contains
\[\RHom_R(\Sigma^{\height(\qq)-\f(\qq)}D_{\widehat{R_\qq}},\Sigma^{\height(\pp)-\f(\pp)}D_{\widehat{R_\pp}})=\RHom_R(\Sigma^{\height(\qq)}D_{\widehat{R_\qq}},\Sigma^{\height(\pp)}D_{\widehat{R_\pp}})\] 
as a direct summand. Thus $H^1\RHom_R(C,C)\neq 0$, but this contradicts that $C\in \Perp{>0}C$. Therefore $\f=\f_\Phi $ is strictly increasing, that is, $\Phi$ is a slice sp-filtration.

Next, suppose that $C$ is cotilting. Then $\f$ is strictly increasing by the above argument.
If $\f$ is not a codimension function, there is a saturated chain $\pp\subsetneq \qq$ such that $\f(\qq)-\f(\pp)>1$. 
Putting $n:=\f(\qq)-\f(\pp)$, we have
\[Y:=\RHom_R(\Sigma^{\height(\qq)-\f(\qq)}D_{\widehat{R_\qq}},\Sigma^{\height(\pp)-\f(\pp)}D_{\widehat{R_\pp}})=\Sigma^{n-\height(\qq)+\height(\pp)}\RHom_R(D_{\widehat{R_\qq}},D_{\widehat{R_\pp}}).\]
Recall that  $H^1\RHom_R(\RGamma_{\pp}R_\pp,\RGamma_{\qq}R_\qq)\neq 0 $ by \cref{minimal-case}. Hence  \cref{vanishing}, applied to $i=1$, yields $H^{1-\height(\qq)+\height(\pp)}\RHom_R(D_{\widehat{R_\qq}},D_{\widehat{R_\pp}})\neq 0$.
Therefore
\[H^{-n+1}Y=H^{1-\height(\qq)+\height(\pp)}\RHom_R(D_{\widehat{R_\qq}},D_{\widehat{R_\pp}})\neq 0,\]
where $-n+1<0$. 
Then $C \notin C\Perp{<0}$ as $Y$ is a direct summand of $C$, but this contradicts that $C$ is cotilting. Thus  $\f=\f_{\Phi}$ is a codimension function, that is, $\Phi$ is a codimension filtration.
\end{proof}

\cref{minimal-case-c,End-cotilt} below are cotilting versions of \cref{minimal-case,End-tilt}, respectively.

\begin{theorem}\label{minimal-case-c}
Let $\pp\in \Spec R$ and $W\subseteq \min (V(\pp)\sm \{\pp\})$. Let $\varkappa$ be any cardinal. Then there is a natural isomorphism
\[\RHom_{R}(\prod_{\qq\in W} \Sigma^{\height(\qq)-\height(\pp)-1} D_{\widehat{R_\qq}}^{\varkappa},D_{\widehat{R_{\pp}}})\cong\Hom_{\D(R)}(\prod_{\qq\in W} \Sigma^{\height(\qq)-\height(\pp)-1} D_{\widehat{R_\qq}}^{\varkappa},D_{\widehat{R_{\pp}}}).\]
in $\D(R)$. Moreover, the right-hand side is a flat $R$-module. 
\end{theorem}

\begin{corollary}\label{End-cotilt}
Assume that $R$ admits a codimension function $\cd$.
For each $n\in \ZZ$, set $W_n:=\{\pp\in \Spec R \mid \cd(\pp)=n\}$, $C(n):=\prod_{\pp\in W_n}\Sigma^{\height(\pp)-n}D_{\widehat{R_\pp}}$, and $Y(n):=C(n)\oplus C(n+1)$.
Let $\varkappa$ be any cardinal. Then there is a natural isomorphism
\[\RHom_R(Y(n)^{\varkappa},Y(n)) \cong \Hom_{\D(R)}(Y(n)^{\varkappa}, Y(n))\]
in $\D(R)$. Moreover, the right-hand side is a flat $R$-module.
\end{corollary}

It will be seen in the proof of \cref{End-cotilt} that $\RHom_{R}(C(n)^\varkappa, C(n+1))= 0$, so $\End_{\D(R)}(Y(n))$ naturally becomes an upper triangular matrix ring. However the proof of \cref{minimal-case-c} will show that a more concrete  description of the right-hand side in \cref{minimal-case-c} is more complicated than that of \cref{minimal-case}. Hence, for the cotilting side, we do not give a result like \cref{End-tilt-2}. See also \cref{End-cotilt-ex} below.

\cref{End-cotilt} yields the next theorem. As already mentioned, there is no formal way to dualize the tilting condition to the cotilting condition, so we can not deduce this theorem directly from \cref{1-dim}.

\begin{theorem}\label{1-dim-cotilt}
	Assume that the Krull dimension of $R$ is at most one. Let $\Phi$ be a codimension filtration $\Phi$ of $\Spec R$. Then $C_\Phi$ is a cotilting object in $\D(R)$.
\end{theorem}

\begin{proof}
By \cref{slice-cosilt}, $C_\Phi$ is a cosilting object in $\D(R)$. To show that $C_\Phi$ is cotilting, we may assume that $\Phi$ is the height filtration; see \cref{non-connected}. Then this theorem follows from \cref{End-cotilt}.
\end{proof}

Let $R$ be as in \cref{1-dim-cotilt} and $\Phi$ the height filtration. 
Then $C_{\Phi}= \prod_{\pp\in \Spec R}D_{\widehat{R_\pp}}$, and in the derived category $\D(R)$, we may interpret $C_{\Phi}$ as a complex of injective $R$-modules concentrated in degrees $0$ and $1$; see \cref{bd-comp-inj}.
Hence, instead of using \cref{End-cotilt}, one can also complete the proof of \cref{1-dim-cotilt} by the following result due to Pavon and Vitória. For its proof, see also \cite[Remark 4.8]{PV21} and \cref{silt-complex-module-com} below.

\begin{theorem}[{\cite[Corollary 5.12]{PV21}}]\label{PV-theorem}
Let $R$ be a commutative noetherian ring and $C$ a 2-term cosilting complex over $R$. Then $C$ is a cotilting object in $\D(R)$.
\end{theorem}

Let us now prove \cref{minimal-case-c,End-cotilt}.
For this purpose, we remark that
\begin{equation}\label{cosupport-D}
\cosupp_R D_{\widehat{R_\pp}}=\{\pp\}
\end{equation}
for each $\pp \in \Spec R$.
Moreover, given a subset $W$ of $\Spec R$, $Y\in \cC^W$, and $\pp\in \Spec R$ with $W\cap U(\pp)=\emptyset$, we have 
\begin{equation}\label{cosupport-RHom}
\RHom_R(X, Y)=0
\end{equation}
for every complex $X$ of $R_\pp$-modules because $X\in \cL_{U(\pp)} =\Perp{0}\cC^{V}$ by \cref{bi-loc} and $\cC^{W}\subseteq \cC^{V}$, where $V:=U(\pp)^\cp$.
We will also need the fact that, for any $\pp\in \Spec R$, there is a quasi-isomorphism
\begin{equation}\label{min-case-isos-co-3}
\check{C}(\pp)\otimes_R D_{\widehat{R_\pp}} \to \Sigma^{-\height(\pp)} E_R(R/\pp)
\end{equation}
of complexes of $\widehat{R_\pp}$-modules; see  \cref{Cech-comp}, \cref{dc-loc-coh}, and \cref{base-change}.

\begin{proof}[Proof of \cref{minimal-case-c}] 
Let $X := \prod_{\qq\in W} \Sigma^{\height(\qq)-\height(\pp)-1}D_{\widehat{R_\qq}}^{\varkappa}$. By \cref{local-dual-2}, we have 
	\begin{equation}\label{min-case-iso-co}
		\RHom_{R}(X,D_{\widehat{R_{\pp}}}) \cong \RHom_{R_\pp}(X \LotimesR \Sigma^{\height (\pp)}\RGamma_\pp R_\pp, E_R(R/\pp)).
	\end{equation}
Note that $E_R(R/\pp)\cong E_{R_\pp}(\kappa(\pp))$ is an injective cogenerator in $\Mod R_\pp$ and $X \LotimesR \Sigma^{\height (\pp)}\RGamma_\pp R_\pp \cong \Sigma^{\height (\pp)}\RGamma_\pp X_\pp$ in $\D(R)$ by \cref{Cech-comp}.
	Hence it suffices to show that  $\Sigma^{\height (\pp)}\RGamma_\pp X_\pp$ is isomorphic to an injective $R$-module in $\D(R)$; see \cite[Theorem 3.2.16]{EJ11}.

	Consider the approximation triangle
	\begin{equation}\label{min-case-tri-co}
	\gamma_{U(\pp)^\cp}X \to X \to X_\pp\xrightarrow{+},
	\end{equation}
	where $X_\pp\cong \lambda^{U(\pp)}X$ by \cref{typic-(co)loc}.
	Since $W\subseteq U(\pp)^\cp$ by assumption, it follows from \cref{cosupport-D,cosupport-RHom} that $\RHom_R(R_\pp,X) = 0$. Thus application of $\RGamma_\pp\RHom_R(R_\pp,-)$ to \cref{min-case-tri-co} yields the natural isomorphism 
	\begin{align}\label{min-case-nat-co}
\RGamma_\pp X_\pp\isoto \Sigma\RGamma_\pp\RHom_R(R_\pp,\gamma_{U(\pp)^\cp}X).
	\end{align}
	Moreover, by \cref{closed-iso,GM-dual,ATJLL97}, tensor-hom adjunction, and \cite[Proposition 6.1]{BIK08}, there are natural isomorphisms
	\begin{align*}
\RGamma_\pp\RHom_R(R_\pp,\gamma_{U(\pp)^\cp}X)
&\cong
\RGamma_\pp\RHom_R(\RGamma_\pp R,\RHom_R(R_\pp,\gamma_{U(\pp)^\cp}X))\\
&\cong\RGamma_\pp \RHom_R(R_\pp, \RHom_R(\RGamma_\pp R,\gamma_{U(\pp)^\cp}X))\\
&\cong\RGamma_\pp \RHom_R(R_\pp, \RHom_R(\RGamma_\pp R,\gamma_{V(\pp)\cap U(\pp)^\cp}X))\\
&\cong\RGamma_\pp \RHom_R(\RGamma_\pp R,\RHom_R(R_\pp,\gamma_{V(\pp)\cap U(\pp)^\cp}X))\\
&\cong\RGamma_\pp \RHom_R(R_\pp,\gamma_{V(\pp)\cap U(\pp)^\cp}X).
	\end{align*}
Thus we can rewrite \cref{min-case-nat-co} to
\begin{align}\label{min-case-nat-co-2}
\RGamma_\pp X_\pp\isoto \Sigma \RGamma_\pp\RHom_R(R_\pp,\gamma_{V(\pp)\cap U(\pp)^\cp}X).
	\end{align}
Since $X\in \cC^{\overline{W}^{\mathrm{g}}}$ and $V(\pp)\cap U(\pp)^\cp \cap \overline{W}^{\mathrm{g}}=(V(\pp)\sm\{\pp\}) \cap \overline{W}^{\mathrm{g}}=W$, we have natural isomorphisms
\begin{align*}
	\gamma_{V(\pp)\cap U(\pp)^\cp}X \cong \gamma_{V(\pp)\cap U(\pp)^\cp}\gamma_{\overline{W}^{\mathrm{g}}}X \cong \gamma_{V(\pp)\cap U(\pp)^\cp \cap \overline{W}^{\mathrm{g}}}X \cong \bigoplus_{\qq \in W}\RGamma_\qq\RHom_R(R_\qq,X)
	\end{align*}
	in $\D(R)$, where the second and the third isomorphisms follow from  \cite[Proposition 3.21]{NY18b} and \cite[Theorem 3.12]{NY18a}, respectively.
Moreover, we see from \cref{cosupport-D,cosupport-RHom} that 
 \[\RGamma_{\qq}\RHom_R(R_\qq,X) \cong \RGamma_{\qq}\RHom_R(R_\qq,\Sigma^{\height(\qq)-\height(\pp)-1}D_{\widehat{R_\qq}}^{\varkappa})\cong \Sigma^{\height(\qq)-\height(\pp)-1}\RGamma_{\qq}D_{\widehat{R_\qq}}^{\varkappa}.\]
Thus we can rewrite \cref{min-case-nat-co-2} to
	\begin{align}\label{min-case-isos-co-2}
\RGamma_\pp X_\pp \isoto  \Sigma^{\height(\qq)-\height(\pp)}\RGamma_\pp\RHom_R(R_\pp,\bigoplus_{\qq \in W}\RGamma_\qq D_{\widehat{R_\qq}}^{\varkappa}).
	\end{align}
	
	Now, using \cref{Cech-comp,iso-kappa-product,min-case-isos-co-3}, we compute
	\begin{align*}
	\RGamma_\qq D_{\widehat{R_\qq}}^{\varkappa} 
	&\cong \check{C}(\qq) \otimes_R D_{\widehat{R_\qq}}^{\varkappa}\\
	&\cong \check{C}(\qq) \otimes_R (D_{\widehat{R_\qq}} \otimes_{\widehat{R_\qq}} \widehat{R_\qq}^{\varkappa})\\
	 &\cong (\check{C}(\qq) \otimes_R D_{\widehat{R_\qq}}) \otimes_{\widehat{R_\qq}} \widehat{R_\qq}^{\varkappa}\\
	 &\cong \Sigma^{-\height(\qq)} E_R(R/\qq) \otimes_{\widehat{R_\qq}} \widehat{R_\qq}^{\varkappa}
\end{align*}
in $\D(R)$. 
Combining this with \cref{min-case-isos-co-2}, we have
\begin{align*}
\Sigma^{\height (\pp)}\RGamma_\pp X_\pp
&\cong 
\Sigma^{\height(\qq)}\RGamma_\pp\RHom_R(R_\pp,\bigoplus_{\qq \in W}\RGamma_\qq D_{\widehat{R_\qq}}^{\varkappa})\\
&\cong  \Gamma_\pp\Hom_R(R_\pp,\bigoplus_{\qq \in W} E_{R}(R/\qq) \otimes_{\widehat{R_\qq}} \widehat{R_\qq}^{\varkappa})
\end{align*}
in $\D(R)$, where the second isomorphism holds because the $R$-modules $E_{R}(R/\qq) \otimes_{\widehat{R_\qq}} \widehat{R_\qq}^{\varkappa}
$ and $\Hom_R(R_\pp,\bigoplus_{\qq \in W} E_{R}(R/\qq) \otimes_{\widehat{R_\qq}} \widehat{R_\qq}^{\varkappa})$ are both injective; see \cite[Theorems 3.2.9 and 3.2.16]{EJ11}.
Therefore it follows from \cref{Gamma-Inj} that $\Sigma^{\height (\pp)}\RGamma_\pp X_\pp$ is isomorphic to an injective $R$-module, as desired.
\end{proof}

\begin{proof}[Proof of \cref{End-cotilt}]
Let $i$ and $j$ be integers.
 There are natural isomorphisms 
\begin{align*}
\RHom_{R}(C(i)^\varkappa,C(j))
&\cong \prod_{\pp\in W_j}\RHom_{R}(\prod_{\qq\in W_i}\Sigma^{\height(\qq)-i}D^\varkappa_{\widehat{R_\qq}},\Sigma^{\height(\pp)-j}D_{\widehat{R_\pp}})\\
&\cong \prod_{\pp\in W_j}\RHom_{R}(\prod_{\qq\in V(\pp)\cap W_i}\Sigma^{\height(\qq)-i}D^\varkappa_{\widehat{R_\qq}},\Sigma^{\height(\pp)-j}D_{\widehat{R_\pp}}),
\end{align*}
where the second isomorphism follows from \cref{bi-loc} and \cref{cosupport-D}.

If $i=j$ and $\pp\in W_i$, then $V(\pp)\cap W_i=\{\pp\}$, so we have
\begin{align}\label{End-cotilt-iso}
\RHom_{R}(C(i)^\varkappa,C(i))\cong \prod_{\pp\in W_i}\RHom_{R}(D^\varkappa_{\widehat{R_\pp}},D_{\widehat{R_\pp}})\cong \prod_{\pp\in W_i}\Hom_{\D(R)}(D^\varkappa_{\widehat{R_\pp}},D_{\widehat{R_\pp}}),
\end{align}
where the second follows from \cref{Cech-comp}, \cref{local-dual-4}, \cref{iso-kappa-product}, and \cref{min-case-isos-co-3} because 
\begin{align*}
\RHom_R(D_{\widehat{R_{\pp}}}^\varkappa,D_{\widehat{R_{\pp}}})
&\cong 
\RHom_R(\RGamma_{\pp}D^\varkappa_{\widehat{R_{\pp}}}, \Sigma^{-\height(\pp)}E_{R}(R/\pp))\\
&\cong \RHom_R(E_{R}(R/\pp) \otimes_{\widehat{R_\pp}} \widehat{R_\pp}^{\varkappa}, E_{R}(R/\pp))\\
&\cong \Hom_R(E_{R}(R/\pp) \otimes_{\widehat{R_\pp}} \widehat{R_\pp}^{\varkappa}, E_{R}(R/\pp)).
\end{align*}
The last $R$-module is flat since $E_{R}(R/\pp) \otimes_{\widehat{R_\pp}} \widehat{R_\pp}^{\varkappa}$ is injective over $R$; see \cite[Theorems 3.2.9 and 3.2.16]{EJ11}.
 
 If $(i,j)=(n+1,n)$ for some integer $n$, then we have 
 \begin{align*}
 \RHom_{R}(C(n+1)^\varkappa,C(n))
&\cong \prod_{\pp\in W_{n}}\RHom_{R}(\prod_{\qq\in V(\pp)\cap W_{n+1}}\Sigma^{\height(\qq)-n-1}D_{\widehat{R_\qq}}^\varkappa,\Sigma^{\height(\pp)-n}D_{\widehat{R_\pp}})\\
&=  \prod_{\pp\in W_{n}}\RHom_{R}(\prod_{\qq\in V(\pp)\cap W_{n+1}}\Sigma^{\height(\qq)-\height(\pp)-1}D_{\widehat{R_\qq}}^\varkappa,D_{\widehat{R_\pp}})\\
&\cong  \prod_{\pp\in W_{n}}\Hom_{\D(R)}(\prod_{\qq\in V(\pp)\cap W_{n+1}}\Sigma^{\height(\qq)-\height(\pp)-1}D_{\widehat{R_\qq}}^\varkappa,D_{\widehat{R_\pp}})
\end{align*}
in $\D(R)$, where the third isomorphism holds by \cref{minimal-case-c}, which also shows that the last $R$-module is flat.

If $i<j$ and $\pp\in W_j$, then $V(\pp)\cap W_i=\emptyset$, so 
\[\RHom_{R}(C(i)^\varkappa, C(j))= 0\]
in $\D(R)$.
We have completed the proof.
\end{proof}

\begin{example}\label{End-cotilt-ex}
Let $R$ be a 1-dimensional Gorenstein ring.
The cotilting object $C_{\Phi_{\height}}$ for the height filtration ${\Phi_{\height}}$ is of the form $(\prod_{\height(\pp)=0} \widehat{R_\pp})\oplus (\prod_{\height(\mm)=1} \widehat{R_\mm})$ (see \cref{ex-(co)tilt}), where $\prod_{\height(\pp)=0} \widehat{R_\pp}=\prod_{\height(\pp)=0} R_\pp$ can be identified with the total quotient ring $Q(R)$ of $R$. 
By \cite[Lemma 4.1.8]{Xu96} or by the proof of \cref{End-cotilt} applied to $\varkappa=1$ and Matlis duality, we can verify that there is a natural ring isomorphism
\begin{equation*}
\End_{\D(R)}(C_{\Phi_{\height}})\cong
\begin{pmatrix}
	Q(R) &  \displaystyle{\Hom_R(\prod_{\height(\mm)=1} \widehat{R_{\mm}},Q(R))}\vspace{2mm}\\
	0 & \displaystyle{\prod_{\height(\mm)=1} \widehat{R_\mm}}
	\end{pmatrix}.
\end{equation*}
\end{example}

\begin{remark}\label{1-dim-no-dc}
By \cite[Proposition 3.1]{FR70}, there exists a 1-dimensional local domain that does not admit a dualizing complex; see also \cite[Theorem 3.8]{Sha77} and \cite[Proposition 3.7]{Olb12}. Thus \cref{1-dim,1-dim-cotilt} are not covered by \cref{tilt-dc,dc-cotilt}. 
A principle hidden here would be that $R$ is a homomorphic image of a Cohen--Macaulay ring. We will discuss this theme in the next section. 
\end{remark}

\begin{remark}
The existence of a codimension function is not essential for \cref{X(n)-X(n+1),End-tilt} as explained in the following:
Let $R$ be a commutative noetherian ring and fix an integer $n$. Let $W_n$ and $W_{n+1}$ be subsets of $\Spec R$ with $\dim W_n\leq 0$ and $\dim W_{n+1}\leq 0$.
Assume that  $W_n\cap V(\qq)=\emptyset$ for any $\qq\in W_{n+1}$ and that any chain $\pp\subsetneq \qq$ in $W_n\cup W_{n+1}$ is saturated. 
Set $T(i):=\bigoplus_{\pp\in W_i}\Sigma^i\RGamma_{\pp}R_\pp$ for $i=n, n+1$ and $Y(n):=T(n)\oplus T(n+1)$. 
Then we have the isomorphisms in \cref{X(n)-X(n+1),End-tilt}, and the second claim of \cref{End-tilt} also holds.

Similarly, letting $W_n$ and $W_n+1$ be as above, and setting $C(i):=\prod_{\pp\in W_i}\Sigma^{\height(\pp)-i}D_{\widehat{R_\pp}}$ for $i=n, n+1$ and $Y(n):=C(n)\oplus C(n+1)$, we have the isomorphism in \cref{End-cotilt}, and the second claim of the corollary also holds. So the existence of a codimension function is not essential for the corollary. 
\end{remark}

The next result is the silting counterpart of \cref{PV-theorem}.
 
\begin{theorem}\label{2-term}
Let $R$ be a commutative noetherian ring and $T$ a 2-term silting complex over $R$. Then $T$ is a tilting object in $\D(R)$.
\end{theorem}

It is possible to deduce \cref{2-term} from \cref{PV-theorem} of Pavon and Vitória by some duality argument, but we here prove \cref{2-term} in a more direct way. There does not seem to be any duality argument which recovers the cosilting result \cref{PV-theorem} from \cref{2-term}.

\begin{remark}
Assume $d:=\dim R<\infty$. 
If $\pp$ is minimal, then $\RGamma_\pp R_\pp\cong R_\pp$ in $\D(R)$, so $\RGamma_\pp R_\pp$ is isomorphic to a bounded complex of projective $R$-modules concentrated in degrees from $-d$ to $0$; see the proof of \cref{fin-proj-dim}.

If $\mm$ is maximal, then $\RGamma_\mm R_\mm\cong \RGamma_\mm R$ in $\D(R)$ (see \cref{RHom-TPhi}).
Let $\boldsymbol{x}=x_1,\ldots,x_{d_\mm}$ be a sequence of elements in $R$ such that their images in $R_\mm$ form a system of parameters of $R_\mm$, where $d_\mm=\dim R_\mm$ (see \cref{radical-torsion-sop}). Let $I$ be the ideal generated by $\boldsymbol{x}$.
Then $R_\mm/IR_\mm$ is an artinian local ring, so we see that the (Zariski) closed subset $V(I)$ of $\Spec R$ decomposes into the disjoint union of $V(\mm)$ and another closed set $V(J)$, where $V(J)$ may be empty (i.e., $J$ may be $R$). Hence $\Gamma_I\cong \Gamma_{V(I)}\cong \Gamma_{V(\mm)}\oplus \Gamma_{V(J)}\cong \Gamma_\mm\oplus \Gamma_J$ as functors $\Mod R\to \Mod R$; the second isomorphism easily follows from \cref{Gamma-Inj}. Thus $\RGamma_\mm R$ is a direct summand of $\RGamma_{I} R\cong \check{C}(\boldsymbol{x})$ in $\D(R)$. Since $\check{C}(\boldsymbol{x})$ is isomorphic in $\D(R)$ to a bounded complex of projective $R$-modules concentrated in degrees from $0$ to $d_\mm=\height(\mm)$ (see, e.g., \cite[Lemma 6.9]{DG02}), so is $\RGamma_\mm R\cong \RGamma_\mm R_\mm$.

When $d=1$, the above observation implies that, in $\D(R)$,  $T_{\Phi_{\height}}=\bigoplus_{\pp\in \Spec R} \Sigma^{\height(\pp)}\RGamma_\pp R_\pp$ is isomorphic to a complex $P$ of projective $R$-modules concentrated in degrees from $-1$ to $0$, and $P$ is a 2-term silting complex by \cref{slice-silting}.
Thus one may also use \cref{2-term} to complete the proof of \cref{1-dim} instead of \cref{End-tilt}.
\end{remark}
\begin{remark}\label{silt-complex-module-com}
Silting and cosilting modules were introduced in \cite{AHMV16a} and \cite{BP17} as module-theoretic alternatives to 2-term silting and cosilting complexes. The assignment $X \mapsto H^0X$ induces a bijection between equivalence classes of silting (resp. cosilting) complexes concentrated in degrees $-1$ and $0$ (resp. degrees $0$ and $1$) and equivalence classes of silting (resp. cosilting) modules. 

Over a commutative noetherian ring $R$, any cosilting module $M$ in $\Mod R$ cogenerates a hereditary torsion pair with torsion class $ \Mod R\cap \Perp{0}M$, and in fact, all hereditary torsion pairs in $\Mod R$ are of this form. Then cosilting (resp. silting) modules in $\Mod R$ are, up to equivalence, in bijection with specialization closed subsets of $\Spec R$. Under this bijection, a cosilting (resp. silting) module $M$ corresponds to a specialization closed subset $W$ such that for any ideal $I$ we have $\Hom_R(R/I,M) = 0$ (resp. $(R/I) \otimes_R M = 0$) if and only if $V(I) \subseteq W$. 
Moreover,  if $T$ is a silting module corresponding to a specialization closed subset $W$, its dual $T^+=\Hom_R(T,E)$ by any injective cogenerator $E$ is, up to equivalence, a cosilting module corresponding to $W$.
See \cite[Corollaries 3.7 and 4.1, Lemma 4.2, and Theorem 5.1]{AHH17} for details.

This classification of (co)silting modules is compatible with \cref{cosilt-cof-sp,silt-fin-sp} in the following sense:
If $C$ is a 2-term cosilting complex concentrated in degrees 0 and 1, then the sp-filtration $\Phi$ associated with $C$ by \cref{cosilt-cof-sp} satisfies $\Phi(-1) = \Spec R$ and $\Phi(1) = \emptyset$ (see the references in the last paragraph of \cref{silt-equiv}), and $\Phi(0)$ corresponds, up to equivalence, to the cosilting module $H^0C$ via \cite[Theorem 5.1]{AHH17}; see \cite[Proposition 2.8]{HB21} (cf. \cite[\S 5.2]{HHZ21}). 
By duality \cref{(co)silt-dual}, we have a 2-term silting complex $T$ such that the cosilting complex $T^+$ is equivalent to $C$, and $\Phi$ is the sp-filtration associated with $T$ by \cref{silt-fin-sp}. 
Since the cosilting module $(H^0T)^+\cong H^0(T^+)$ is equivalent to $H^0C$, the silting module $H^0T$ corresponds to $\Phi(0)$ via \cite[Theorem 5.1]{AHH17}.
\end{remark}
\begin{proof}[Proof of \cref{2-term}]
	We need to show that $T^{(\varkappa)} \subseteq T\Perp{<0}$ for any cardinal $\varkappa$. By shifting $T$ appropriately, we can assume that $T$ is a complex of projective $R$-modules concentrated in degrees $-1$ and $0$. 
	Clearly, $\Hom_{\D(R)}(T,\Sigma^i T^{(\varkappa)}) = 0$ for any $i<-1$, so we only need to show $\Hom_{\D(R)}(T,\Sigma^{-1}T^{(\varkappa)}) = 0$. 
Note that
\[\Hom_{\D(R)}(T,\Sigma^{-1}T^{(\varkappa)}) \cong \Hom_{R}(H^0T,(H^{-1}T)^{(\varkappa)}).\] Therefore, it is enough to show that there is no nonzero $R$-homomorphism from $H^0T$ to $H^{-1}T$.

Denote by $\Gen(H^0 T)$ the class of all epimorphic images of all coproducts of copies of the module $H^0 T$ in $\Mod R$.
The specialization closed subset $W$ corresponding to $H^0T$ via \cite[Theorem 5.1]{AHH17} (see \cref{silt-complex-module-com}) is determined by the property that $\Gen(H^0T)$ equals the class of all $R$-modules $M$ with $M=IM$ for any  ideal $I$ satisfying $V(I)\subseteq W$. 
If $W = \emptyset$, the classification implies that the silting complex $T$ is equivalent to the tilting complex $R$. Hence we may assume $W \neq \emptyset$.
Take an arbitrary $f \in \Hom_{R}(H^0T,H^{-1}T)$.

Suppose first that $R$ is local and let $N$ denote the image of $f$ in $H^{-1}T$. By definition, $N$ belongs to $\Gen (H^0T)$, so $N=\pp N$ for any $\pp \in W$. Since $N$  is contained in the projective $R$-module $T^{-1}$, we can find a nonzero $R$-homomorphism $g: N \to R$ whenever $N$ is nonzero, but then Nakayama Lemma implies that the image of $g$ is zero; consequently $N$ must be zero. Hence $f$ is zero as well. 

If $R$ is non-local, $T_\mm$ is a 2-term silting complex over $R_\mm$ for each maximal ideal $\mm$ (\cite[Lemma 6.3]{HHZ21}). Then the induced map
\[f \otimes_R R_\mm:(H^0T)\otimes_R R_\mm\cong H^0(T_\mm) \to H^{-1}(T_\mm)\cong (H^{-1}T)\otimes_R R_\mm\]
is zero for each maximal ideal $\mm$ by the previous paragraph. This implies that the image of $f$ is zero, so that $f$ is zero, as desired. 
\end{proof}

We close this section by showing the following result, which will be used in the next section.

\begin{proposition}\label{flat-End}
Let $R$ be a Cohen--Macaulay ring of finite Krull dimension and $\Phi$ a codimension filtration of $\Spec R$.
Let $\varkappa$ be any cardinal. Then the following hold:

\begin{enumerate}[label=(\arabic*), font=\normalfont]
\item \label{flat-End-tilt} $\Hom_{\D(R)}(T_{\Phi}, T_{\Phi}^{(\varkappa)})$ is a flat $R$-module.
In particular, the canonical ring homomorphism $R\to \End_{\D(R)}(T_{\Phi})$ is faithfully flat. 
\item \label{flat-End-cotilt} $\Hom_{\D(R)}(C_{\Phi}^{\varkappa}, C_{\Phi})$ is a flat $R$-module.
In particular, the canonical ring homomorphism $R\to \End_{\D(R)}(C_{\Phi})$ is faithfully flat. 
\end{enumerate}
\end{proposition}

\begin{proof}
In view of \cref{non-connected}, we may assume that $\Spec R$ is connected. Then the codimension function $\f_\Phi:\Spec R\to \ZZ$ coincides with the height function $\height:\Spec R\to \ZZ$ up to shift; see \cref{cd-func}. So we may assume that $\Phi$ is the height filtration.

\cref{flat-End-tilt}:
By \cref{tilt-CM}, we have
\[T_\Phi=\bigoplus_{\pp\in \Spec R} \Sigma^{\height(\pp)}\RGamma_{\pp}R_\pp\cong  \bigoplus_{\pp \in \Spec R}H_\pp^{\height(\pp)} R_\pp\]
in $\D(R)$ and the corollary also shows that there is a natural isomorphism
\begin{align}\label{tilt-CM-Hom}
\RHom_R(T_\Phi, T_\Phi^{(\varkappa)})\cong \Hom_{\D(R)}(T_\Phi, T_\Phi^{(\varkappa)})
\end{align}
in $\D(R)$.
Moreover
\begin{align*}
\RHom_R(T_\Phi, T_\Phi^{(\varkappa)})
&\cong \prod_{\pp\in \Spec R}\RHom_R(\RGamma_{\pp}R_\pp,\Sigma^{-\height(\pp)} T_\Phi^{(\varkappa)})\\
&\cong\prod_{\pp\in \Spec R}\LLambda^{\pp}\RHom_R(R_\pp, \Sigma^{-\height(\pp)}T_\Phi^{(\varkappa)})
\end{align*}
by \cref{RHom-TPhi}. Each component $\LLambda^{\pp}\RHom_R(R_\pp, \Sigma^{-\height(\pp)}T_\Phi^{(\varkappa)})$ is isomorphic in $\D(R)$ to an $R_\pp$-module $M(\pp)$ by \cref{tilt-CM-Hom}, and any product of flat modules is flat since $R$ is noetherian (\cite[Theorem 3.2.24]{EJ11}). So once we show that $M(\pp)$ is a flat $R$-module for each  $\pp\in \Spec R$, $\Hom_{\D(R)}(T_{\Phi}, T_{\Phi}^{(\varkappa)})$ will be a flat $R$-module. 

Recall that $R_\pp$ and $T_\Phi^{(\varkappa)}$ are isomorphic to  bounded complexes of flat $R$-modules in $\D(R)$; see \cref{fin-proj-dim} and its proof.
Hence $\RHom_R(R_\pp, \Sigma^{-\height(\pp)}T_\Phi^{(\varkappa)})$ is isomorphic to a bounded complex of flat $R$-modules, but it can be further replaced by a bounded complex $F$ of flat $R_\pp$-modules because 
$\RHom_R(R_\pp, \Sigma^{-\height(\pp)}T_\Phi^{(\varkappa)})\cong \RHom_R(R_\pp, \Sigma^{-\height(\pp)}T_\Phi^{(\varkappa)})\otimes_R R_\pp$.
Then we have
\[\LLambda^{\pp}\RHom_R(R_\pp, \Sigma^{-\height(\pp)}T_\Phi^{(\varkappa)})\cong \LLambda^{\pp}F\cong \Lambda^{\pp}F\]
in $\D(R)$, where  $\Lambda^{\pp}F$ is a bounded complex of $\pp$-adic completions of flat $R_\pp$-modules. Note that the modules constituting $\Lambda^{\pp}F$ can be identified with $\pp$-adic completions of free $R_\pp$-modules; see \cite[Lemma 6.7.4]{EJ11}. 
Thus, according to \cite[Lemma 1.5]{NT20}, the complex $\Lambda^{\pp}F$ decomposes as a coproduct $F'\oplus F''$ in $\C(R)$ such that $\kappa(\pp)\otimes_{R_\pp} F'$ has zero differential and $F''$ is contractible (i.e., isomorphic to the zero complex in $\K(R)$). As a consequence, $M(\pp)$ is isomorphic to $F'$ in $\D(R)$.
By a version of the Auslander--Buchsbaum formula (see, e.g., \cite[Theorem 2.4]{FI03}), we have 
\begin{align*}
\depth_{R_\pp} M(\pp)=\depth_{R_\pp} R_\pp+\inf \kappa(\pp)\otimes_{R_\pp} F',
\end{align*}
where $\depth_{R_\pp} R_\pp=\dim R_\pp$ since $R_\pp$ is Cohen--Macaulay. Hence
\begin{align*}
\depth_{R_\pp} M(\pp)-\dim R_\pp =\inf \kappa(\pp)\otimes_{R_\pp} F'.
\end{align*}

If we show $\depth_{R_\pp} M(\pp)\geq \dim R_\pp$, then $\inf \kappa(\pp)\otimes_{R_\pp} F'\geq 0$, and this means that $F'$ is concentrated in non-negative degrees because $\kappa(\pp)\otimes_{R_\pp} F'$ has zero differential, $F'$ consists of $\pp$-adic completions of free $R_\pp$-modules, and $\kappa(\pp)\otimes_{R_\pp}\Lambda^\pp (R_\pp^{(\mu)})\cong \kappa(\pp)^{(\mu)}$ for any cardinal $\mu$. Thus, noting that $\kappa(\pp)\LotimesRp M(\pp)\cong \kappa(\pp)\otimes_{R_\pp} F'$ in $\D(R_\pp)$ by construction, we can conclude that the complex $F'$ is concentrated in degree zero. In particular, $F'$ is a flat $R$-module and $M(\pp)\cong F'$.

Hence it remains to show the inequality $\depth_{R_\pp} M(\pp)\geq \dim R_\pp$.
This follows from the following equalities
\begin{align*}
\depth_{R_\pp} M(\pp)
&=\inf \RHom_R(\kappa(\pp),M(\pp))\\
&=\inf \RHom_R(\kappa(\pp),\LLambda^{\pp}\RHom_R(R_\pp, \Sigma^{-\height(\pp)}T_\Phi^{(\varkappa)}))\\
&=\inf \RHom_R(\RGamma_{\pp}\kappa(\pp),\RHom_R(R_\pp, \Sigma^{-\height(\pp)}T_\Phi^{(\varkappa)}))\\
&=\inf \RHom_R(\kappa(\pp),\RHom_R(R_\pp, \Sigma^{-\height(\pp)}T_\Phi^{(\varkappa)}))\\
&=\inf \RHom_{R}(\kappa(\pp)\LotimesR R_\pp, \Sigma^{-\height(\pp)}T_\Phi^{(\varkappa)}))\\
&=\inf \RHom_{R}(\kappa(\pp), \Sigma^{-\height(\pp)}T_\Phi^{(\varkappa)}))\\
&=\inf \RHom_{R}(\kappa(\pp), \Sigma^{-\height(\pp)}\textstyle{\bigoplus_{\qq \in \Spec R}H_\qq^{\height(\qq)} R_\qq^{(\varkappa)})}\\
&\geq \height(\pp)=\dim R_\pp.
\end{align*}
The second statement of \cref{flat-End-tilt} follows from \cref{ring-hom} because $\End_R(T_\Phi)$ contains $\prod_{\pp\in \Spec R}\widehat{R_\pp}$ as a direct summand and $\prod_{\pp\in \Spec R}\widehat{R_\pp}$ is a faithfully flat $R$-module (see \cite[Theorems 4.6 and 8.14]{Mat89}).

\cref{flat-End-cotilt}: As assumed, $\Phi$ is the height filtration, so that 
\[C_\Phi = \prod_{\pp \in \Spec R}D_{\widehat{R_\pp}}\cong\prod_{\pp \in \Spec R}\omega_{\widehat{R_\pp}}\]
by \cref{cotilt-CM}. The corollary also shows that
\begin{align}\label{cotilt-CM-Hom}
\RHom_R(C_\Phi^\varkappa,C_\Phi)\cong \Hom_{\D(R)}(C_\Phi^\varkappa,C_\Phi)
\end{align}
in $\D(R)$. Moreover, using \cref{local-dual-2}, we have 
\[\RHom_R(C_\Phi^\varkappa,C_\Phi) \cong \prod_{\pp \in \Spec R}\RHom_{R}(C_\Phi^\varkappa \LotimesR \Sigma^{\height(\pp)}\RGamma_\pp R_\pp,E_{R}(R/\pp)),\]
where $\RHom_R$ in the right-hand side can be replaced by $\RHom_{R_\pp}$, and $E_R(R/\pp)\cong E_{R_\pp}(\kappa(\pp))$ is an injective cogenerator in $\Mod R_\pp$.
For each $\pp \in \Spec R$, $C_\Phi^\varkappa \LotimesR \Sigma^{\height(\pp)}\RGamma_\pp R_\pp$ is isomorphic in $\D(R)$ to an $R_\pp$-module $N(\pp)$ by \cref{cotilt-CM-Hom}.
Thus $\Hom_{\D(R)}(C_\Phi^\varkappa,C_\Phi)$ is the product of $\Hom_{R}(N(\pp),E_{R}(R/\pp))$ for all $\pp\in \Spec R$. So once we show that $N(\pp)$ is an injective $R$-module for each $\pp \in \Spec R$, then $\Hom_{\D(R)}(C_\Phi^\varkappa,C_\Phi)$ will be a flat $R$-module.

Since $\dim R < \infty$ and $\Phi$ is the height filtration, $C_\Phi\cong \prod_{\pp \in \Spec R}\omega_{\widehat{R_\pp}}$ is an $R$-module of finite injective dimension (\cref{bd-comp-inj}). It follows that $C_\Phi^\varkappa \otimes_R R_{\pp}$ is an $R_\pp$-module of injective dimension bounded by $\dim R_\pp = \height (\pp)$ (\cite[Corollary 5.5]{Bas62}). Thus $C_\Phi^\varkappa \LotimesR \RGamma_\pp R_\pp\cong \RGamma_\pp (C_\Phi^\varkappa \otimes_R R_{\pp})$ is isomorphic in $\D(R)$ to a bounded complex of injective $R$-modules concentrated in degrees from $0$ to $\height(\pp)$ (\cref{Gamma-Inj}). 
On the other hand, 
we know that $C_\Phi^\varkappa \LotimesR \Sigma^{\height(\pp)}\RGamma_\pp R_\pp$ is isomorphic in $\D(R)$ to the module $N(\pp)$, so
 \[\inf C_\Phi^\varkappa \LotimesR \RGamma_\pp R_\pp=\inf \RGamma_\pp (C_\Phi^\varkappa \otimes_R R_\pp) \geq \height(\pp).\]  
Consequently, there is an injective $R$-module $E$ such that $\RGamma_\pp (C_\Phi^\varkappa \otimes_R R_\pp)$ is isomorphic in $\D(R)$ to $\Sigma^{-\height(\pp)}E$. This fact in turn shows that $N(\pp) \cong \Sigma^{\height (\pp)}\RGamma_\pp (C_\Phi^\varkappa \otimes_R R_\pp) \cong E$, so $N(\pp)$ is an injective $R$-module.

The second statement of \cref{flat-End-cotilt} follows similarly as in \cref{flat-End-tilt}; see the second paragraph of the proof of \cref{End-cotilt} and use Matlis duality.
\end{proof}

\begin{remark}\label{flatcotorsion}
An $R$-module $M$ is called \emph{cotorsion} if $\Ext^1_R(F,M)=0$ for any flat $R$-module.
A \emph{flat cotorsion} $R$-module means an $R$-module which is flat and cotorsion.
It is known that an $R$-module $M$ is flat cotorsion if and only if $M$ is isomorphic to a product $\prod_{\pp\in \Spec R} T_\pp$, where $T_\pp$ is the $\pp$-adic completion of a free $R_\pp$-module (\cite[Theorem 5.3.28]{EJ11}).
We see from \cref{X(n)-X(n+1)} that the right-hand side of the isomorphism in \cref{End-tilt} is in fact a flat cotorsion $R$-module. Similarly, the right-hand side of the isomorphism in \cref{minimal-case-c} is a flat cotorsion $R$-module.

More generally, $\Hom_{\D(R)}(T_\Phi,T_\Phi^{(\varkappa)})$ and $\Hom_{\D(R)}(C_\Phi^\varkappa,C_\Phi)$ in \cref{flat-End} are not only flat but cotorsion as $R$-modules.
Indeed, the proof of \cref{flat-End}\cref{flat-End-tilt} shows that $\Hom_{\D(R)}(T_\Phi,T_\Phi^{(\varkappa)})$ is isomorphic to the direct product $\prod_{\pp\in \Spec R} M(\pp)$. By \cref{RHom-TPhi}, each $M(\pp)$ belongs to $\cC^{\{\pp\}}$, and then \cite[Lemma 4.2]{NY18b} implies that $M(\pp)$ is cotorsion. So the product $\prod_{\pp\in \Spec R} M(\pp)$ is cotorsion.
To see that $\Hom_{\D(R)}(C_\Phi^\varkappa,C_\Phi)$ is cotorsion, recall from the proof of \cref{flat-End}\cref{flat-End-cotilt} that $\Hom_{\D(R)}(C_\Phi^\varkappa,C_\Phi)$ is isomorphic to the product $\prod_{\pp\in \Spec R} \Hom_R(N(\pp), E_R(R/\pp))$. 
Then each $\Hom_R(N(\pp), E_R(R/\pp))$ is pure-injective (see the proof of \cite[Proposition 5.3.7]{EJ11} or \cite[Proposition 4.3.29]{Pre09}), so it is cotorsion (\cite[Lemma 5.3.23]{EJ11}). Thus $\prod_{\pp\in \Spec R} \Hom_R(N(\pp), E_R(R/\pp))$ is cotorsion.
\end{remark}

\section{Homomorphic images of Cohen--Macaulay rings}\label{Hom-Im-CM}

We proved in \cref{tilt-CM,cotilt-CM} that if $R$ is a Cohen--Macaulay ring and $\Phi$ is a codimension filtration of $\Spec R$, then $T_\Phi$ is tilting and $C_\Phi$ is cotilting.
In this section, we will show that every homomorphic image of $R$ inherits this property whenever $R$ is a Cohen--Macaulay ring of finite Krull dimension (\cref{finite-theorem}). This fact will lead us to the deep theme on characterizing homomorphic images of Cohen--Macaulay rings. Indeed, it will turn out that this theme is closely related to the tilting and the cotilting conditions on $T_\Phi$ and $C_\Phi$, respectively (\cref{flat-End-thm,2-dim-tilt}).

We begin with an elementary observation on sp-filtrations via change of rings.
Let $\varphi:R\to A$ be a homomorphism of commutative noetherian rings and let $f:\Spec A\to \Spec R$ be the canonical map induced by $\varphi$.
Let $\Phi$ be an sp-filtration of $\Spec R$, and define $f^{-1}\Phi: \ZZ\to 2^{\Spec A}$ as the map given by $n\mapsto f^{-1}(\Phi(n))$. Clearly $f^{-1}\Phi$ is an sp-filtration of $\Spec A$. Moreover, if $\Phi$ is non-degenerate, then so is $f^{-1}\Phi$. Indeed, we have  $\bigcap_{n\in \ZZ}f^{-1}(\Phi(n))=f^{-1}(\bigcap_{n\in \ZZ}\Phi(n))=f^{-1}(\emptyset)=\emptyset$ and $\bigcup_{n\in \ZZ}f^{-1}(\Phi(n))=f^{-1}(\bigcup_{n\in \ZZ}\Phi(n))=f^{-1}(\Spec R)=\Spec A$.

Recall that there is a canonical homeomorphism $f^{-1}\{\pp\}\isoto \Spec A\otimes_R\kappa(\pp)$ for each $\pp\in \Spec R$ (see \cite[p.~47]{Mat89}). 
Suppose $\varphi$ is \emph{finite}, that is, $A$ is finitely generated as an $R$-module. Then the fiber $f^{-1}\{\pp\}$ over $\pp$ is a (possibly empty) finite set and $\dim f^{-1}\{\pp\}\leq 0$ since the ring $A\otimes_R\kappa(\pp)$ is artinian. 

\begin{lemma}\label{sp-filt-finite}
Let $\varphi:R\to A$ be a finite homomorphism of commutative noetherian rings and let $f:\Spec A\to \Spec R$ be the canonical map induced by $\varphi$.
If $\Phi$ is a slice sp-filtration of $\Spec R$, then $f^{-1}\Phi$ is a slice sp-filtration of $\Spec A$.
\end{lemma}

\begin{proof}
Let $\Phi$ be a slice sp-filtration of $\Spec R$. If $\mathfrak{P}, \mathfrak{Q} \in f^{-1}(\Phi(n))\sm f^{-1}(\Phi(n+1))$ and $\mathfrak{P}\subseteq \mathfrak{Q}$, then $f(\mathfrak{P}), f(\mathfrak{Q}) \in \Phi(n)\sm \Phi(n+1)$ and $f(\mathfrak{P})\subseteq f(\mathfrak{Q})$. It follows that 
 $f(\mathfrak{P})= f(\mathfrak{Q})$ because $\Phi$ is a slice sp-filtration. Putting $\pp:= f(\mathfrak{P})= f(\mathfrak{Q})$, we have $\mathfrak{P}, \mathfrak{Q}\in f^{-1}\{\pp\}$ and $\mathfrak{P}\subseteq \mathfrak{Q}$, but then $\mathfrak{P}= \mathfrak{Q}$ as $\dim f^{-1}\{\pp\}= 0$.
\end{proof}

Let $R$ be a commutative noetherian ring. 
In the rest of the section, by an $R$-algebra, we will mean a commutative $R$-algebra for simplicity. A \emph{module-finite} $R$-algebra will mean a (commutative) $R$-algebra $A$ with structure map $\varphi: R\to A$ such that $\varphi$ is finite.
Assume $(R,\mm, k)$ is a local ring and let $A$ be a module-finite $R$-algebra with structure map $\varphi: R\to A$. Let $f:\Spec A\to \Spec R$ be the canonical map induced by $\varphi$. 
The set of maximal ideals of $A$ equals the fiber $f^{-1}(\mm)$ over $\mm$, and there is a canonical bijection between this and that of the artinian ring $A/\mm A=A \otimes_R k$; in particular, $A$ is semi-local. It then follows that there is a natural isomorphism 
\begin{equation*}
\RGamma_{\mm A}A\cong \bigoplus_{\mathfrak{P}\in \mSpec A}\RGamma_{\mathfrak{P}}A_\mathfrak{P}
\end{equation*}
in $\D(A)$; see \cref{comparison}\cref{Gamma-adjoint} and \cref{Gamma-Inj}.
Let $\boldsymbol{x}=x_1,\ldots,x_n$ be a system of generators of $\mm$.
Since $\RGamma_{\mm}R\cong \check{C}(\boldsymbol{x})$ in $\D(R)$,
regarding $-\LotimesR A$ as a functor $\D(R)\to \D(A)$, we obtain a natural isomorphism
\begin{equation*}
(\RGamma_{\mm}R)\LotimesR A\cong \check{C}(\boldsymbol{x})\LotimesR A=\check{C}(\boldsymbol{x})\otimes_R A
\end{equation*}
in $\D(A)$. 
The ideal $\mm A$ of $A$ is generated by $\varphi(\boldsymbol{x}):=\varphi(x_1),\ldots,\varphi(x_n)$, so 
\[\check{C}(\boldsymbol{x})\otimes_RA=\check{C}(\varphi(\boldsymbol{x}))\cong \RGamma_{\mm A}A\]
in $\D(A)$.
As a consequence, there is a canonical isomorphism
\begin{equation}\label{local-corresp}
(\RGamma_{\mm}R)\LotimesR A\cong\bigoplus_{\mathfrak{P}\in \mSpec A}\RGamma_{\mathfrak{P}}A_\mathfrak{P}
\end{equation}
in $\D(A)$.

We next remark that there is a canonical isomorphism $A\otimes_R\widehat{R}\cong\prod_{\mathfrak{P}\in \mSpec A}\widehat{A_\mathfrak{P}}$ of rings (\cite[Theorems 8.7 and 8.15]{Mat89}). Notice from this isomorphism that each $\widehat{A_\mathfrak{P}}$ has a canonical ring homomorphism from $\widehat{R}$, by which $\widehat{A_\mathfrak{P}}$ is a module-finite $\widehat{R}$-algebra.
Regarding $\RHom_R(A,-)$ and $\RHom_{\widehat{R}}(\widehat{A_\mathfrak{P}},-)$ as functors $\D(\widehat{R})\to \D(A)$, 
we have a canonical isomorphism
\begin{align}\label{local-corresp2}
\RHom_R(A,D_{\widehat{R}})\cong\prod_{\mathfrak{P}\in \mSpec A}\RHom_{\widehat{R}}(\widehat{A_\mathfrak{P}},D_{\widehat{R}})
\end{align}
in $\D(A)$ because 
\[\RHom_R(A,D_{\widehat{R}})=\RHom_R(A,\Hom_{\widehat{R}}(\widehat{R},D_{\widehat{R}}))\cong\RHom_{\widehat{R}}(A\otimes_{R}\widehat{R},D_{\widehat{R}}).\]
Moreover, we can regard $\RHom_{\widehat{R}}(\widehat{A_\mathfrak{P}},D_{\widehat{R}})$ as a dualizing  complex for $\widehat{A_\mathfrak{P}}$ since $\widehat{A_\mathfrak{P}}$ is a module-finite $\widehat{R}$-algebra; see \cref{dc-facts}\cref{image-dc}.

\begin{lemma}\label{finite-1}
Let $R$ be a commutative noetherian ring and $A$ a module-finite $R$-algebra.
Let $f:\Spec A\to \Spec R$ be the canonical map induced by the structure map $R\to A$ and let $\Phi$ be a non-degenerate sp-filtration of $\Spec R$.
Then there are natural isomorphisms 
$T_{\Phi}\LotimesR A\cong T_{f^{-1}\Phi}$ and $\RHom_R(A,C_{\Phi})\cong C_{f^{-1}\Phi}$
 in $\D(A)$.
\end{lemma}

\begin{proof}
Since $T_\Phi = \bigoplus_{\pp \in \Spec R} \Sigma^{\f_\Phi(\pp)}\RGamma_{\pp} R_{\pp}$ and $C_\Phi = \prod_{\pp \in \Spec R}\Sigma^{\height(\pp)-\f_\Phi(\pp)}D_{\widehat{R_\pp}}$, this lemma follows from \cref{local-corresp,local-corresp2}. Note that, given $\pp\in \Spec R$, we have $\f_{\Phi}(\pp)=\f_{f^{-1}\Phi}(\mathfrak{P})$ for every $\mathfrak{P}\in f^{-1}(\pp)$; see \cref{composite-filt} below.
\end{proof}

\begin{remark}\label{composite-filt}
Let $\Phi$ be a non-degenerate sp-filtration.
Let $\f_{\Phi}:\Spec R\to \ZZ$ and $\f_{f^{-1}\Phi}:\Spec A\to \ZZ$ be the functions associated to $\Phi$ and $f^{-1}\Phi$, respectively (\cref{order-preserv}).
For each $\mathfrak{P}\in \Spec A$, we have
\[\f_{f^{-1}\Phi}(\mathfrak{P})=\sup \{n\in \ZZ \mid \mathfrak{P} \in f^{-1}(\Phi(n))\}+1=\sup \{n\in \ZZ \mid f(\mathfrak{P}) \in \Phi(n)\}+1=\f_\Phi(f(\mathfrak{P})).
\]
Therefore the function $\f_{f^{-1}\Phi}:\Spec A\to \ZZ$ is just the composition of $f:\Spec A\to \Spec R$ and $\f_{\Phi}:\Spec R\to \ZZ$.
The same fact holds for any sp-filtration, replacing $\ZZ$ by $\ZZ\cup \{-\infty,\infty\}$.
\end{remark}

\begin{proposition}\label{finite-2}
Let $R$ be a commutative noetherian ring of finite Krull dimension and $A$ a module-finite $R$-algebra.
Let $f:\Spec A\to \Spec R$ be the canonical map induced by the structure map $R\to A$ and let $\Phi$ be a bounded sp-filtration of $\Spec R$.
For any cardinal $\varkappa$, there are canonical isomorphisms  in $\D(A)$:
\begin{align*}
\RHom_R(T_\Phi, T_\Phi^{(\varkappa)})\LotimesR A
&\cong \RHom_{A}(T_{f^{-1}\Phi}, T_{f^{-1}\Phi}^{(\varkappa)}),\\
\RHom_R(C_{\Phi}^{\varkappa},C_\Phi)\LotimesR A
&\cong \RHom_A(C_{f^{-1}\Phi}^{\varkappa},C_{f^{-1}\Phi}).
\end{align*}
\end{proposition}

\begin{proof}
By \cref{fin-proj-dim}, $T_\Phi$ is isomorphic in $\D(R)$ to a bounded complex $P$ of projective $R$-modules. 
We have standard isomorphisms
\[\Hom_R(P, P^{(\varkappa)})\otimes_R A\cong\Hom_R(P, P^{(\varkappa)}\otimes_R A)\cong \Hom_A(P\otimes_RA, P^{(\varkappa)}\otimes_RA)\]
in $\C(A)$; see, e.g., \cite[Proposition 2.1(vi)]{CH09} for the first isomorphism.
On the other hand, $C_{\Phi}$ is isomorphic in $\D(R)$ to a bounded complex $I$ of injective $R$-modules by \cref{bd-comp-inj}.
We have standard isomorphisms
\[\Hom_R(I^{\varkappa}, I)\otimes_R A\cong\Hom_R(\Hom_R(A,I^{\varkappa}), I))\cong \Hom_A(\Hom_R(A,I^{\varkappa}), \Hom_R(A,I))\]
in $\C(A)$; see, e.g., \cite[Proposition 2.1(ii)]{CH09} for the first isomorphism.
Notice that both $\Hom_R(P, P^{(\varkappa)})$ and $\Hom_R(I^{\varkappa}, I)$ are bounded complexes of flat $R$-modules, so 
\begin{align*}
\RHom_R(T_\Phi, T_\Phi^{(\varkappa)})\LotimesR A
&\cong \Hom_R(P, P^{(\varkappa)})\otimes_R A,\\
\RHom_R(C_{\Phi}^{\varkappa},C_\Phi)\LotimesR A
&\cong \Hom_R(I^{\varkappa}, I)\otimes_R A
\end{align*}
in $\D(A)$.
Moreover, \cref{finite-1} implies that
\begin{align*}
\Hom_A(P\otimes_RA, P^{(\varkappa)}\otimes_R A)
&\cong \RHom_A(T_{f^{-1}\Phi}, T_{f^{-1}\Phi}^{(\varkappa)}),\\
\Hom_A(\Hom_R(A,I^{\varkappa}), \Hom_R(A,I))
&\cong\RHom_A(C_{f^{-1}\Phi}^\varkappa, C_{f^{-1}\Phi})
\end{align*}
in $\D(A)$. This completes the proof.
\end{proof}

Recall that every Cohen--Macaulay ring admits a codimension function (e.g., the height function); see \cref{cd-func}. Moreover, if $R$ is a Cohen--Macaulay ring, then the polynomial ring $R[x_1,\ldots, x_n]$ is also a Cohen--Macaulay ring for every $n\geq 1$ (\cite[Theorem 17.7]{Mat89}).
Hence whenever $A$ is an algebra over a Cohen--Macaulay ring $R$ such that $A$ is finitely generated as an $R$-algebra, then $A$ admits a codimension function.

\begin{theorem}\label{finite-theorem}
Let $R$ be a Cohen--Macaulay ring of finite Krull dimension and $A$ a module-finite $R$-algebra.
Let $\Phi$ be any codimension filtration of $\Spec A$. Then  $T_{\Phi}$ is tilting and $C_{\Phi}$ is cotilting  in $\D(A)$.
Furthermore, both the $A$-modules $\Hom_{\D(A)}(T_{\Phi},T_{\Phi}^{(\varkappa)})$ and $\Hom_{\D(A)}(C_{\Phi}^\varkappa,C_{\Phi})$ are flat for any cardinal $\varkappa$.
\end{theorem}

\begin{proof}
In view of \cref{non-connected}, it suffices to show the claims for some codimension filtration of $\Spec A$.

Let $f:\Spec A\to \Spec R$ be the canonical map induced by the structure map $R\to A$.
Let $\Psi$ be a codimension filtration of $\Spec R$. 
By \cref{sp-filt-finite}, $\Phi':=f^{-1}\Psi$ is a slice sp-filtration of $\Spec A$.
Hence $T_{\Phi'}$ is silting and $C_{\Phi'}$ is cosilting in $\D(A)$ by \cref{slice-cosilt,slice-silting}. 
\cref{tilt-CM,cotilt-CM} yield the natural isomorphisms $\RHom_R(T_{\Psi},T_{\Psi}^{(\varkappa)})\cong \Hom_{\D(R)}(T_{\Psi},T_{\Psi}^{(\varkappa)})$ and $\RHom_R(C_{\Psi}^\varkappa,C_{\Psi})\cong \Hom_{\D(R)}(C_{\Psi}^\varkappa,C_{\Psi})$ in $\D(R)$ for any cardinal $\varkappa$. 
By \cref{flat-End,finite-2,finite-2}, we have the following isomorphisms in $\D(A)$:
\begin{align*}
&\RHom_{A}(T_{{\Phi'}},T_{{\Phi'}}^{(\varkappa)}) \cong \RHom_{R}(T_{\Psi},T_{\Psi}^{(\varkappa)}) \LotimesR A \cong \Hom_{\D(R)}(T_{\Psi},T_{\Psi}^{(\varkappa)})\otimes_RA,\\
&\RHom_{A}(C_{\Phi'}^\varkappa,C_{\Phi'})\cong \RHom_{R}(C_{\Psi}^\varkappa,C_{\Psi})\LotimesR A\cong \Hom_{\D(R)}(C_{\Psi}^\varkappa,C_{\Psi})\otimes_RA.
\end{align*}
Hence it follows that $T_{\Phi'}^{(\varkappa)}\in T_{\Phi'}^{\perp_{<0}}$ and $C_{\Phi'}^\varkappa\in \Perp{<0}C_{\Phi'}$. Thus $T_{\Phi'}$ is tilting and $C_{\Phi'}$ is cotilting in $\D(A)$.
Then \cref{necessity} implies that ${\Phi'}$ is a codimension filtration of $\Spec A$. In view of \cref{non-connected}, we can without loss of generality assume that $\Phi = \Phi'$. Moreover, we see from the above isomorphisms that $\RHom_{A}(T_{\Phi},T_{\Phi}^{(\varkappa)})\cong \Hom_{\D(A)}(T_{\Phi},T_{\Phi}^{(\varkappa)})$ and 
$\RHom_{A}(C_{\Phi}^\varkappa,C_{\Phi})\cong \Hom_{\D(A)}(C_{\Phi}^\varkappa,C_{\Phi})$ in $\D(A)$.
The right-hand sides of these isomorphisms are flat $A$-modules because  $\Hom_{\D(R)}(T_{\Psi},T_{\Psi}^{(\varkappa)})$ and $\Hom_{\D(R)}(C_{\Psi}^\varkappa,C_{\Psi})$ are flat $R$-modules by \cref{flat-End} again.
\end{proof}

The above proof shows that, under the assumption of the theorem, $f^{-1}\Psi$ is a codimension filtration of $\Spec A$ whenever $\Psi$ is a codimension filtration of $\Spec R$.
By the paragraph before the theorem, $R$ is also a homomorphic image of a Cohen--Macaulay ring of finite Krull dimension, and hence it is in fact possible to reduce the proof to the case when the structure map $R\to A$ is surjective. In this case, it is obvious that $f^{-1}\Psi$ is a codimension filtration by \cref{composite-filt}.

A commutative noetherian ring $R$ is said to be \emph{universally catenary} provided that any commutative algebra $A$ over $R$ is catenary whenever $A$ is finitely generated as an $R$-algebra.
When $R$ is local and $\pp\in \Spec R$, the ring $\widehat{R}\otimes_R\kappa(\pp)$ is called 
the \emph{formal fiber} of $R$ over $\pp$.
A formal fiber of $R$ means the ring $\widehat{R}\otimes_R \kappa(\pp)$ for some $\pp\in \Spec R$.
We now recall the following deep result due to Kawasaki.

\begin{theorem}[{\cite[Corollaries 1.2 and 1.4]{Kaw02} and \cite[Theorem 1.3]{Kaw08}}]\label{Kawasaki}
\leavevmode
\begin{enumerate}[label=(\arabic*), font=\normalfont]
\item\label{Kawasaki-dc} A commutative noetherian ring $R$ admits a classical dualizing complex if and only if it is a homomorphic image of a Gorenstein ring of finite Krull dimension.\item\label{Kawasaki-local} A commutative noetherian local ring $R$ is a homomorphic image of a Cohen--Macaulay local ring if and only if it is universally catenary and all the formal fibers of $R$ are Cohen--Macaulay. 
\item \label{Kawasaki-nonloc}
A commutative noetherian ring $R$ is a homomorphic image of a Cohen--Macaulay ring if and only if it admits a codimension function and satisfies the following conditions:
\begin{enumerate}[label=(\roman*), font=\normalfont]
\item \label{univ-cat} $R$ is universally catenary.
\item\label{formal-fiber} all the formal fibers of all the localizations of $R$ are Cohen--Macaulay.
\item\label{CM-locus} the Cohen--Macaulay locus of each finitely generated $R$-algebra is Zariski open.
\end{enumerate}
\end{enumerate}
\end{theorem}

\cref{finite-theorem,Kawasaki} make us have a naive question:

\begin{center}\emph{Can we characterize homomorphic images of Cohen--Macaulay rings by (co)tilting objects?}
\end{center}

If we replace ``Cohen--Macaulay rings'' by ``Gorenstein rings of finite Krull dimension'', this question is affirmatively solved by Kawasaki in the following sense: A classical dualizing complex for $R$ is a cotilting object in the bounded derived category $\D^{\bd}_{\fg}(R)$ (see \cref{dc-cotilt-rem} below), and the existence of a dualizing complex is equivalent to that $R$ is a homomorphic image of a Gorenstein ring of finite Krull dimension by \cref{Kawasaki}\cref{Kawasaki-dc}.

\begin{remark}\label{dc-cotilt-rem}
Assume that $R$ admits a classical dualizing complex $D$. Since the contravariant functor $\RHom_R(-,D): \D^{\bd}_{\fg}(R)\to \D^{\bd}_{\fg}(R)$ yields a duality, it transforms the standard t-structure in $\D^{\bd}_{\fg}(R)$ to a t-structure in $\D^{\bd}_{\fg}(R)$, which can be written as \[(\D^{\bd}_{\fg}(R)\cap \Perp{\leq 0}D, \D^{\bd}_{\fg}(R) \cap \Perp{> 0}D).\]
In \cite{ATJLS10}, this t-structure is called the \emph{Cohen--Macaulay t-structure} with respect to $D$.

Since $D\in \Perp{\neq 0} D$, the description of the Cohen--Macaulay t-structure implies that $D$ is a cotilting object in $\D^{\bd}_{\fg}(R)$ in the sense of \cite{PV18}. 
Further, if $D$ can be taken as an $R$-module and of finite injective dimension, it is a cotilting $(R,R)$-bimodule in the sense of \cite{Miy96}. 

Let $\cd$ be the codimension function associated to a dualizing complex $D$.
The Cohen--Macaulay t-structure with respect to $D$ extends to a compactly generated t-structure in $\D(R)$, and this is nothing but $(\cU_{\Phi_\cd}, \cV_{\Phi_\cd})$; see \cref{w-s-Cousin} and \cite[\S 6.4]{ATJLS10}. In general, $D$ is not a cotilting object in $\D(R)$, while \cref{slice-cosilt} shows that the t-structure $(\cU_{\Phi_\cd}, \cV_{\Phi_\cd})$ is induced by $C_{\Phi_\cd}$, which is cotilting in $\D(R)$ by \cref{dc-cotilt}.
However, \cref{slice-cosilt} more generally shows that $C_{\Phi_\cd}$ is silting and $(\cU_{\Phi_\cd}, \cV_{\Phi_\cd})=(\Perp{\leq 0}C_{\Phi_\cd}, \Perp{>0}C_{\Phi_\cd})$ for any commutative noetherian ring $R$ admitting a codimension function $\cd$ on $\Spec R$. This observation naturally motivates us to replace the existence of a dualizing complex by the cotilting property of $C_{\Phi_\cd}$ to determine whether $R$ is a homomorphic image of a Cohen--Macaulay ring or not.
\end{remark}

As a precise formulation of the above question, we suggest the following.

\begin{question}\label{question}
Let $R$ be a commutative noetherian ring with a codimension function $\cd$. 
Is $R$ a homomorphic image of a Cohen--Macaulay ring if and only if $T_{\Phi_{\cd}}$ is tilting (resp. $C_{\Phi_{\cd}}$ is cotilting) in $\D(R)$?
\end{question}

\begin{definition}\label{cm-heart}
Let $R$ be a commutative noetherian ring with a codimension function $\cd$. 
Motivated by \cref{dc-cotilt-rem}, we call the heart of the compactly generated t-structure $(\cU_{\Phi_\cd},\cV_{\Phi_\cd})$ the \emph{Cohen--Macaulay heart} of $R$. 
Although changing the codimension function results in a different compactly generated t-structure in $\D(R)$, its heart remains equivalent (\cref{non-connected}), so the Cohen--Macaulay heart is well-defined up to equivalence.
\end{definition}

If $\dim R<\infty$ and $\cd$ is a codimension function on $\Spec R$, then $C_{\Phi_\cd}$ is a bounded cosilting object by \cref{slice-cosilt,bd-comp-inj}. Hence we see from \cref{subsec-real-func,subsec-der-eq} and \cref{dc-cotilt-rem} that the cotilting part of \cref{question} with $\dim R<\infty$ is equivalent to the following:
\begin{question}\label{CM-equiv}
Let $R$ be a commutative noetherian ring of finite Krull dimension and assume that $R$ admits a codimension function. Let $\cH$ be the Cohen--Macaulay heart of $R$ and let $\D^\bd(\cH)\to \D^\bd(R)$ be a realization functor.
Is $R$ a homomorphic image of a Cohen--Macaulay ring if and only if the realization functor is a triangulated equivalence?
\end{question}

\begin{remark}
We may also replace the realization functor in \cref{CM-equiv} by the realization functor $\D(\cH)\to \D(R)$ (due to Virili) between the unbounded derived categories; see \cref{subsec-der-eq}.

\cref{finite-theorem} shows that the ``only if'' part of \cref{question} (and \cref{CM-equiv}) holds true when $R$ has finite Krull dimension. Moreover, if $R$ has Krull dimension at most one, then $R$ is always a homomorphic image of a Cohen--Macaulay ring of finite Krull dimension; see \cref{1-dim-CM} below. Therefore, in terms of \cref{1-dim,1-dim-cotilt}, the above questions holds true as far as $\dim R\leq 1$.
We will later show that the ``if'' part of the question holds true when $R$ is a 2-dimensional local ring; see also  \cref{non-CM-im} below.
\end{remark}

The next theorem is a direct consequence of \cite{Kaw08}. 
\begin{theorem}\label{1-dim-CM}
Let $R$ be a commutative noetherian ring of Krull dimension at most one. Then $R$ is a homomorphic image of a Cohen--Macaulay ring of finite Krull dimension.
\end{theorem}

\begin{proof}
If $\dim R=0$, there is nothing to show. So we may assume $\dim R= 1$. Note that $R$ has a codimension function (e.g., the height function), and \cite[(QU)]{Kaw08} holds for any finitely generated $R$-module.
By \cite[Theorem 1.4]{Kaw08}, \cref{univ-cat}--\cref{CM-locus} of \cref{Kawasaki}\cref{Kawasaki-nonloc} hold if and only if for any finitely generated $R$-module $M$, all the cohomologies of the Cousin complex $C(M)$ of $M$ (in the sense of Sharp \cite{Sha69}) are finitely generated, where 
\[C(M)=(0\to M\xrightarrow{d^{-1}_{C(M)}} \bigoplus_{\substack{\pp\in \Spec R\\ \height(\pp)=0}}M_\pp \xrightarrow{d^{0}_{C(M)}}  \Coker d^{-1}_{C(M)} \to 0)\]
and $d^{-1}_M$ is the morphism induced by the localization maps $M\to M_\pp$ and $d^{0}_{C(M)}$ is the canonical surjection.
Notice that $\Coker d^{-1}_{C(M)}$ is naturally isomorphic to $\bigoplus_{\substack{\mm\in \Spec R\\ \height(\mm)=1}}(\Coker d^{-1}_{C(M)})_{\mm}$ because $\Coker d^{-1}_{C(M)}\cong H^1\RGamma_{V}M\cong \bigoplus_{\substack{\mm\in \Spec R\\ \height(\mm)=1}} H^1\RGamma_{\mm}M_{\mm}$ for $V:=\{\mm\in \Spec R\mid \height(\mm)=1\}$; see \cref{Gamma-Inj}. This justifies the above description of $C(M)$; see \cite[Definition 5.1]{Kaw08}.
Then the possibly nonzero cohomology of $C(M)$ is only $\Ker d^{-1}_{C(M)}$, which is finitely generated.
Therefore, we can conclude by \cref{Kawasaki}\cref{Kawasaki-nonloc} that $R$ is a homomorphic image of a Cohen--Macaulay ring, where the Cohen--Macaulay ring is constructed so that it has finite Krull dimension; see \cite[p.~123]{Kaw02}, \cite[p.~2738]{Kaw08}, and \cite[Theorem 15.7]{Mat89}.
\end{proof}

In the rest of this section, we discuss what can be deduced from the condition that $T_{\Phi}$ is tilting or $C_{\Phi}$ is cotilting.

Suppose $R$ has finite Krull dimension. Recall that $R$ is said to be \emph{equidimensional} provided that $\dim R = \dim R/\pp$ for every minimal prime ideal $\pp$ of $R$.

\begin{lemma}\label{equidim}
Let $R$ be a commutative noetherian ring and $\pp \subseteq \qq$ be a chain in $\Spec R$.
Assume that the cohomology of $\RGamma_{\pp}(R_\pp \otimes_R  D_{\widehat{R_\qq}})$ is concentrated in some degree $n$. Then
$n=\dim R_\qq-\dim R_\qq /\pp R_\qq$ and 
the local ring $\widehat{R_\qq}/\pp \widehat{R_\qq}$ is equidimensional. 
\end{lemma}

\begin{proof}
By assumption, $H^i\RGamma_{\pp}(R_\pp \otimes_R  D_{\widehat{R_\qq}})=0$ for all $i\neq n$.
Regard $D_{\widehat{R_\qq}}$ as a bounded complex of injective $\widehat{R_\qq}$-modules that is minimal (see \cref{minimal}). 
By \cref{bd-comp-inj}, $D_{\widehat{R_\qq}}$ is also a complex of injective $R$-modules, so $\RGamma_{\pp}(R_\pp \otimes_R  D_{\widehat{R_\qq}})=\Gamma_{\pp}(R_\pp\otimes_R D_{\widehat{R_\qq}})=\Gamma_{\pp R_\qq}(R_\pp\otimes_R D_{\widehat{R_\qq}})$ in $\D(R)$.
In the rest of the proof, we write $R$ and $\pp$ for $R_\qq$ and $\pp R_\qq$, respectively.
Then $H^i\Gamma_{\pp}(R_\pp\otimes_R D_{\widehat{R}})=0$ for all $i\neq n$ in $\D(R)$.
We are going to show that $n=\dim R-\dim R /\pp R$ and $\widehat{R}/\pp \widehat{R}$ is equidimensional.

Let $f:\Spec \widehat{R}\to \Spec R$ be the canonical map induced by the completion map $R\to \widehat{R}$.
We remark that every minimal element $\mathfrak{P}$ in $V(\pp \widehat{R})$ belongs to $f^{-1}(\pp)$. 
Indeed, we have $\pp\subseteq \mathfrak{P}\cap R$ and the ring homomorphism $R\to \widehat{R}$  is flat, so the going-down theorem (\cite[Theorem 9.5]{Mat89}) gives an element $\mathfrak{Q}\in \Spec \widehat{R}$ with  $\mathfrak{Q}\subseteq \mathfrak{P}$ and $\pp=\mathfrak{Q}\cap R$, but then $\mathfrak{Q}\in V(\pp \widehat{R})$, so the minimality of $\mathfrak{P}$ yields $\mathfrak{Q}=\mathfrak{P}$. Consequently $\pp=\mathfrak{P}\cap R$, that is, $\mathfrak{P}\in f^{-1}(\pp)$.

Now, since the complex $D_{\widehat{R}}$ is minimal in $\C(\widehat{R})$, its $i$th component can be written as \[D_{\widehat{R}}^i\cong \bigoplus_{\substack{\mathfrak{P}\in \Spec \widehat{R}\\i=d-\dim \widehat{R}/\mathfrak{P}}} E_{\widehat{R}}(\widehat{R}/\mathfrak{P})\] 
for every $i\in \ZZ$, where $d:=\dim \widehat{R}$.
Noting that $\Gamma_{\pp}=\Gamma_{V(\pp \widehat{R})}$ as functors $\Mod \widehat{R}\to \Mod \widehat{R}$, we have
\begin{equation}\label{dc-fiber}
\Gamma_{\pp}(R_\pp \otimes_RD_{\widehat{R}}^i)\cong \bigoplus_{\substack{\mathfrak{P}\in f^{-1}(\pp)\\ i=d-\dim \widehat{R}/\mathfrak{P}}} E_{\widehat{R}}(\widehat{R}/\mathfrak{P}),
\end{equation}
where $f^{-1}(\pp)\subseteq V(\pp \widehat{R})$; see also \cref{Gamma-Inj}.
Since every minimal element $\mathfrak{P}$ in $V(\pp \widehat{R})$ belongs to $f^{-1}(\pp)$, the non-trivial components of $\Gamma_{\pp}(R_\pp \otimes_RD_{\widehat{R}})$ appear precisely from degree $d-\dim \widehat{R}/\pp \widehat{R}$ onwards. 
Notice from \cref{dc-fiber} that the complex $\Gamma_{\pp}(R_\pp \otimes_RD_{\widehat{R}})$ is also minimal in $\C(\widehat{R})$, so we have
\begin{equation}\label{inf-dc-fiber}
\inf \Gamma_{\pp}(R_\pp \otimes_RD_{\widehat{R}})=d-\dim \widehat{R}/\pp \widehat{R}=\dim R-\dim R/\pp R.
\end{equation}
Since $H^i\Gamma_{\pp}(R_\pp\otimes_R D_{\widehat{R}})$ in $\D(R)$ for all $i\neq n$, it follows that $n=\dim R-\dim R/\pp R$.

To show that $\widehat{R}/\pp \widehat{R}$ is equidimensional, suppose that there exists a minimal element $\mathfrak{P}$ in $V(\pp \widehat{R})$ such that $\dim(\widehat{R}/\mathfrak{P}) < \dim(\widehat{R}/\pp \widehat{R})$.
It follows from the second paragraph of the proof that $\widehat{R}_\mathfrak{P}\otimes_RR_\pp\cong \widehat{R}_\mathfrak{P}$. Then we have natural isomorphisms in $\C(\widehat{R})$: 
\[\widehat{R}_\mathfrak{P} \otimes_{\widehat{R}}\Gamma_{\pp \widehat{R}}(R_\pp \otimes_R D_{\widehat{R}}) \cong \widehat{R}_\mathfrak{P} \otimes_{\widehat{R}} \Gamma_{\pp \widehat{R}}D_{\widehat{R}} \cong \Sigma^{\dim \widehat{R}/\mathfrak{P}-\dim \widehat{R}}E_{\widehat{R}}(\widehat{R}/\mathfrak{P}),\] 
where these isomorphisms follow from \cref{Gamma-Inj} and minimality of $D_{\widehat{R}}$.
The above isomorphisms show that  
$\widehat{R}_\mathfrak{P} \otimes_{\widehat{R}}\Gamma_{\pp \widehat{R}}(D_{\widehat{R}}  \otimes_R R_\pp)$ has nontrivial cohomology in degree $m:=\dim \widehat{R}  - \dim \widehat{R}/\mathfrak{P}$, but $m > \dim R  - \dim R/\pp=n$. This is a contradiction because 
$H^m(\widehat{R}_\mathfrak{P}\otimes_{\widehat{R}}\Gamma_{\pp}(R_\pp\otimes_R D_{\widehat{R}}))\cong \widehat{R}_\mathfrak{P}\otimes_{\widehat{R}}H^m\Gamma_{\pp}(R_\pp\otimes_R D_{\widehat{R}})=0$ in $\D(R)$.
\end{proof}

\begin{proposition}\label{univ-catenary}
	Let $R$ be a commutative noetherian ring and let $\Phi$ be a codimension filtration of $\Spec R$. 
Assume $T_\Phi$ is tilting or $C_\Phi$ is cotilting in $\D(R)$. Then $R$ is universally catenary.
\end{proposition}

\begin{proof}
It suffices to show that $R_\qq$ is universally catenary for each $\qq\in \Spec R$; see, e.g., \cite[\href{https://stacks.math.columbia.edu/tag/0AUN}{Tag 0AUN}]{stacks-project}. 
The local ring $R_\qq$ is universally catenary if and only if $\widehat{R_\qq}/\pp \widehat{R_\qq}$ is equidimensional for each $\pp\in \Spec R$ with $\pp\subseteq \qq$ (\cite[Theorem 31.7]{Mat89}).
For such a chain $\pp\subseteq \qq$, there is an integer $n$ such that $H^i\RGamma_{\pp}(R_\pp \otimes_R  D_{\widehat{R_\qq}})=0$ for all $i\neq 0$ by assumption and \cref{vanishing}. Then \cref{equidim} implies that $\widehat{R_\qq}/\pp \widehat{R_\qq}$ is equidimensional, as desired.
\end{proof}

By a \emph{stalk complex}, we mean a complex $X$  concentrated in some degree $n$, that is, $X^i=0$ whenever $i\neq n$.

\begin{proposition}\label{stalk}
Let $R$ be a commutative noetherian ring and $\pp \subseteq \qq$ be a chain in $\Spec R$.
The following conditions are equivalent.
\begin{enumerate}[label=(\arabic*), font=\normalfont]
\item \label{stalk-lc} $\RHom_R(\RGamma_{\pp}R_\pp,\RGamma_{\qq}R_\qq)$ is isomorphic in $\D(R)$ to a stalk complex of a flat $R$-module.
\item \label{stalk-dc-lc} $\RGamma_{\pp}(R_\pp \otimes_R  D_{\widehat{R_\qq}})$ is isomorphic in $\D(R)$ to a stalk complex of an injective $R$-module.
\item \label{stalk-dc} $\RHom_R(D_{\widehat{R_\qq}},D_{\widehat{R_\pp}})$ is isomorphic in $\D(R)$ to a stalk complex of a flat $R$-module.
\end{enumerate}
Under each condition, the formal fiber $\kappa(\pp) \otimes_{R_\qq}\widehat{R_\qq}$ is Cohen--Macaulay.
\end{proposition}

\begin{proof}
We may assume that $R$ is a local ring and $\qq$ is the maximal ideal $\mm$.
We remark that there are natural isomorphisms
\begin{align}
\RHom_R(\RGamma_{\pp}R_\pp,\RGamma_{\mm}R)&\cong \RHom_{\widehat{R}}(\Sigma^{\height(\mm)}\RGamma_\pp (R_\pp\otimes_RD_{\widehat{R}}), E_{\widehat{R}}(k))\label{adjoint-Prop-1}\\
\RHom_R(D_{\widehat{R}},D_{\widehat{R_\pp}}) 
&\cong \RHom_{R_\pp}(\Sigma^{\height(\pp)}\RGamma_\pp (R_\pp\otimes_RD_{\widehat{R}}),E_{R_\pp}(\kappa(\pp))).\label{adjoint-Prop-2}
\end{align}
in $\D(R)$; see the proof of \cref{vanishing}.

Suppose that the left-hand side of \cref{adjoint-Prop-1} is isomorphic in $\D(R)$ to a stalk complex of a flat $R$-module.
Since $E_{\widehat{R}}(k)$ is an injective cogenerator in $\Mod \widehat{R}$,  $\RGamma_\pp (R_\pp\otimes_RD_{\widehat{R}})$ is isomorphic in $\D(\widehat{R})$ to a stalk complex, so there are an $\widehat{R}$-module $M$ and an integer $n$ such that $\Sigma^{\height(\mm)}\RGamma_\pp (R_\pp\otimes_RD_{\widehat{R}})\cong \Sigma^n M$ in $\D(\widehat{R})$. 
It follows that $\Hom_{\widehat{R}}(M, E_{\widehat{R}}(k))$ is a stalk complex of a flat $R$-module. In other words, $M$ is flat as an $R$-module, and this is equivalent to that $M$ is injective as an $R$-module; see \cref{complete-cogen} below. Thus the implication \cref{stalk-lc}$\Rightarrow$\cref{stalk-dc-lc} follows. 

To see the converse implication, suppose that $\RGamma_\pp (R_\pp\otimes_RD_{\widehat{R}})$ is isomorphic in $\D(R)$ to a complex of an injective $R$-module. Then 
there are an $\widehat{R}$-module $M$ and an integer $n$ such that $\Sigma^{\height(\mm)}\RGamma_\pp (R_\pp\otimes_RD_{\widehat{R}})\cong \Sigma^n M$ in $\D(\widehat{R})$.
Since $\Sigma^n M$ is isomorphic in $\D(R)$ to $\RGamma_\pp (R_\pp\otimes_RD_{\widehat{R}})$, it follows that $M$ is injective as an $R$-module. Then the left-hand side of \cref{adjoint-Prop-1} is isomorphic in $\D(R)$ to $\Hom_{\widehat{R}}(\Sigma^{n}M,E_{\widehat{R}}(k))$, which is a stalk complex of a flat $R$-module by \cref{complete-cogen}. Hence we have \cref{stalk-dc-lc}$\Rightarrow$\cref{stalk-lc}.

Next, suppose that $\RGamma_\pp (R_\pp\otimes_RD_{\widehat{R}})$ is isomorphic in $\D(R)$ a stalk complex of an injective $R$-module. Note that the injective $R$-module is also an injective $R_\pp$-module because we may regard $\RGamma_\pp (R_\pp\otimes_RD_{\widehat{R}})$ as a complex of $R_\pp$-modules in $\D(R)$.
Letting $M$ be the injective $R_\pp$-module, we have $\Sigma^{\height(\pp)}\RGamma_\pp (R_\pp\otimes_RD_{\widehat{R}})\cong \Sigma^n M$ in $\D(R_\pp)$ for some integer $n$.
Hence the right-hand side of \cref{adjoint-Prop-2} is isomorphic in $\D(R)$ to $\Hom_{R_\pp}(\Sigma^{n}M,E_{R_\pp}(\kappa(\pp)))$, which is a stalk complex of a flat $R$-module.
Thus the implication \cref{stalk-dc-lc}$\Rightarrow$\cref{stalk-dc} follows.

To see the converse implication, suppose that the left-hand side of \cref{adjoint-Prop-2} is isomorphic in $\D(R)$ to a stalk complex of a flat $R$-module. Note that the flat $R$-module is also a flat $R_\pp$-module. Since $E_{R_\pp}(\kappa(\pp))$ is an injective cogenerator in $\Mod R_\pp$, it follows that $\Sigma^{\height(\pp)}\RGamma_\pp (R_\pp\otimes_RD_{\widehat{R}})$ is isomorphic in $\D(R_\pp)$ to a stalk complex of an injective $R_\pp$-module. In other words, $\RGamma_\pp (R_\pp\otimes_RD_{\widehat{R}})$ is isomorphic in $\D(R)$ to a stalk complex of an injective $R$-module.
Thus we have \cref{stalk-dc}$\Rightarrow$\cref{stalk-dc-lc}.

Finally, we show that $\kappa(\pp)\otimes_R \widehat{R}$ is Cohen--Macaulay, assuming that $\RGamma_{\pp}(R_\pp \otimes_R  D_{\widehat{R}})$ is isomorphic in $\D(R)$ to a stalk complex of an injective $R$-module.
We remark that there are natural isomorphisms in $\D(R)$:
\begin{align*}
\RHom_R(R/\pp, \RGamma_{\pp}(R_\pp \otimes_R  D_{\widehat{R}}))
&\cong \RHom_R(R/\pp, R_\pp \otimes_R  D_{\widehat{R}})\\
&\cong \Hom_{R_\pp}(\kappa(\pp),R_\pp \otimes_R D_{\widehat{R}})\\
&\cong \Hom_{R_\pp}(\kappa(\pp), \Hom_{R_\pp\otimes_R \widehat{R}}(R_\pp\otimes_R \widehat{R},R_\pp \otimes_R D_{\widehat{R}}))\\
&\cong \Hom_{R_\pp\otimes_R \widehat{R}}(\kappa(\pp)\otimes_R \widehat{R}, R_\pp \otimes_R D_{\widehat{R}}).
\end{align*}
The last complex is a dualizing complex for $\kappa(\pp)\otimes_R \widehat{R}$ because $\kappa(\pp)\otimes_R \widehat{R}$ is a homomorphic image of $R_\pp\otimes_R \widehat{R}$ and $R_\pp \otimes_R D_{\widehat{R}}$ is a dualizing complex for $R_\pp\otimes_R \widehat{R}$.
By assumption, the cohomology of $\RHom_R(R/\pp, \RGamma_{\mm}(R_\pp \otimes_R  D_{\widehat{R}}))$ is concentrated in some degree, and the same holds for the dualizing complex for $\kappa(\pp)\otimes_R \widehat{R}$ by the isomorphisms above. This implies that $\kappa(\pp)\otimes_R \widehat{R}$ is a Cohen--Macaulay ring by local duality \cref{local-dual-1} applied to the local ring $(\kappa(\pp)\otimes_R\widehat{R})_{\mathfrak{P}}$ for each $\mathfrak{P}\in \Spec (\kappa(\pp)\otimes_R\widehat{R})$.
\end{proof}

\begin{remark}\label{complete-cogen}
Let $(R,\mm,k)$ be a commutative noetherian local ring and $M$ be an $\widehat{R}$-module.
For any finitely generated $R$-module $N$, we have standard isomorphisms
\begin{align*}\Tor_i^R(N,\Hom_{\widehat{R}}(M,E_{\widehat{R}}(k)))
&\cong \Tor_i^{\widehat{R}}(N\otimes_{R}\widehat{R},\Hom_{\widehat{R}}(M,E_{\widehat{R}}(k)))\\
&\cong \Hom_{\widehat{R}}(\Ext^i_{\widehat{R}}(N\otimes_{R}\widehat{R},M),E_{\widehat{R}}(k))\\
&\cong \Hom_{\widehat{R}}(\Ext^i_{R}(N,\Hom_{\widehat{R}}(\widehat{R},M)),E_{\widehat{R}}(k))\\
&\cong \Hom_{\widehat{R}}(\Ext^i_{R}(N,M),E_{\widehat{R}}(k))
\end{align*}
for all $i\geq 0$, where the second isomorphism follows from \cite[Theorem 3.2.13]{EJ11} and the first and third holds since $\widehat{R}$ is flat over $R$.
Since $E_{\widehat{R}}(k)$ is an injective cogenerator in $\Mod \widehat{R}$, it follows from the above isomorphisms that the $\widehat{R}$-module $M$ is injective over $R$ if and only if $\Hom_{\widehat{R}}(M,E_{\widehat{R}}(k))$ is flat over $R$.
\end{remark}

\begin{corollary}\label{flat-end-fibres}
	Let $R$ be a commutative noetherian ring and let $\Phi$ be a codimension filtration of $\Spec R$. Assume that at least one of the following conditions is true:
	\begin{enumerate}[label=(\arabic*), font=\normalfont]
		\item\label{cor-tilt-ass} $T_\Phi$ is tilting in $\D(R)$ and $\End_{\D(R)}(T_\Phi)$ is a flat $R$-module,
		\item\label{cor-cotilt-ass} $C_\Phi$ is cotilting in $\D(R)$ and $\End_{\D(R)}(C_\Phi)$ is a flat $R$-module.
	\end{enumerate}
	Then all the formal fibers of all the localizations of $R$ are Cohen--Macaulay.
\end{corollary}
\begin{proof}
	If \cref{cor-tilt-ass} holds then $\RHom_R(T_\Phi,T_\Phi)$ is isomorphic in $\D(R)$ to the flat $R$-module $\End_{\D(R)}(T_\Phi)$. Furthermore, for any inclusion $\pp \subseteq \qq$ we have that the object $\RHom_R(\RGamma_{\pp}R_\pp,\RGamma_{\qq}R_\qq)$ is, up to a suitable shift, a direct summand of $\RHom_R(T_\Phi,T_\Phi)$. It follows that $\RHom_R(\RGamma_{\pp}R_\pp,\RGamma_{\qq}R_\qq)$ is isomorphic to a stalk complex of a flat $R$-module in $\D(R)$. Then \cref{stalk} yields that the formal fiber $\kappa(\pp) \otimes_{R_\qq} \widehat{R_\qq}$ is Cohen--Macaulay.

	The proof is completely analogous when the assumption \cref{cor-cotilt-ass} is satisfied instead.
\end{proof}
Now we are ready to prove \cref{intro-flat-End} together with its cotilting counterpart.
	\begin{theorem}\label{flat-End-thm}
	Let $R$ be a commutative noetherian local ring with a codimension filtration $\Phi$. The following conditions are equivalent:
	\begin{enumerate}[label=(\arabic*), font=\normalfont]
		\item\label{thm-flat-End-tilt} $T_\Phi$ is tilting and $\End_{\D(R)}(T_\Phi)$ is a flat $R$-module.
		\item \label{thm-flat-End-cotilt} $C_\Phi$ is cotilting and $\End_{\D(R)}(C_\Phi)$ is a flat $R$-module.
		\item  \label{thm-flat-End-hom} $R$ is a homomorphic image of a Cohen--Macaulay local ring.
	\end{enumerate}
\end{theorem}
\begin{proof}
	\cref{finite-theorem} yields both the implications \cref{thm-flat-End-hom} $\Rightarrow$ \cref{thm-flat-End-tilt}, \cref{thm-flat-End-cotilt}. 
	
	Conversely, if we assume either one of the assumptions \cref{thm-flat-End-tilt} or \cref{thm-flat-End-cotilt} then $R$ is universally catenary by \cref{univ-catenary} and \cref{flat-end-fibres} shows that the formal fibres of $R$ are Cohen--Macaulay. Since $R$ is local, the condition \cref{thm-flat-End-hom} follows by \cref{Kawasaki}\cref{Kawasaki-local}.
\end{proof}
\begin{question}\label{Q-flatend}
	Let $R$ be a commutative noetherian ring.
	\begin{enumerate}[label=(\arabic*), font=\normalfont]
		\item\label{Q-flatend-tilt} Let $T$ be a tilting object (resp. $C$ be a cotilting object) in $\D(R)$. Is $\End_{\D(R)}(T)$ (resp. $\End_{\D(R)}(C)$) flat as an $R$-module?
		\item\label{Q-flatend-codim} Is the answer to \cref{Q-flatend-tilt} affirmative at least in the case $T=T_\Phi$ (resp. $C=C_\Phi$) for a codimension filtration $\Phi$ of $\Spec R$?
	\end{enumerate}
\end{question}

\begin{remark}
By \cref{flat-End-thm}, an affirmative answer to \cref{Q-flatend}\cref{Q-flatend-codim} implies an affirmative answer to \cref{question} for all local rings.

By \cref{finite-theorem}, \cref{Q-flatend}\cref{Q-flatend-codim} holds true whenever $R$ is a module-finite algebra over a Cohen--Macaulay ring of finite Krull dimension.
Moreover, \cref{End-tilt,End-cotilt} shows the validity of \cref{Q-flatend}\cref{Q-flatend-codim} when  $R$ is a commutative noetherian ring of Krull dimension at most one; \cref{Q-flatend}\cref{Q-flatend-codim} is also verified by \cref{finite-theorem,1-dim-CM} in this case.
We will show that  \cref{Q-flatend}\cref{Q-flatend-codim} holds true when $R$ is a 2-dimensional ring (\cref{2-dim-tilt-flat-end}).
\end{remark}

\begin{remark}\label{non-CM-im}
There is a 2-dimensional commutative noetherian local domain such that its generic formal fiber is not Cohen--Macaulay (\cite[Proposition 4.5]{HRW01}). 
There is also a 2-dimensional commutative noetherian local domain which is  catenary (by definition) but not universally catenary and whose formal fibers are (geometrically) regular (\cite[Example 2]{Nag75}); see \cite[Example 2.6]{Nis12}.
These rings admit codimension functions but are not homomorphic images of Cohen--Macaulay rings by \cref{Kawasaki}\cref{Kawasaki-local}.
\end{remark}

The next theorem affirmatively answers \cref{Q-flatend}\cref{Q-flatend-codim} in the case of a ring of Krull dimension two.

\begin{theorem}\label{2-dim-tilt-flat-end}
Let $R$ be a 2-dimensional commutative noetherian ring with a codimension filtration $\Phi$. Then:
\begin{enumerate}[label=(\arabic*), font=\normalfont]
	\item\label{2-dim-tilt-tilt} If $T_\Phi$ is tilting, then $\End_{\D(R)}(T_\Phi)$ is flat as an $R$-module.
	\item\label{2-dim-tilt-cotilt} If $C_\Phi$ is cotilting, then $\End_{\D(R)}(C_\Phi)$ is flat as an $R$-module.
\end{enumerate}
\end{theorem}

To prove this theorem, we make a remark.
\begin{remark}\label{generic-fiber}
Let $R$ be a commutative noetherian ring of finite Krull dimension and let $\pp$ be a minimal prime ideal of $R$. Assume that $X$ is a bounded complex of flat $R_\pp$-modules and $Y$ a bounded complex of flat $R$-modules. If $\RHom_{R}(X, Y)$ is concentrated in some degree, then $\RHom_{R}(X, Y)$ is isomorphic in ($\D(R)$) to a flat $R$-module.
Indeed, $X$ and $Y$ are isomorphic to bounded complexes of projective $R$-modules by \cite[Part II, Corollary 3.2.7]{RG71}. 
Thus $\RHom_{R}(X,Y)$ is isomorphic to a bounded complex of flat $R$-modules. 
By assumption, we have  $\RHom_{R}(X,Y)\cong \Sigma^n M$ for some $R_\pp$-module $M$ and $n\in \ZZ$.
Then $\Sigma^n M\cong F$, and thus $M$ has finite flat dimension over $R_\pp$. Since $\pp$ is minimal, $R_\pp$ is artinian,  so that every flat $R_\pp$-module is projective (\cite[Theorems 23.20 and 24.25]{Lam91}). Consequently $M$ has finite projective dimension over $R_\pp$, which is bounded by $\dim R_\pp=0$ (\cite[Theorem 3.2.6]{RG71}). Therefore $M$ is a projective $R_\pp$-module. In particular, $M$ is a flat $R$-module, as desired.

Similarly, if $X$ is a bounded complex of injective $R$-modules, $Y$ a bounded complex of injective $R_\pp$-modules, and $\RHom_{R}(X, Y)$ is concentrated in some degree, then $\RHom_{R}(X, Y)$ is isomorphic to an $R_\pp$-module of finite projective dimension over $R_\pp$, so the $R_\pp$-module is projective. Hence 
$\RHom_{R}(X, Y)$ is isomorphic in $\D(R)$ to a flat $R$-module.
\end{remark}

\begin{proof}
Let $\cd$ be the codimension function on $\Spec R$ such that $\Phi = \Phi_{\cd}$.
In view of \cref{non-connected}, we may assume that $\Spec R$ is connected.
Moreover, since $\dim R = 2$, we may assume that $\cd$ takes values in the interval $\{0,1,2\}$. 

\cref{2-dim-tilt-tilt}: Suppose that $T_\Phi$ is tilting.  
Let us use the notation of \cref{End-tilt}: $W_n = \{\pp \in \Spec R \mid \cd(\pp) = n\}$ and $T(n) = \bigoplus_{\pp \in W_n}\Sigma^{n}\RGamma_{\pp}R_\pp$. Then $T_\Phi = T(0) \oplus T(1) \oplus T(2)$.
By the proof of \cref{X(n)-X(n+1)}, we have 
$\Hom_{\D(R)}(T(i), T(j))=0$ if $i>j$. Hence the $R$-module $\End_{\D(R)}(T_\Phi)$ is the direct sum of 
$\Hom_{\D(R)}(T(0), T(0)\oplus T(1) \oplus T(2))$, $\Hom_{\D(R)}(T(1), T(1) \oplus T(2))$, and $\Hom_{\D(R)}(T(2), T(2))$. 
By \cref{X(n)-X(n+1),End-tilt}, it suffices to show that $\Hom_{\D(R)}(T(0), T(2))$ is a flat $R$-module.
Note that each $\pp \in W_0$ is necessarily a minimal prime of $R$ and 
\[\Hom_{\D(R)}(T(0), T(2))\cong \bigoplus_{\pp\in W_0}\Hom_{\D(R)}(\RGamma_\pp R_\pp, T(2)).\]
Then $\Hom_{\D(R)}(T(0), T(2))$ is flat as an $R$-module by \cref{Cech-comp,generic-fiber}.

\cref{2-dim-tilt-cotilt}: This follows from a parallel argument to \cref{2-dim-tilt-tilt}. 
Use \cref{bd-comp-inj,vanishing,generic-fiber,End-cotilt}.
\end{proof}

We conclude the paper by the following result, which affirmatively answers \cref{question} (and \cref{CM-equiv}) in the case of a local ring of Krull dimension two.
\begin{corollary}\label{2-dim-tilt}
	Let $(R,\mm,k)$ be a 2-dimensional commutative noetherian local ring with a codimension filtration $\Phi$. Then the following conditions are equivalent:
\begin{enumerate}[label=(\arabic*), font=\normalfont]
\item \label{2-dim-tilting} $T_\Phi$ is a tilting object in $\D(R)$.
\item \label{2-dim-cotilting} $C_\Phi$ is a cotilting object in $\D(R)$.
\item \label{2-dim-CM} $R$ is a homomorphic image of a Cohen--Macaulay local ring.
\end{enumerate}
\end{corollary}
\begin{proof}
	By \cref{finite-theorem}, the condition \cref{2-dim-CM} implies \cref{2-dim-tilting,2-dim-cotilting}. Conversely, assuming either \cref{2-dim-tilting} or \cref{2-dim-cotilting}, we obtain \cref{2-dim-CM} from Theorems~\ref{Kawasaki}\ref{Kawasaki-local}~and~\ref{2-dim-tilt-flat-end}, \cref{univ-catenary}, \cref{flat-end-fibres}.
\end{proof}


\bibliographystyle{amsalpha}
\bibliography{bibitems}

\end{document}